\documentclass[11pt]{ucr}

\usepackage{epsfig}
\usepackage{amsfonts,amsmath,amssymb,amsthm,amscd,mathrsfs,xspace,multicol}
\usepackage{undertilde}
\usepackage{pb-diagram}
\usepackage{latexsym}
\usepackage[all]{xy}
\usepackage[breaklinks=true]{hyperref}
\usepackage{graphicx}
\usepackage{tikz}
\usetikzlibrary{matrix,arrows}
\usetikzlibrary{decorations.pathmorphing}
\usetikzlibrary{decorations.pathreplacing}
\usetikzlibrary{decorations.shapes}
\usetikzlibrary{decorations.markings}
\usetikzlibrary{calc}

\tikzset{l/.style={font=\fontsize{8}{8}\selectfont}}
\definecolor{salmon}{rgb}{1,0.47,0.425}

\input{epsf}

\newcommand{\Vect}{{\rm Vect}}

\newcommand{\FinVect}{{\rm FinVect}}
\newcommand{\Span}{{\rm Span}}
\newcommand{\SES}{{\rm SES}}
\newcommand{\Rep}{{\rm Rep}}
\newcommand{\Gpd}{{\rm Gpd}}

\newcommand{\Ext}{{\rm Ext}}
\newcommand{\Hom}{{\rm Hom}}

\newcommand{\im}{{\rm im}}
\renewcommand{\hom}{{\rm hom}}

\renewcommand{\to}{\rightarrow}
\newcommand{\tensor}{\otimes}
\newcommand{\ten}{\otimes }
\newcommand{\maps}{\colon}

\newcommand{\id}{{\rm id}}

\newcommand{\iso}{\cong}

\newcommand{\g}{{\mathfrak g}}

\newcommand{\R}{{\mathbb R}}
\newcommand{\C}{{\mathbb C}}
\newcommand{\CC}{{\mathcal C}}
\newcommand{\F}{{\mathbb F}}
\newcommand{\N}{\ensuremath{\mathbb{N}}\xspace}
\newcommand{\Q}{{\mathbb Q}}
\newcommand{\Z}{\ensuremath{\mathbb{Z}}\xspace}

\newcommand{\Aut}{{\rm Aut}}
\newcommand{\EXT}{{\rm EXT}}

\newcommand{\Hall}{\mathcal{H}}
\newcommand{\aut}{{\rm aut}}

\usepackage{amsfonts,amsmath,amssymb,amsthm,amscd,mathrsfs,xspace,multicol}
\usepackage{undertilde}
\usepackage{latexsym}
\usepackage[all]{xy}
\usepackage[breaklinks=true]{hyperref}
\usepackage{graphicx}

\DeclareMathOperator{\Mor}{\ensuremath{\mathrm{Mor}}\xspace}

\newcommand{\comment}[1]{}
\renewcommand{\u}[1]{\underline{#1}}

\newtheorem{theorem}{Theorem}
\newtheorem{definition}[theorem]{Definition}
\newtheorem{defn}[theorem]{Definition}
\newtheorem{lemma}[theorem]{Lemma}

\newtheorem{proposition}[theorem]{Proposition}
\newtheorem{example}[theorem]{Example}
\newtheorem{conjecture}[theorem]{Conjecture}

\newtheorem*{theorem*}{Theorem}
\newtheorem*{definition*}{Definition}
\newtheorem*{lemma*}{Lemma}
\newtheorem*{corollary*}{Corollary}
\newtheorem*{proposition*}{Proposition}
\newtheorem*{example*}{Example}
\newtheorem*{conjecture*}{Conjecture}
\newtheorem*{remark*}{Remark}
\newtheorem*{notation*}{Notation}
\newtheorem*{convention*}{Convention}

\newdir{ >}{{}*!/-7pt/@{>}}
\newdir{> >}{*!/0pt/@{>}*!/-5pt/@{>}}

\begin{document}
\sloppy
\begin{frontmatter}
\author{Christopher D. Walker}
\title{A Categorification of Hall Algebras}
\campus{Riverside}
\degree{Doctor of Philosophy}
\field{Mathematics}
\degreeyear{2011}
\degreemonth{June}
\chair{Dr. John Baez}
\othermembers{Dr. Vijayanthi Chari\\Dr. Julia Bergner}
\numberofmembers{3}
\maketitle
\copyrightpage{}
\approvalpage{}
\begin{dedication}
\begin{center}
To\\ 
Nela,\\
Mikaela, \\
Taylor, \\
and Avery
\end{center}
\end{dedication}
\begin{abstract}
In recent years, there has been great interest in the study of categorification, specifically as it applies to the theory of quantum groups. In this thesis, we would like to provide a new approach to this problem by looking at Hall algebras. It is know, due to Ringel, that a Hall algebra is isomorphic to a certain quantum group. It is our goal to describe a categorification of Hall algebras as a way of doing so for their related quantum groups. To do this, we will take the following steps. First, we describe a new perspective on the structure theory of Hall algebras. This view solves, in a unique way, the classic problem of the multiplication and comultiplication not being compatible. Our solution is to switch to a different underlying category $\Vect^K$ of vector spaces graded by a group $K$ called the Grothendieck group. We equip this category with a nontrivial braiding which depends on the $K$-grading. With this braiding and a given antipode, we find that the Hall algebra does become a Hopf algebra object in $\Vect^K$. Second, we will describe a categorification process, call `groupoidification', which replaces vector spaces with groupoids and linear operators with `spans' of groupoids. We will use this process to construct a braided monoidal bicategory which categorifies $\Vect^K$ via the groupoidification program. Specifically, graded vector spaces will be replaced with groupoids `over' a fixed groupoid related to the Grothendieck group $K$. The braiding structure will come from an interesting groupoid $\EXT(M,N)$ which will behave like the Euler characteristic for the Grothendieck group $K$. We will finish with a description of our plan to, in future work, apply the same concept to the structure maps of the Hall algebra, which will eventually give us a Hopf $2$-algebra object in our braided monoidal bicategory. 
\end{abstract}

\newpage
\tableofcontents
\end{frontmatter}
\chapter{Introduction}
In this thesis we will categorify `half' of the quantum group associated to a simply-laced Dynkin diagram, meaning $U^+_q(\g)$ where $\g$ is the Lie algebra corresponding to that Dynkin diagram. It is known that $U^+_q(\g)$ is isomorphic to the `Hall algebra' associated to any quiver Q formed by drawing arrows on the edges of the Dynkin diagram. Here we categorify the Hall algebra using a method called `groupoidification'. The stardard Hall algebra construction itself can be thought of as a decategorification. Specifically, the Hall algebra is a decategoried version of Rep(Q), the category of representations of the quiver $Q$. However, the challenge is showing that $\Rep(Q)$ is a kind of categorified Hopf algebra in some braided monoidal bicategory. What we will do is construct this braided monoidal bicategory, and describe why this is the right setting to categorify the theory of Hall algebras.  

The plan of the thesis is as follows. In Chapter \ref{HA}, we start by describing the basic theory of Hall algebras in a new context first described by the author in \cite{HopfObject}. Hall algebras have been a popular topic in recent years because of their connection to quantum groups. As stated previously, the Hall algebra of a quiver is isomorphic to 'half' the quantum group associated to the quiver's underlying Dynkin diagram. By `half' we mean the positive part of the standard triangle decomposition, $U_q^+(\g)$. This construction provides interesting insight into many structures on the quantum group, but unfortunately does not do everything we hope. 

A fundamental problem arises when we try to make a Hall algebra into a Hopf algebra. In the initial definition of the Hall algebra, we start with a nice associative multiplication. We also find that the Hall algebra is a coalgebra with an equally nice coassociative comultiplication. However, when we try to check that the algebra and coalgebra fit together to form a bialgebra, we see this fails. Instead, the combination of these maps obeys `Green's Formula', a relationship between the multiplication and comultiplication which we describe in detail below (Proposition \ref{GF}). This formula basically says that the Hall algebra is `almost' a bialgebra in the standard category $\Vect$. Specifically, we only miss having a bialgebra by a coefficient. To see where this extra coefficient comes from, consider the string diagrams which describe the general bialgebra compatibility axiom. As is standard, we will write multiplication of two elements as:
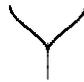
\[\xy
(-5,0);(0,-6)**\crv{(-5,-2)&(0,-4)};
(5,0);(0,-6)**\crv{(5,-2)&(0,-4)};
(0,-6);(0,-9)**\crv{};
\endxy\]
and comultiplication of an element as:
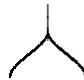
\[\xy
(-5,0);(0,6)**\crv{(-5,2)&(0,4)};
(5,0);(0,6)**\crv{(5,2)&(0,4)};
(0,6);(0,9)**\crv{};
\endxy\]
We can then draw the bialgebra axiom as follows. We first draw multiplication, followed by comultiplication, which looks like this:
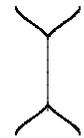
\[ \xy 0;/r.10pc/:
(-10,-20)*{}="b1"; (10,-20)*{}="b2";
(-10,20)*{}="T1"; (10,20)*{}="T2";
(0,10)*{}="C";(0,-10)*{}="D";
"T1";"C"**\crv{(-10,17)&(0,13)};
"T2";"C"**\crv{(10,17)&(0,13)};
"b1";"D"**\crv{(-10,-17)&(0,-13)};
"b2";"D"**\crv{(10,-17)&(0,-13)};
"D";"C"**\crv{};
\endxy \] 
\noindent This should equal the result of comultiplying each element and then multiplying the resulting tensor product of elements. This will look like:
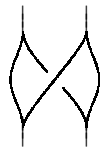
\[\xy 0;/r.10pc/:
(-10,-20)*{}="b1"; (10,-20)*{}="b2";
(-10,20)*{}="t1"; (10,20)*{}="t2";
(-10,15)*{}="A";(10,15)*{}="B";
(-10,-15)*{}="E";(10,-15)*{}="F";
"t1";"A"**\crv{};
"t2";"B"**\crv{};
"A";"E"**\crv{(-10,10)&(-18,0)&(-10,-10)};
"B";"F"**\crv{(10,10)&(18,0)&(10,-10)};
"E";"b1"**\crv{};
"F";"b2"**\crv{};
"B";"E"**\crv{(10,10)&(-10,-10)}\POS?(.5)*{\hole}="H";
"A";"H"**\crv{(-10,10)};
"H";"F"**\crv{(10,-10)};
\endxy \]
But there is wrinkle, namely the braiding of the strings halfway down the diagram. This means we must be working in a braided monoidal category. For the Hall algebra, the seemingly natural choice to work in would be $\Vect$, the category of vector spaces and linear operators. In $\Vect$ the obvious braiding would simply swap elements with no coefficient. However we have already noted that in $\Vect$ the Hall algebra does not satisfy the bialgebra condition as desired.

To `fix' this, a new structure called a `twisted' bialgebra is usually introduced, where swapping the order of elements can still be done, but at the price of an extra coefficient. This coefficient becomes $q^{-\langle A,D\rangle}$ when swapping elements $A$ and $D$, where $\langle A,D\rangle$ is a bilinear form on a group $K$, which is the Grothendieck group of $\Rep(Q)$. 

To obtain a true (untwisted) bialgebra, one then extends the Hall algebra to some larger algebra and alters the multiplication and comultiplication. This process is interesting in its own right, because the result is isomorphic to a larger piece of a quantum group, namely the universal enveloping algebra of the Borel subalgebra, $\mathfrak{b}$. However, we want to take a different direction to avoid the artificial nature of this fix.

We will approach the problem directly. Instead of describing the Hall algebra as a `twisted' bialgebra, we will find a braided monoidal category other than $\Vect$ where the Hall algebra is a true bialgebra object. We accomplish this by giving the category of $K$-graded vector spaces, $\Vect^K$, a braiding that encodes the twisting in the Hall algebra. This works since the extra coefficient $q^{-\langle A,D\rangle}$ from Green's Formula depends on the crossing strands in the diagram for the bialgebra axiom. This approach is more natural because it accounts for the extra coefficient without introducing a `twisted' bialgebra and later changing the multiplication and comultiplication. This more elegant approach was mentioned by Kapranov \cite{Kapranov} but details were not provided. Also, Kapranov was working with the same twisted multiplication and comultiplication as Ringel \cite{Ringel2}, where we are using the simpler, non-twisted versions of the maps instead.

We conclude Chapter \ref{HA} by providing the antipode for this bialgebra to show that the Hall algebra is a Hopf algebra object in our new category. The next step is to categorify all of this. What this means is that we would like to replace our braided monoidal category of $K$-graded vector spaces with a braided monoidal bicategory. We expect that the categorified Hall algebra will live in here. All of this is based on a general categorification of linear algebra called `groupoidification'. In Chapter \ref{groupoidification} we describe the basics of groupoidification, as first described by Baez, Hoffnung, and the current author \cite{BaezHoffnungWalker:2009HDA7}. Specifically, we show how one can replace vector spaces with groupoids, and linear operators with spans of groupoids. While $\Vect$ is a category, groupoids form a bicategory, because we also have a notion of maps between spans of groupoids.

However, the Hall algebra is not just a vector space, but a graded vector space. So we really need to not just categorify vector spaces, but graded vector spaces as well. To do this, we need to consider not just groupoids, but groupoids equipped with a functor to a fixed groupoid $\mathcal{A}_0$.  We call these `groupoids over $\mathcal{A}_0$'.  Specifically, the groupoid $\mathcal{A}_0$ will be the underlying groupoid of the abelian category $\Rep(Q)$. A groupoid over $\mathcal{A}_0$ resembles a $K$-graded vector space, because $K$ is the Grothendieck group of $\Rep(Q)$. For any object in this groupoid, the vector it corresponds to will have grade equal to the image of that element in $\mathcal{A}_0$.

In Chapter \ref{GHA} we will describe in detail the braided monoidal bicategory of these groupoids over $\mathcal{A}_0$, along with the appropriate notion of `span' and `map between spans'. For this bicategory, the monoidal structure will be constructed very simply from the direct sum in the abelian category $\Rep(Q)$ and the cartesian product of groupoids. The braiding, however, will we significantly more interesting. We first note that in the Hall algebra construction described in Chapter \ref{HA}, the braiding was primarily defined by the positive, rational coefficients $q^{-\langle A,D\rangle}$. One of the interesting features of the groupoidification program is that we have a way to describe any positive real number as the cardinality of some groupoid. We find such a groupoid, called $\EXT$, which combines the familiar functors $\Hom$ and $\Ext^1$. This groupoid is the building block of the braiding span for this monoidal bicategory. We finish the section by verifying that our construction indeed gives a braided monoidal bicategory.

What remains to be done is verifying that there is an object in this bicategory which is an analog of the Hall algebra. It turns out that the correct choice for this will be the groupoid $\mathcal{A}_0$ over itself! In Chapter \ref{H2A} we will describe the multiplication and comultiplication spans, and verify that they degroupoidify into the multiplication and comultiplication for the Hall algebra. What remains to be done is verifying that these spans fit together to form a `Hopf $2$-algebra' in our braided monoidal bicategory. There have been several notions of Hopf $2$-algebra presented in recent years. One of the first versions came from Neuchl \cite{Neuchl} who called it a Hopf category. Hopf $2$-algebras have also been studied by  Pfeiffer \cite{Pfeiffer}, and another version was recently described by Fregier and Wagemann \cite{Fregier:H2A}. With any of these definitions, the work lies in verifying quite a large number of coherence laws, but the early calculations presented here lead us to believe we have made the right choices for our categorified Hall algebra.

We should admit that the approach described here is only one of many strategies for categorifying quantum groups, and far from the most sophisticated.
Crane and Frenkel \cite{CraneFrenkel} first sketched out the idea of a Hopf category, and conjectured the existence of a canonical basis for any quantum group, which suggests that one can construct a Hopf category with these basis elements as objects. This canonical basis was later made precise by Lusztig \cite{Lusztig1, Lusztig2, Lusztig3}and Kashiwara \cite{Kashiwara1, Kashiwara2, Kashiwara3}. Recently there has been significant progress by others in providing a categorification of quantum groups. Khovanov and Lauda \cite{KL1, KL2, KL3} presented an approach which uses a diagramatic calculus to provide a categorification directly from the Cartan datum. They initially presented calculations for $U_q(\mathfrak{sl}_2)$, but the also gave a conjecture for the quantum group $U_q(\g)$ for any simple Lie algebra $\g$. Also, categorifying the representation theory of quantum groups has recently received attention from Chuang and Rouquier \cite{ChuangRouquier, Rouquier}. There is still more work to be done. In this thesis we work with quantum groups for values of $q$ which are powers of primes. Categorification of quantum groups over roots of unity has received some attention \cite{Khovanov}, but definitely needs to be developed further. Also, we have yet to construct a braided monoidal bicategory which categorifies the representation theory.   

\chapter{Hall Algebras}\label{HA}
In this chapter we will introduce the theory of Hall algebras, and provide a new perspective on the structure theory which we will later use to categorify everything. 
\section{Hall Algebras}\label{hall}
In this section we will describe the construction of the Ringel-Hall algebra. We begin with a quiver $Q$ (i.e. a directed graph) whose underlying graph is that of a simply-laced Dynkin diagram. We will then consider the abelian category $\Rep(Q)$ of all finite dimensional representation of the quiver $Q$ over a fixed finite field $\mathbb{F}_q$.

We start by fixing a finite field $\F_q$ and a directed graph $D$, 
which might look like this:
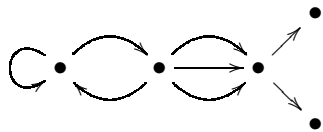
\[ \xymatrix@=10pt{
 &&&&& \bullet  \\
 \bullet \ar@(ul,dl)[] \ar@/^1pc/[rr]
 && \bullet \ar@/^1pc/[ll] \ar@/^1pc/[rr] \ar@/_1pc/[rr] \ar[rr]
 && \bullet \ar[ur] \ar[dr] \\
 &&&&& \bullet  
 }
\]
We shall call the category $Q$ freely generated by $D$ a \textbf{quiver}.
The objects of $Q$ are the vertices of $D$, while the morphisms are
edge paths, with paths of length zero serving as identity morphisms.

By a \textbf{representation} of the quiver $Q$ we mean a functor 
\[   R \maps Q \to \FinVect_q, \]
where $\FinVect_q$ is the category of finite-dimensional vector spaces
over $\F_q$.  Such a representation simply assigns a vector space
$R(d) \in \FinVect_q$ to each vertex of $D$ and a linear operator
$R(e) \maps R(d) \to R(d')$ to each edge $e$ from $d$ to $d'$.  By a
\textbf{morphism} between representations of $Q$ we mean a natural
transformation between such functors.  So, a morphism $\alpha \maps R
\to S$ assigns a linear operator $\alpha_d \maps R(d) \to S(d)$ to
each vertex $d$ of $D$, in such a way that
\[
\xymatrix{
R(d) \ar[d]_{\alpha_d} \ar[r]^{R(e)} & R(d') \ar[d]^{\alpha_{d'}} \\
S(d) \ar[r]_{S(d)} & S(d')
}
\]
commutes for any edge $e$ from $d$ to $d'$.  There is a category
$\Rep(Q)$ where the objects are representations of $Q$ and the
morphisms are as above.  This is an abelian category, so we can speak
of indecomposable objects, short exact sequences, etc.\ in this
category.

In 1972, Gabriel \cite{Gabriel} discovered a remarkable fact.  Namely:
a quiver has finitely many isomorphism classes of indecomposable
representations if and only if its underlying graph, ignoring the
orientation of edges, is a finite disjoint union of Dynkin diagrams of
type $A, D$ or $E$.  These are called {\bf simply laced} Dynkin
diagrams.

Henceforth, for simplicity, we assume the underlying graph of our
quiver $Q$ is a simply laced Dynkin diagram when we ignore the
orientations of its edges.  Let $X$ be the underlying groupoid of
$\Rep(Q)$: that is, the groupoid with representations of $Q$ as
objects and \textit{isomorphisms} between these as morphisms.  We will
use this groupoid to construct the Hall algebra of $Q$.

As a vector space, the Hall algebra is just $\R[\u{X}]$.
Recall that this is the vector space
whose basis consists of isomorphism classes of objects in $X$.  In
fancier language, it is the zeroth homology of $X$.

We now focus our attention on the Hall algebra product.  Given three quiver
representations $M,N,$ and $E$, we define the set:
\[\mathcal{P}_{MN}^E=
\{(f,g): 
0\to N \stackrel{f}{\rightarrow} E \stackrel{g}{\rightarrow} M \to 0 
\textrm{\; is exact} \} \]
and we call its set cardinality $P_{MN}^E$. In the chosen category this set has a finite cardinality, since each representation is a finite-dimensional vector space over a finite field. 
The Hall algebra product counts these exact sequences, but with a
subtle `correction factor':
\[[M] \cdot [N] =\sum_{[E] \in \u{X}} 
\frac{P_{MN}^E}{\aut(M) \, \aut(N)}\, [E] \,.\]
Where we call $\aut(M)$ the set cardinality of the group $\Aut(M)$.

Somewhat surprisingly, the above product is associative.  In fact,
Ringel \cite{Ringel} showed that the resulting algebra is isomorphic
to the positive part $U_q^+ \g$ of the quantum group corresponding to
our simply laced Dynkin diagram!  So, roughly speaking, the Hall algebra
of a simply laced quiver is `half of a quantum group'.

This isomorphism also relates to a coalgebra structure on the Hall algebra. Using the same ideas from the multiplication formula, we can define a comultiplication on the Hall algebra to be a carefully weighted sum on ways to `factor' a representation via short exact sequences. Formulaically this becomes:
\[\Delta(E)=\sum_{[M],[N] \in \u{X}} 
\frac{|\mathcal{P}_{MN}^E|}{ \aut(E)}\, [N]\ten [M] \,.\]
Again, Ringel found that these are the correct factor to make the comultiplication coassociative. However, we immediately run into a problem; these two maps do not satisfy the compatibility condition for a bialgebra.

\section{The Category of K-graded Vector Spaces}\label{gvs}

It is interesting to note that the standard multiplication and comultiplication on $U_q^+ \g$ (which the Hall algebra is isomorphic to) also do not satisfy the compatibility axiom of a bialgebra, so we should not expect the Hall algebra to, either. This does not mean there is not an interesting relationship between the multiplication and comultiplication in the Hall algebra. This relationship is often described as being a `twisted' bialgebra, where we do not use the standard extension of the multiplication to the tensor product. We would like to take a different point of view here. It turns out that the bialgebra axiom can be satisfied if we change the category in which we ask for them to be compatible. 

In order to describe this new category, we will start with a definition of the Grothendieck group of a general abelian category.
\begin{definition}
Let $\mathcal{A}$ be an abelian category. We can define an equivalence relation on isomorphism classes of objects in $\mathcal{A}$ by $[A]+[B]=[C]$ if there exists a short exact sequence $0\to A\to C\to B\to 0$. The set of equivalence classes under this relation form a group $K_0(\mathcal{A})$ called the {\rm \textbf{Grothendieck group}}.
\end{definition} 
$K_0(\mathcal{A})$ has a universal property in the following sense. Given any abelian group $G$, any additive function $f$ from isomorphism classes of $\mathcal{A}$ to the group $G$ will give a unique abelian group homomorphism $\tilde{f}\maps K_0(\mathcal{A})\to G$ such that the following diagram commutes:

\[\xymatrix{
\mathcal{A}\ar[rr]\ar[dr]_f &  &  K_0(\mathcal{A}) \ar[dl]^{\exists !\tilde{f}} \\
  & G & \\
}\]

The original purpose of the Grothendieck group was to study Euler characteristics, and this is precisely why we are interested in them here.
In many of the standard references for Hall algebras \cite{Hubery, Schiffman} the characteristics of the Grothendieck group of $\Rep(Q)$ are explained explicitly. Many of these properties follow from the fact that $\Rep(Q)$ is hereditary.  We can also describe these properties in the general case of an abelian category $\mathcal{A}$ which has finite homological dimension. However, to construct the entire Hall algebra, our
abelian category will need to hold to the extra finiteness properties that the
groups $\Ext^i(M,N)$ must be finite. This condition is sufficient since it makes the sets $\mathcal{P}^E_{MN}$ finite, and makes the bilinear form in the next proposition well defined.

\begin{proposition}Let $\mathcal{A}$ be an abelian $k$-linear category for some field $k$. Suppose that $\mathcal{A}$ has finite homological dimension $d$ and $\dim\Ext^i(M,N)$ is finite for all objects $M,N\in \mathcal{A}$. If $K=K_0(\mathcal{A})$ is the Grothendieck group of $\mathcal{A}$, then $K$ admits a bilinear form $\langle\cdot,\cdot\rangle\maps K\times K\to \C$ given by:
\[\langle \underline{m},\underline{n}\rangle = \sum_{i=0}^d (-1)^{i}\dim\Ext^i(M,N)\]
\end{proposition}

\begin{proof}
We prove the theorem for $d=1$ (i.e. when the category is hereditary) since this is the main case we will use. The case when $d=0$ is simply bilinearity of $\Hom$, and the cases where $d>1$ follows by a similar argument to $d=1$.\\

For $d=1$ we need to show that:
\[\dim\Hom(M,N_1\oplus N_2)-\dim\Ext^1(M,N_1\oplus N_2)=\]
\[\dim\Hom(M,N_1)-\dim\Ext^1(M,N_1)+\dim\Hom(M,N_2)-\dim\Ext^1(M,N_2)\]
we begin with the short exact sequence:
\[0\to N_1\stackrel{i_1}{\rightarrow} N_1\oplus N_2 \stackrel{\pi_2}{\rightarrow} N_2\to 0\]
which, since $d=1$, gives rise to the long exact sequence:
\[0\to \Hom(M,N_1) \stackrel{\tilde{i_1}}{\rightarrow} \Hom(M,N_1\oplus N_2) \stackrel{\tilde{\pi_2}}{\rightarrow} \Hom(M,N_2) \stackrel{h}{\rightarrow}\] \[\Ext^1(M,N_1)\stackrel{\hat{i_1}}{\rightarrow} \Ext^1(M,N_1\oplus N_2) \stackrel{\hat{\pi_2}}{\rightarrow} \Ext^1(M,N_2) \to 0.\]
Using a variety of basic equations from the fact that this sequence is exact, as well as some dimension arguments, the left hand sides becomes:
\[\dim\Hom(M,N_1\oplus N_2)-\dim\Ext^1(M,N_1\oplus N_2)\]
\[=\dim \im \tilde{\pi_2}+\dim \ker \tilde{\pi_2}-\dim \im \hat{\pi_2}-\dim \ker \hat{\pi_2}\]
and the right hand side turns into:
\[\dim\Hom(M,N_1)-\dim\Ext^1(M,N_1)+\dim\Hom(M,N_2)-\dim\Ext^1(M,N_2)\]
\[=\dim \im h+\dim \ker h+\dim\im \tilde{i_1}-\dim \im \hat{i_1}-\dim \ker \hat{i_1}-\dim\im \hat{\pi_2}\]
\[=\dim \ker \hat{i_1}+\dim \im \tilde{\pi_2}+\dim\ker \tilde{\pi_2}-\dim \ker \hat{\pi_2}-\dim \ker \hat{i_1}-\dim\im \hat{\pi_2}\]
\[=\dim \im \tilde{\pi_2}+\dim\ker \tilde{\pi_2}-\dim \ker \hat{\pi_2}-\dim\im \hat{\pi_2}.\]
\end{proof}

In general, it is possible to construct a braided monoidal category $\Vect^G$ from any abelian group $G$ equipped with a bilinear form $\langle\cdot,\cdot\rangle$. One common example is the category of super-algebras, which can be thought of in this context in terms of the group $\Z_2$ with its unique non-trivial bilinear form. Joyal and Street \cite{JoyalStreet} described the general idea of constructing a braided monoidal category from a bilinear form. In the next theorem, we will describe how this braiding works in detail for our desired case of the Grothendieck group $K=K_0(\mathcal{A})$ with the previously described bilinear form. 

\begin{theorem}\label{BMC} Let $\mathcal{A}$ be an abelian, $k$-linear category with finite homological dimension. Let $K=K_0(\mathcal{A})$ be its Grothendieck group, and suppose $\dim\Ext^i(M,N)$ is finite for all objects $M,N \in\mathcal{A}$. Then, the category $\Vect^K$ of $K$-graded vector spaces and grade preserving linear operators is a braided monoidal category, with the braiding given by:
\[B_{V,W}:V\ten W\to W\ten V\]
\[v\ten w \mapsto q^{-\langle \underline{n},\underline{m}\rangle}w\ten v\]
where $q$ is a non-zero element of $k$.
\end{theorem}

\begin{proof} The monoidal structure on this category is just the tensor product in the category $\Vect$. To define a braiding on this category, we first note that the braiding is defined by isomorphisms in the category which are graded linear operators. Because of linearity, it is enough to define these isomorphisms on a single graded piece. Also note that for any two $K$-graded vector spaces $V$ and $W$, a graded piece of the tensor product $V\ten W$ can be written as the sum of tensor products of graded pieces from $V$ and $W$, or more precisely:
\[(V\ten W)_{\underline{d}}=\bigoplus_{\underline{n}\in K} V_{\underline{n}}\ten W_{\underline{d}-\underline{n}}.\]
This lets us define the braiding $B_{V,W}\maps V\ten W\to W\ten V$ only on the tensor product of the graded piece $V_{\underline{n}}\ten W_{\underline{m}}$. 
We thus define the map: 
\[B_{\underline{n},\underline{m}}\maps V_{\underline{n}}\ten W_{\underline{m}}\to W_{\underline{m}}\ten V_{\underline{n}}\]
\[v\ten w \mapsto q^{-\langle \underline{n},\underline{m}\rangle}w\ten v\]
which is easily seen to be an isomorphism. We only need to check the hexagon equations, i.e. ones of the form: 
\[\xymatrix{
             & (W\ten V)\ten U\ar[r]^\alpha & W\ten (V\ten U)\ar[dr]^{1\ten B_{V,U}} & \\
(V\ten W)\ten U \ar[ur]^{B_{V,W}\ten 1}\ar[dr]_\alpha & & & W\ten (U\ten V) \\
             & V\ten (W\ten U) \ar[r]_{B_{V,W\ten U}} & (W\ten U)\ten V\ar[ur]_\alpha & \\
}\]
We will make the argument for the above hexagon identity, noting the the other versions follow by a similar argument. Now, since we have restricted ourselves to vector spaces with a single grade, it is enough to chase a general element around this diagram. let $v\in V_{\underline{n}}$, $w\in W_{\underline{m}}$, and $u\in U_{\underline{p}}$. The top path of the hexagon diagram yields the composite:
\[(v\ten w)\ten u \mapsto q^{-\langle\underline{n},\underline{m}\rangle-\langle\underline{n},\underline{p}\rangle}w\ten(u\ten v).\]
For the bottom path we note that $v\ten w\in (V\ten W)_{\underline{m+p}}$, so we get the composite:
\[(v \ten w)\ten u \mapsto q^{-\langle \underline{n}, \underline{m+p}\rangle} w \ten(u \ten v).\]
Hence, commutativity of the diagram will follow from the equality \[-\langle\underline{m},\underline{n}\rangle-\langle\underline{m},\underline{p}\rangle=-\langle \underline{m}, \underline{n+p}\rangle,\] which is precisely bilinearity of the form $\langle\cdot, \cdot\rangle$.
\end{proof}

\section{The Hopf Algebra Structure}
Now we consider our Hall algebra in the braided monoidal category $\Vect^K$. The concept of a Hopf algebra object in a braided monoidal category was described by Majid \cite{Majid}, where he called it a `braided group', but later \cite{Majid2} described it in the way we will use here. The basic idea is to ask if the standard defining commutative diagrams for a Hopf algebra hold in some braided monoidal category, instead of the symmetric monoidal category $\Vect$. For the remainder of this section, we will let $Q$ be a simply laced Dynkin quiver. We will focus back on the specific abelian category $\Rep(Q)$ and the category of $K_0(\Rep(Q))-$graded vector spaces, which we showed in Section \ref{gvs} to be a braided monoidal category. Remember that $\Rep(Q)$ is hereditary and satisfies all the finiteness conditions of Section \ref{gvs}. We can now state the main theorem of this paper.
\begin{theorem}\label{HO} The Hall algebra of $\Rep(Q)$ is a Hopf algebra object in the category $Vect^K$. \end{theorem}
To prove this theorem we need to work through the following lemmas. For what follows, we will set $X$ to be the underlying groupoid of $\Rep(Q)$, $\underline{X}$ to be the set of isomorphism classes in $X$, and $K=K_0(\Rep(Q))$. Recall that $R[\underline{X}]$ is the vector space of all finite linear combinations of elements of $\u{X}$. This vector space, which is the underlying vector space of the Hall algebra, is easily seen to be $K$ graded.
\begin{lemma}
The vector space $\Hall=\R[\underline{X}]$ is a $K$-graded vector space, with the grading on each isomorphism class $[M]\in\u{X}$ given by its image in $K$.
\end{lemma}

For the next two lemmas, we note that the multiplication and comultiplication described were shown to be associative and coassociative in the original category $\Vect$ by Ringel \cite{Ringel}. This fact passes to our new category since neither axiom requires or depends on the particular braiding on vector spaces, so we will not repeat the argument. After stating both lemmas, we will provide a brief description of why each one is a morphism in the new category $\Vect^K$.

\begin{lemma}\label{mult} The multiplication map $m:\Hall\ten \Hall\to \Hall$ defined on basis elements by:

\[m([M]\ten[N]) =\sum_{[E]} 
\frac{P_{MN}^E}{\aut(M) \, \aut(N)}\, [E]\]
is a morphism in $\Vect^K$.
\end{lemma}

\begin{lemma}\label{comult} The comultiplication map $\Delta:\Hall\to \Hall\ten \Hall$ defined on basis elements by:

\[\Delta(E)=\sum_{[M],[N]} 
\frac{P_{MN}^E}{ \aut(E)}\, [N]\ten [M]\]
is a morphism in $\Vect^K$. 
\end{lemma}
Note that when Q is a simply-laced Dynkin quiver, the sums in Lemmas \ref{mult} and \ref{comult} are finite. Both of these lemmas are true for a similar reason. The important fact to note here is that for a fixed $M$, $N$, and $E$ in either sum, there is a short exact sequence $0\to N\to E\to M\to 0$. So by the definition of the Grothendieck group $K$, we have that their images obey the identity $[M]+[N] = [E]$. These images determine the grade of the corresponding graded piece they sit in, so the grade is clearly preserved by both maps.

Now we can focus on the compatibility of the new maps, which was the main reason for constructing this new category. We first need an important identity for the multiplication and comultiplication known as Green's Formula.

\begin{proposition}\label{GF}{\rm (Green's Formula).} For all $M$, $N$, $X$, and $Y$ in $\Rep(Q)$ we have the identity:
\[\sum_{[E]}{\frac{P^E_{MN}P^E_{XY}}{\aut(E)}}=\sum_{[A],[B],[C],[D]}{q^{-\langle A,D\rangle}\frac{P^{M}_{AB}P^{N}_{CD}P^{X}_{AC}P^{Y}_{BD}}{\aut(A)\aut(B)\aut(C)\aut(D)}}.\]
\end{proposition}
The proof of Green's formula is quite complex, and involves a large amount of homological algebra. It was first presented by Ringel \cite{RingelGreen}, and also appears in \cite{Hubery} and \cite{Schiffman} with good explanations. What we are interested in is the consequence of Green's formula.

We observe in Green's formula the presence of our braiding coefficient $q^{-\langle A,D \rangle}$. It is important to note that this coefficient depends on what some might view as the ``outside'' objects $A$ and $D$, and not the ``inside'' objects $B$ and $C$. We deal with this by using a different comultiplication than the one usually described in the literature \cite{Hubery, Schiffman}. In fact, in the category $\Vect$ our chosen comultiplication is the opposite of the standard choice.

\begin{lemma}
In the category $\Vect^K$ the multiplication $m$ and comultiplication $\Delta$ satisfy the bialgebra condition, and thus $\Hall$ is a bialgebra object in $\Vect^K$.
\end{lemma}

\begin{proof}
All the hard work for this proof was done in proving Green's Formula. We now just need to check that Green's Formula gives us the bialgebra compatibility. First we will multiply two objects, then comultiply the result to get:
\[\begin{array}{rl}
\Delta([M]\cdot[N]) & = \displaystyle{\sum_{[E]} \frac{P^E_{MN}}{\aut(M)\aut(N)}\Delta([E])} \\
                    & = \displaystyle{\sum_{[X],[Y]}\sum_{[E]} \frac{P^E_{MN}P^E_{XY}}{\aut(M)\aut(N)\aut(E)}[Y]\ten [X]}\\
\end{array}
\]
On the other hand, if we first comultiply each object, then multiply the resulting tensor products we have:
\[\Delta([M])\cdot \Delta([N]) = \sum_{[A],[B],[C],[D]} \frac{P^M_{AB}P^N_{CD}}{\aut(M)\aut(N)} ([B]\ten[A])\cdot ([D]\ten [C])\]
To continue, we need to remember the in our category $\Vect^K$ the braiding is non-trivial. This means that if we want to extend the multiplication on $\mathcal{H}$ to $\mathcal{H}\ten \mathcal{H}$ we must include the braiding coefficient. Specifically, we get the formula:
\[([B]\ten [A])\cdot ([D]\ten [C]) = q^{-\langle A,D\rangle} [B]\cdot[D]\ten [A]\cdot[C]\]
When substituted above, this yields:
\[\begin{array}{c}
\displaystyle{\sum_{[A],[B],[C],[D]} \frac{P^M_{AB}P^N_{CD}}{\aut(M)\aut(N)} ([B]\ten[A])\cdot ([D]\ten [C])}\\
 = \displaystyle{\sum_{[A],[B],[C],[D]} q^{-\langle A,D\rangle}\frac{P^M_{AB}P^N_{CD}}{\aut(M)\aut(N)} [B]\cdot[D]\ten [A]\cdot[C]}\\
 = \displaystyle{\sum_{[X],[Y]}\sum_{[A],[B],[C],[D]}\frac{q^{-\langle A,D\rangle} P^{M}_{AB}P^{N}_{CD}P^{X}_{AC}P^{Y}_{BD}}{\aut(M)\aut(N)\aut(A)\aut(B)\aut(C)\aut(D)}[Y]\ten[X]}\\
\end{array}
\]
Thus, Green's formula give the equality of the two sides.
\end{proof}

For completeness, we will also define an antipode for this bialgebra object to make it a Hopf object. This map is also a morphism in $\Vect^K$ since it clearly preserves the grading.
\begin{lemma}
The map $S:\Hall\to \Hall$ defined on generators by:
\[S([M])=-[M]\]
is a $K$-grade preserving linear operator, and is an antipode for $\Hall$. Thus $\Hall$ is a Hopf algebra object in $\Vect^K$.
\end{lemma}

It is possible to generalize these results to other abelian categories, provided they obey the same finiteness properties as $\Rep(Q)$.

\begin{theorem}
Let $\mathcal{A}$ be an abelian, $k$-linear, hereditary category. Let $K=K_0(\mathcal{A})$ be its Grothendieck group, and suppose $\dim\Ext^i(M,N)$ is finite for all objects $M,N \in\mathcal{A}$. If the sum 
\[\sum_{[M],[N]} 
\frac{P_{MN}^E}{ \aut(E)}\, [N]\ten [M]\]
is finite for all objects $E\in \mathcal{A}$, then the Hall algebra $\mathcal{H}(\mathcal{A})$ is a Hopf object in $\Vect^K$.
\end{theorem}

\begin{proof}
Examining the proof of Theorems \ref{BMC} and \ref{HO}, we see these are the conditions that we need to generalize the result from the case $\mathcal{A} = \Rep(Q)$ to other abelian categories. Specifically, we need that $\mathcal{A}$ is hereditary to prove Green's Theorem.
\end{proof}

\chapter{Groupoidification}\label{groupoidification}
In this chapter, we will give a brief summary of the groupoidification program as it pertains to Hall algebras. For our purposes we will only provide the relevant definitions and theorems, written in the form needed for this example. What we mean by this, is that we will be making two convention choices; namely, we will work with `homology' and `$\alpha=1$'. The meaning of these is only necessary if one wishes to compare the work here to the more general form of groupoidification described by Baez, Hoffnung, and the author in Higher Dimensional Algebra VII: Groupoidification \cite{BaezHoffnungWalker:2009HDA7} (henceforth denoted HDA 7). The work here is selfcontained, except that we omit the proofs of most theorems to conserve space.

The general idea of groupoidification is to replace vector spaces with groupoids and linear operators with some kind of map between groupoids. As we will see, the correct type of morphism between groupoids will be a `span'. We will describe a systematic process for turning groupoids
into vector spaces and `nice' spans into linear operators. This process, `degroupoidification', is in fact a kind of functor.
`Groupoidification' is the attempt to {\it undo} this functor.  To
`groupoidify' a piece of linear algebra means to take some structure
built from vector spaces and linear operators and try to find
interesting groupoids and spans that degroupoidify to give this
structure.  So, to understand groupoidification, we need to master
degroupoidification.

We begin by describing how to turn a groupoid into a vector space.  In
what follows, all our groupoids will be \textbf{essentially small}.  This
means that they have a {\it set} of isomorphism classes of objects,
not a proper class.  We also assume our groupoids are \textbf{locally
finite}: given any pair of objects, the set of morphisms from one object 
to the other is finite.

\begin{definition}
Given a groupoid $X$, let $\u{X}$ be the set of isomorphism
classes of objects of $X$.
\end{definition}
\begin{definition}
\label{vectorspace}
Given a groupoid $X$, let the {\bf degroupoidification} of $X$ be $\R[\u{X}]$,
the vector space with basis $\u{X}$.
\end{definition}

A nice example is the groupoid of finite sets and bijections:
\begin{example}
\label{power_series}
Let $E$ be the groupoid of finite sets and bijections. Then
$\u{E}\iso \N$, so $\R[\u{E}]\cong \R[x]$, the vector space of polynomials in one variable.
\end{example}

A sufficiently nice groupoid over a groupoid $X$ will give a vector
in $\R[\u{X}]$.  To construct this, we use the concept of
groupoid cardinality:

\begin{definition}\label{cardinality}
The {\bf cardinality} of a groupoid $X$ is
\[ |X| = \sum_{[x]\in \u{X}} \frac{1}{|\Aut(x)|} \]
where $|\Aut(x)|$ is the cardinality of the automorphism group of an object
$x$ in $X$.  If this sum diverges, we say $|X| = \infty$.
\end{definition}

The cardinality of a groupoid $X$ is a well-defined nonnegative rational
number whenever $\u{X}$ and all the automorphism groups of
objects in $X$ are finite.  More generally, we say:

\begin{definition}
A groupoid $X$ is {\bf tame} if it is essentially small,
locally finite, and $|X| < \infty$.
\end{definition}
\noindent

We also have an alternate formula for groupoid cardinality when the groupoid is tame.

\begin{lemma}\label{ALTCARD}
If $X$ is a tame groupoid with finitely many objects in each isomorphism
class, then 
\[ |X| 
= \sum_{x \in X} \frac{1}{|\Mor(x,-)|} \]
where $\Mor(x,-) = \bigcup_{y\in X}\hom(x,y)$ is
the set of morphisms whose source is the object $x \in X$.  
\end{lemma}

\begin{proof}
See the proof of Lemma 5.6 in HDA 7 \cite{BaezHoffnungWalker:2009HDA7}.
\end{proof}
It is also important to note that groupoid cardinality is well defined.

\begin{lemma}\label{EQUIVGRPD}
Given equivalent groupoids $X$ and $Y$, $|X| = |Y|$.
\end{lemma}

\begin{proof}
See the proof of Lemma A.13 in HDA 7 \cite{BaezHoffnungWalker:2009HDA7}.
\end{proof}

The reason we use $\R$ rather than $\Q$ as our ground field is that
there are interesting groupoids whose cardinalities are irrational numbers.
The following example is fundamental:

\begin{example}
\textup{
The groupoid of finite sets $E$ has cardinality
\[ |E| ~=~ \sum_{n \in \N} \frac{1}{|S_n|} ~=~ 
\sum_{n \in \N} \frac{1}{n!} ~=~ e. \]
}
\end{example}

With the concept of groupoid cardinality in hand, we now describe
how to obtain a vector in $\R[\u{X}]$ 
from a sufficiently nice groupoid over $X$.

\begin{definition}
Given a groupoid $X$, a {\bf groupoid over $X$} is a groupoid $\Psi$ 
equipped with a functor $v \maps \Psi \to X$.
\end{definition}
\begin{definition}
Given a groupoid over $X$, say $v \maps \Psi \to X$, and an object $x \in X$,
we define the {\bf full inverse image} of $x$,
denoted $v^{-1}(x)$, to be the groupoid where:
\begin{itemize}
\item
an object is an object $a \in \Psi$ such that $v(a) \cong x$;
\item
a morphism $f \maps a \to a'$ is any morphism in $\Psi$ from $a$ to
$a'$.
\end{itemize}
\end{definition}
\begin{definition}
A groupoid over $X$, say $v \maps \Psi \to X$, is {\bf tame} if the 
groupoid $v^{-1}(x)$ is tame for all $x\in X$. 
\end{definition}

\noindent 
We sometimes loosely say that $\Psi$ is a tame groupoid over $X$.
When we do this, we are referring to a functor $v \maps \Psi \to X$
that is tame in the above sense.  We do not mean that $\Psi$ is tame
as a groupoid.

We also need to remember that a vector in $\R[\u{X}]$ is a finite linear combination of basis vectors. Another way to think of this is to consider the vector as a functor with finite support. We then need to describe groupoids over $X$ with the same property.

\begin{definition}
\label{degroupoidification_of_vectors}
Given a tame groupoid over $X$, say $v \maps \Psi\to X$, 
there is a function $\utilde{\Psi}\maps \u{X}\to \R$ 
defined by:
\[ \utilde{\Psi}([x]) = |\Aut(x)||v^{-1}(x)|. \]
We say that a tame groupoid $\Psi$ over $X$ is {\bf finitely supported} 
if $\utilde{\Psi}$ is a finitely supported function on 
$\u{X}$. In this case $\utilde{\Psi}\in \R[\u{X}]$.
\end{definition}

Both addition and scalar multiplication of vectors have groupoidified
analogues.  We can add two groupoids $\Phi$, $\Psi$ over $X$ by taking
their coproduct, i.e., the disjoint union of $\Phi$ and $\Psi$ with
the obvious map to $X$:
\[
\xymatrix{
\Phi + \Psi \ar[d] \\
X
}
\]
We then have:
\begin{proposition}
Given finitely supported groupoids $\Phi$ and $\Psi$ over $X$,
\[\utilde{\Phi + \Psi} = \utilde{\Phi} + \utilde{\Psi}.\]
\end{proposition}

\begin{proof}
See the proof of Lemma 5.4 in HDA 7 \cite{BaezHoffnungWalker:2009HDA7}.
\end{proof}

We can also multiply a groupoid over $X$ by a `scalar'---that is, a
fixed groupoid.  Given a groupoid over $X$, say $v \maps \Phi\to X$,
and a groupoid $\Lambda$, the cartesian product $\Lambda\times \Psi$
becomes a groupoid over $X$ as follows:
\[\xymatrix{
\Lambda\times \Psi \ar[d]^{v\pi_2}\\ X\\ }\] 
where $\pi_2 \maps \Lambda\times \Psi\to \Psi$ is projection onto the
second factor.  We then have:

\begin{proposition}
Given a tame groupoid $\Lambda$ and a finitely supported groupoid $\Psi$ over $X$, the groupoid
$\Lambda \times \Psi$ over $X$ is finitely supported and satisfies
\[\utilde{\Lambda \times \Psi} = |\Lambda|\utilde{\Psi}.\]
\end{proposition}

\begin{proof}
See the proof of Proposition 6.3 in HDA 7 \cite{BaezHoffnungWalker:2009HDA7}.
\end{proof}

We have seen how degroupoidification turns a groupoid $X$ into a vector space
$\R[\u{X}]$.  Degroupoidification also turns any sufficiently 
nice span of groupoids into a linear operator.

\begin{definition}
Given groupoids $X$ and $Y$, a {\bf span} from $X$ to $Y$ is a diagram
\[\xymatrix{
 & S\ar[dl]_q\ar[dr]^p & \\
 Y & & X \\
}\]
where $S$ is groupoid and $p\maps S\to X$ and $q \maps S\to Y$ are functors.
\end{definition}

To turn a span of groupoids into a linear operator, we need a
construction called the `weak pullback'.  This construction will let
us apply a span from $X$ to $Y$ to a groupoid over $X$ to obtain a
groupoid over $Y$.  Then, since a finitely supported groupoid over $X$ gives a
vector in $\R[\u{X}]$, while a finitely supported groupoid over $Y$ gives a
vector in $\R[\u{Y}]$, a sufficiently nice span from $X$ to
$Y$ will give a map from $\R[\u{X}]$ to $\R[\u{Y}]$.
Moreover, this map will be linear.

As a warmup for understanding weak pullbacks for groupoids, we recall
ordinary pullbacks for sets, also called `fibered products'.
The data for constructing such a pullback is a pair of sets equipped
with functions to the same set:
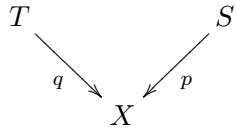
\[
\xymatrix{
& T \ar[dr]_{q} & & S \ar[dl]^{p} & \\
 & & X & & 
}
\]  
The pullback is the set
\[ P = \lbrace (s,t) \in S \times T \, | \; p(s) = q(t) \rbrace  \]
together with the obvious projections $\pi_S \maps P \to S$ and 
$\pi_T \maps P \to T$.  The pullback makes this diamond commute:
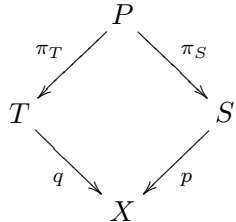
\[
\xymatrix{
& & P \ar[dl]_{\pi_T} \ar[dr]^{\pi_S} & &\\
& T \ar[dr]_{q} & & S \ar[dl]^{p} & \\
 & & X & & 
}
\]  
and indeed it is the `universal solution' to the problem of finding
such a commutative diamond \cite{Mac Lane}.

To generalize the pullback to groupoids, we need to weaken one
condition.  The data for constructing a weak pullback is a pair of
groupoids equipped with functors to the same groupoid:
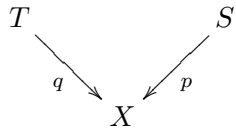
\[
\xymatrix{
& T \ar[dr]_{q} & & S \ar[dl]^{p} & \\
 & & X & & 
}
\]  
But now we replace the {\it equation} in the definition of pullback
by a {\it specified isomorphism}.  So, we define the weak pullback 
$P$ to be the groupoid where an object is a triple $(s,t,\alpha)$ 
consisting of an object $s \in S$, an object $t \in T$, and an 
isomorphism $\alpha \maps p(s) \to q(t)$ in $X$.  A morphism
in $P$ from $(s,t,\alpha)$ to $(s',t',\alpha')$ consists of a morphism
$f \maps s \to s'$ in $S$ and a morphism $g \maps t \to t'$ in $T$
such that the following square commutes:
\[
\xymatrix{
p(s) \ar[d]_{p(f)} \ar[r]^{\alpha} & q(t) \ar[d]^{q(g)} \\
p(s') \ar[r]_{\alpha'} & q(t')
}
\]
Note that any set can be regarded as a {\bf discrete} groupoid:
one with only identity morphisms.  For discrete groupoids, the weak
pullback reduces to the ordinary pullback for sets.
Using the weak pullback, we can apply a span from $X$ to $Y$ to a groupoid
over $X$ and get a groupoid over $Y$.  Given a span of groupoids:
\[
\xymatrix{
& S\ar[dl]_{q} \ar[dr]^{p} & \\
Y & & X
}
\]
and a groupoid over $X$:
\[
\xymatrix{
   & \Psi \ar[dl]_{v} \\
   X &  
}
\]
we can take the weak pullback, which we call $S\Psi$:
\[
\xymatrix{
& & S\Psi \ar[dl]_{\pi_S}\ar[dr]^{\pi_{\Psi}} & \\
& S \ar[dl]_{q} \ar[dr]^{p} & & \Psi \ar[dl]_{v} \\
Y & & X &  
}
\]
and think of $S\Psi$ as a groupoid over $Y$:
\[
\xymatrix{
& S\Psi \ar[dl]_{q \pi_S}  \\
Y  
}
\]
This process will determine a linear operator from $\R[\u{X}]$ 
to $\R[\u{Y}]$ if the span $S$ is sufficiently nice:
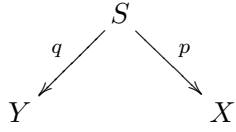
\begin{definition}
A span
\[
\xymatrix{
& S\ar[dl]_{q} \ar[dr]^{p} & \\
Y & & X
}
\]
is {\bf tame} if $v \maps \Psi \to X$ being tame implies that 
$q\pi_{S} \maps S\Psi \to Y$ is tame.
\end{definition}

We also need a concept of a span which preserves groupoids with finite support.

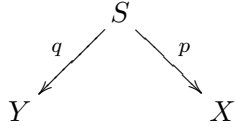
\begin{definition}\label{finitetypespan}
A span:
\[\xymatrix{
& S \ar[dl]_{q} \ar[dr]^{p}& \\
Y & & X
}\]
is of {\bf finite type} if it is a tame span of groupoids and for any
finitely supported groupoid $\Psi$ over $X$, the groupoid $S\Psi$ over
$Y$ (formed by weak pullback) is also finitely supported.
\end{definition}

\begin{theorem}
Given a span of finite type:
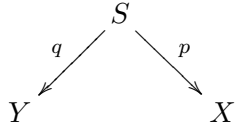
\[
\xymatrix{
& S\ar[dl]_{q} \ar[dr]^{p} & \\
Y & & X
}
\]
there exists a unique linear operator
\[ \utilde{S} \maps \R[\u{X}] \to \R[\u{Y}] \]
such that
\[ \utilde{S}\utilde{\Psi} = \utilde{S\Psi} \]
whenever $\Psi$ is a groupoid over $X$ with finite support.
\end{theorem}

\begin{proof}
See the proof of Theorem 5.7 in HDA 7 \cite{BaezHoffnungWalker:2009HDA7}.
\end{proof}

An explicit criterion for when a span is tame and of finite type is given in \cite{BaezHoffnungWalker:2009HDA7}. For our purposes we would like an explicit formula for the
operator corresponding to a tame span $S$ from $X$ to $Y$.  If
$\u{X}$ and $\u{Y}$ are finite, then
$\R[\u{X}]$ has a basis given by the isomorphism classes
$[x]$ in $X$, and similarly for $\R[\u{Y}]$.  With respect to
these bases, the matrix entries of $\utilde{S}$ are given as follows:
\begin{equation}
\label{matrix_entry}
\utilde{S}_{[y][x]} = 
\sum_{[s]\in\u{p^{-1}(x)}\bigcap \u{q^{-1}(y)} }\frac{|\Aut(y)|}{|\Aut(s)|}
\end{equation}
where $|\Aut(y)|$ is the set cardinality of the automorphism group of
$y \in Y$, and similarly for $|\Aut(s)|$.  Even when $\u{X}$
and $\u{Y}$ are not finite, we have the following formula for
$\utilde S$ applied to $\psi \in \R[{\u X}]$:
\begin{equation}
\label{operator_formula}
(\utilde{S} \psi)([y]) = 
\sum_{[x] \in \u{X}} \;\,
\sum_{[s]\in\u{p^{-1}(x)}\bigcap \u{q^{-1}(y)}}
\frac{|\Aut(y)|}{|\Aut(s)|} \,\, \psi([x]) \, .
\end{equation}

As with vectors, there are groupoidified analogues of addition and scalar
multiplication for operators.  Given two spans from $X$ to $Y$:
\[
\xymatrix{
& S\ar[dl]_{q_S} \ar[dr]^{p_S} & & & T\ar[dl]_{q_T} \ar[dr]^{p_T} & \\
Y & & X & Y & & X
}
\]
we can add them as follows.  By the universal property of the
coproduct we obtain from the right legs of the above spans a functor
from the disjoint union $S + T$ to $X$.  Similarly, from the left legs
of the above spans, we obtain a functor from $S + T$ to $Y$.  Thus, we
obtain a span
\[
\xymatrix{
& S + T\ar[dl] \ar[dr] & \\
Y & & X
}
\]
This addition of spans is compatible with degroupoidification:
\begin{proposition}
If $S$ and $T$ are of finite type from $X$ to $Y$, then so is $S + T$, and
\[    \utilde{S+T} = \utilde{S} + \utilde{T} .\]
\end{proposition}

\begin{proof}
See the proof of Proposition 5.11 in HDA 7 \cite{BaezHoffnungWalker:2009HDA7}.
\end{proof}

We can also multiply a span by a `scalar': that is, a fixed groupoid.
Given a groupoid $\Lambda$ and a span
\[
\xymatrix{
& S\ar[dl]_q \ar[dr]^p & \\
Y & & X
}
\]
we can multiply them to obtain a span
\[
\xymatrix{
& \Lambda \times S \ar[dl]_{q\pi_2} \ar[dr]^{p\pi_2} & \\
Y & & X
}
\]
Again, we have compatibility with degroupoidification:

\begin{proposition}
Given a tame groupoid $\Lambda$ and a span of finite type
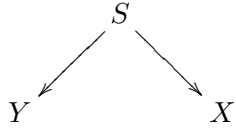
\[
\xymatrix{
& S\ar[dl] \ar[dr] & \\
Y & & X
}
\]
then $\Lambda \times S$ is of finite type and
\[\utilde{\Lambda \times S} = |\Lambda| \, \utilde{S}.\]
\end{proposition}

\begin{proof}
See the proof of Proposition 6.4 in HDA 7 \cite{BaezHoffnungWalker:2009HDA7}.
\end{proof}

Next we turn to the all-important process of {\it composing} spans.
This is the groupoidified analogue of matrix multiplication.  Suppose
we have a span from $X$ to $Y$ and a span from $Y$ to $Z$:
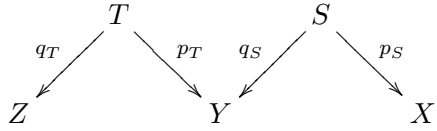
\[
\xymatrix{
& T\ar[dl]_{q_T} \ar[dr]^{p_T} & & S \ar[dl]_{q_S} \ar[dr]^{p_S}& \\
Z & & Y & & X
}
\]  
Then we say these spans are {\bf composable}.
In this case we can form a weak pullback in the middle:
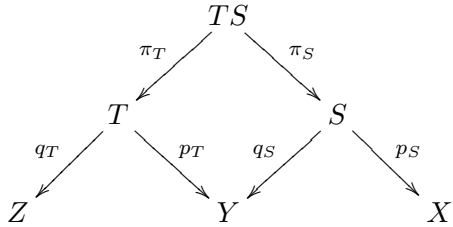
\[
\xymatrix{
& & TS \ar[dl]_{\pi_T} \ar[dr]^{\pi_S} & &\\
& T\ar[dl]_{q_T} \ar[dr]^{p_T} & & S \ar[dl]_{q_S} \ar[dr]^{p_S}& \\
Z & & Y & & X
}
\]  
which gives a span from $X$ to $Z$:
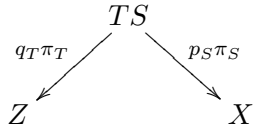
\[\xymatrix{
 & TS \ar[dl]_{q_T \pi_T} \ar[dr]^{p_S\pi_S} & \\
 Z & & X \\
}\]
called the {\bf composite} $TS$.

When all the groupoids involved are discrete, the spans $S$ and $T$
are just matrices of sets.  We urge
the reader to check that in this case, the process of composing spans
is really just matrix multiplication, with cartesian product of sets
taking the place of multiplication of numbers, and disjoint union of
sets taking the place of addition:
\[    (TS)^k_j = \coprod_{j \in Y} T^k_j \times S^j_i . \]
So, composing spans of groupoids is a generalization of matrix
multiplication, with weak pullback playing the role of summing
over the repeated index $j$ in the formula above.

So, it should not be surprising that degroupoidification sends
a composite of tame spans to the composite of their corresponding
operators:

\begin{proposition} If $S$ and $T$ are composable spans of finite type:
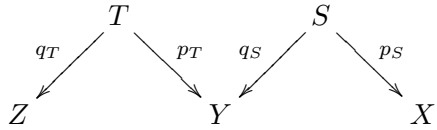
\[
\xymatrix{
& T\ar[dl]_{q_T} \ar[dr]^{p_T} & & S \ar[dl]_{q_S} \ar[dr]^{p_S}& \\
Z & & Y & & X
}
\]  
then the composite span
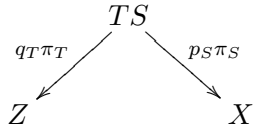
\[\xymatrix{
 & TS \ar[dl]_{q_T \pi_T} \ar[dr]^{p_S\pi_S} & \\
 Z & & X \\
}\]
is also of finite type, and
\[       \utilde{TS} = \utilde{T} \utilde{S}  .\]
\end{proposition}

\begin{proof}
See the proof of Lemma 6.9 in HDA 7 \cite{BaezHoffnungWalker:2009HDA7}.
\end{proof}

What this means is that degroupoidification is a functor 
\[         \utilde{\;\;\;} \,\maps \Span \to \Vect   \]
where $\Vect$ is the category of real vector spaces and linear
operators, and $\Span$ is a category with 
\begin{itemize}
\item
groupoids as objects,
\item
isomorphism classes of finite type spans as morphisms,
\end{itemize}
where composition comes from the method of composing spans we have
just described. So, groupoidification is not merely a way of replacing
linear algebraic structures involving the real numbers with purely
combinatorial structures.  It is also a form of `categorification'
\cite{BaezDolan:1998}, where we take structures defined in the
category $\Vect$ and find analogues that live in the category
$\Span$.

One might notice that our morphisms are defined as isomorphism classes of spans. A deeper approach is to think of
$\Span$ as a bicategory with:
\begin{itemize}
\item
groupoids as objects,
\item
finite type spans as morphisms,
\item
isomorphism classes of maps of spans as $2$-morphisms.
\end{itemize}
Then degroupoidification becomes a functor between bicategories:
\[         \utilde{\;\;} \, \maps \Span \to \Vect   \]
where $\Vect$ is viewed as a bicategory with only identity
$2$-morphisms. Since, again, 2-morphisms are define as isomorphism classes, we could go further and think of $\Span$ as tricategory. This approach is currently being investigated by Hoffnung \cite{Hoffnung:tricat}, but is unnecessary for the example in this project. In the next chapter we will construct study spans of groupoids over a fixed abelian category.  We show that under certain conditions this gives a bicategory which can be made into a braided monoidal bicategory in an interesting way.  This construction groupoidifies the braided monoidal category described in Section \ref{gvs}.

\chapter{The Braided Monoidal Bicategory}\label{GHA}
The main focus of this chapter will be combining the ideas of the previous two chapter to describe a categorification of the Hall algebra. In Theorem \ref{BMC} we showed that if $K$ is the Grothendieck group of a suitable abelian category $\mathcal{A}$, the monoidal category $\Vect^K$ of $K$-graded vector spaces can be given an interesting braiding. Also, the Hall algebra was a Hopf object in this braided monoidal category. So, in this chapter, we describe a braided monoidal bicategory which groupoidifies $\Vect^K$. Starting with the objects, we first have to consider what it would mean for a groupoid (our stand-in for a vector space) to be `graded' over something related to the Grothendieck group $K$. There are two main features of a grading to consider here; what is the grading group, and how is a specific vector related to element of that group. If we want everything to be related to groupoids, we should remember that the Grothendieck group $K$ is generated by equivalence classes of objects in our category $\mathcal{A}=\Rep(Q)$. So we might take the underlying groupoid, $\mathcal{A}_0$, as our replacement for the Grothendieck group. This allows us to describe the `grading' for a specific groupoid as a functor from that groupoid to $\mathcal{A}_0$. 

The chapter will be organized as follows. In Section \ref{bicat}, we will show that there is a bicategory with groupoids over $\mathcal{A}_0$ as objects, certain spans between these as 1-morphisms, and certain maps between spans as 2-morphisms. In Section \ref{monbicat} we set out to find a tensor product that will make this a monoidal bicategory. We start by noting that the category of groupoids has a tensor product given by cartesian product of categories. Also, the abelian category $\Rep(Q)$ has a tensor product given by choosing a specific direct sum for a given pair of objects. By combining these, we can obtain a tensor product for the  bicategory of groupoids over $\mathcal{A}_0$. Now since each of the separate tensor products form a monoidal category, then we will see that coherence laws for the monoidal bicategory structure on groupoids over $\mathcal{A}_0$ will be satisfied trivially; i.e. the pentagon equation holds on the nose.

In Section \ref{EXT} we prepare to define the braided monoidal structure on this bicategory by constructing a groupoid $\EXT(M,N)$ for any pair of objects $M$ and $N$ in an abelian category. This construction is interesting in its own right because it combines the familiar $\Hom$ and $\Ext^1$ functors into a single construction. The reason for this can be seen in the braiding isomorphisms from the Hall algebra construction in Chapter \ref{HA}. One will notice that when braiding two elements in the Hall algebra, the isomorphism adds a coefficient that is dependent on a bilinear form their grades. This bilinear form, called the Euler form, is the difference of the dimensions of $\Hom$ and $\Ext^1$. The value of this form serves as the exponent of $q$ (the dimension of the underlying field for our quiver representations), so `subtraction' becomes `division', and the resulting coefficient can be thought of as the size of $\Ext^1$ divided by the size of $\Hom$. Moving up to groupoids, the cardinality of a groupoid involves division by the size of the automorphism group of each object. A little bit of homological algebra will show that an automorphism of any extension of two representations is equivalent to a homomorphism between the same representations. Thus, the groupoid of extensions of two fixed representations will have the correct groupoid cardinality.

In Section \ref{braidmonbicat} we finish our braided monoidal bicategory by using the groupoids $\EXT$ to construct the braiding span. This leads us to some interesting algebra when verifying the coherence laws for the hexagonator natural isomorphisms. Within these calculations we discover meaning for the laws as facts about splitting and combining short exact sequences of representations. These algebraic facts contain yet another layer of isomorphisms, which hint at the fact that there actually is a tricategory structure lurking around. This however can wait until another time.

\section{Bicategory Structure}\label{bicat}
In this section we will construct the bicategory of groupoids over $\mathcal{A}_0$ and spans. We will be using the definition of bicategory \ref{bicatDEF} from the appendix.
\begin{definition}
Given an abelian category $\mathcal{A}$ and its underlying groupoid $\mathcal{A}_0$, we can define a bicategory $\Span(\Gpd\downarrow\mathcal{A}_0)$ where:
\begin{itemize}
\item An object is a groupoid $X$ equipped with a functor to $\mathcal{A}_0$. We will denote the image in $\mathcal{A}_0$ of an element $x\in X$ by $\u{x}$. 
\item A $1$-morphism is a span of groupoids over $\mathcal{A}_0$ equipped with a natural isomorphism $\alpha$:
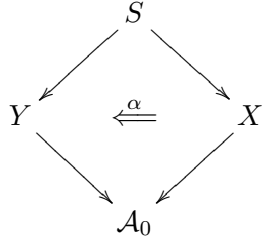
\[\xymatrix{
          & S\ar[dl]\ar[dr]& \\
 Y\ar[dr] &  \stackrel{\alpha}{\Longleftarrow}              & X\ar[dl] \\
          & \mathcal{A}_0  & 
}
\]
\item A $2$-morphisms is an equivalence class of maps between spans. A map between spans is a functor $f\maps S\to S^\prime$ such that the following two triangles commute up to a natural isomorphism:
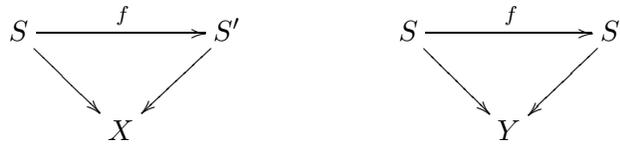
\[\xymatrix{
 S \ar[dr]\ar[rr]^f&   & S^\prime\ar[dl]& & S\ar[dr]\ar[rr]^f &   & S^\prime\ar[dl]\\
                   & X &                & &                   & Y &   \\
}\]
Two such maps between spans are equivalent if they are naturally isomorphic as functors.
\item Composition of morphisms is given by weak pullback of spans:
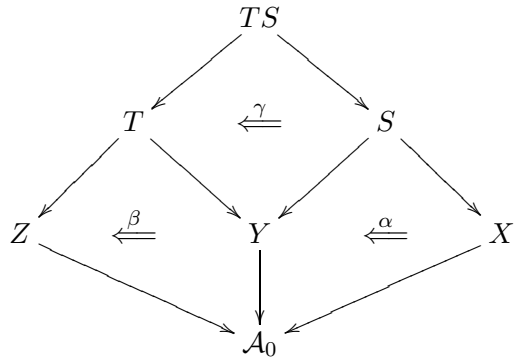
\[\xymatrix{
         &                                  &TS\ar[dl]\ar[dr]  &                                                 &          \\
         &T\ar[dl]\ar[dr]                   &\stackrel{\gamma}{\Longleftarrow} & S\ar[dl]\ar[dr]                                 &          \\
Z \ar[drr]&\stackrel{\beta}{\Longleftarrow} &  Y\ar[d] &  \stackrel{\alpha}{\Longleftarrow}              & X\ar[dll] \\
         &                                  &\mathcal{A}_0     & & \\
}
\]
where $TS$ can be describe explicitly as follows. Let $s_Y$ represent the image of $s$ in $Y$. $TS$ is the groupoid with objects $[(t,s,f)\mid s_Y\stackrel{f}{\rightarrow} t_Y]$ and obvious morphisms. 
\item Associator: for a quadruple of objects $X, Y, Z, W$ and spans $R, S, T$ the associator $a$ gives the $2$-isomorphism:
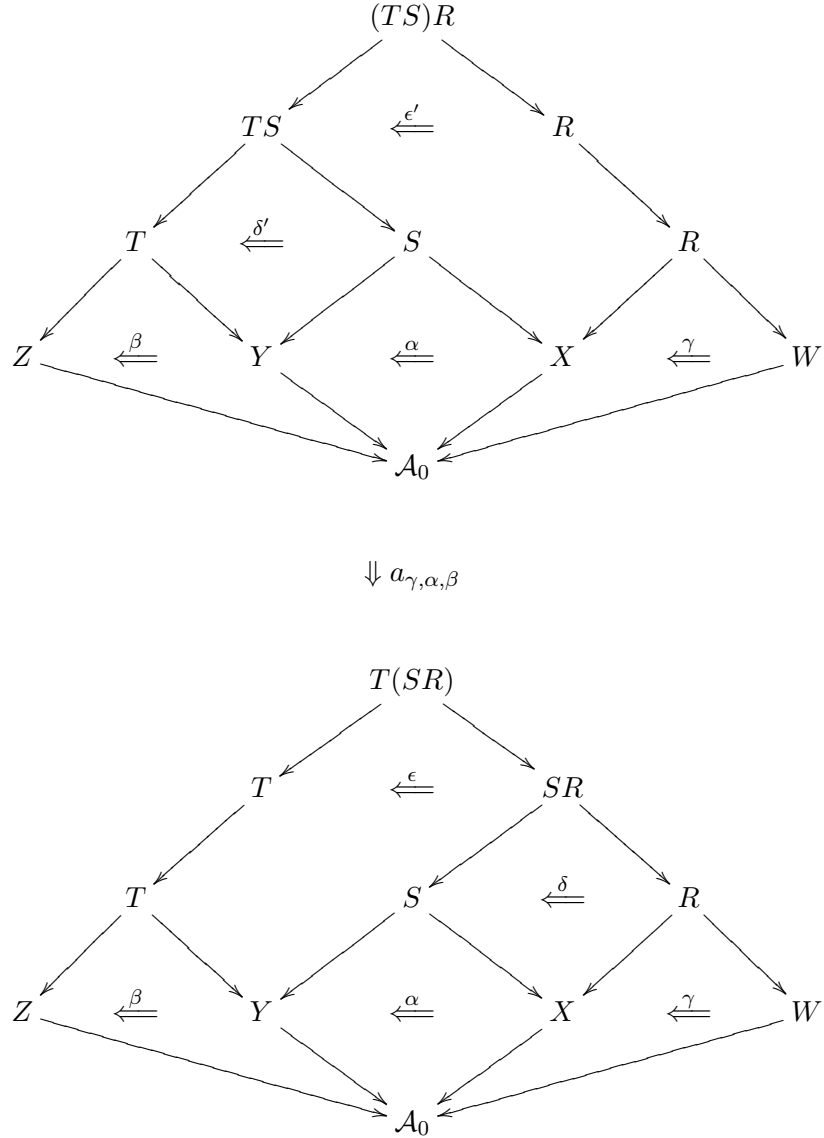
\[\xymatrix{
           &                &         &(TS)R\ar[dr]\ar[dl] &                 &                & \\
           &                &TS\ar[dl]\ar[dr] &\stackrel{\epsilon^\prime}{\Longleftarrow}     &R\ar[dr] &                & \\
           &T\ar[dr]\ar[dl] &\stackrel{\delta^\prime}{\Longleftarrow}         &S\ar[dr]\ar[dl]       & &R\ar[dr]\ar[dl] & \\
Z\ar[drrr] & \stackrel{\beta}{\Longleftarrow} &Y\ar[dr] & \stackrel{\alpha}{\Longleftarrow} &X\ar[dl]         &\stackrel{\gamma}{\Longleftarrow} &W\ar[dlll] \\
           &                &         &\mathcal{A}_0         &                 &                & \\
           &                &         & \Downarrow a_{\gamma, \alpha, \beta}    & & & \\
           &                &         &T(SR)\ar[dr]\ar[dl] &                 &                & \\
           &                &T\ar[dl] &\stackrel{\epsilon}{\Longleftarrow}     &SR\ar[dr]\ar[dl] &                & \\
           &T\ar[dr]\ar[dl] &         &S\ar[dr]\ar[dl]       &\stackrel{\delta}{\Longleftarrow} &R\ar[dr]\ar[dl] & \\
Z\ar[drrr] & \stackrel{\beta}{\Longleftarrow} &Y\ar[dr] & \stackrel{\alpha}{\Longleftarrow} &X\ar[dl]         &\stackrel{\gamma}{\Longleftarrow} &W\ar[dlll] \\
           &                &         &\mathcal{A}_0         &                 &                & \\
}\]
which does the following. First, we can describe $(TS)R$ as the groupoid with objects $[(t,s,f),r,g)\mid r_X\stackrel{g}{\rightarrow}s_X, s_Y\stackrel{f}{\rightarrow}t_Y]$, and $T(SR)$ as the groupoid with objects $[(t,(s,r,g),f)\mid r_X\stackrel{g}{\rightarrow}s_X, s_Y\stackrel{f}{\rightarrow}t_Y]$. Then, $a_{\gamma, \alpha, \beta}:(TS)R\rightarrow T(SR)$ is the functor which simply rearranges and re-parenthesizes the quintuple.

\item Left and right unitor: For a pair of objects $A,B$ the left and right unitors $l$ and $r$ are given by $2$-isomorphisms:
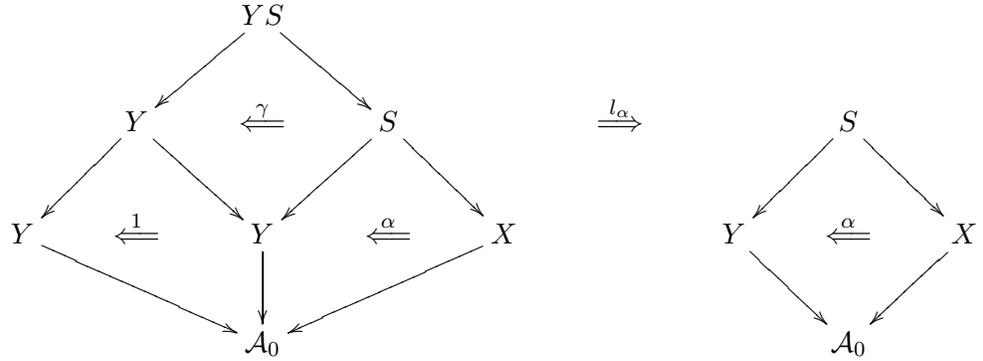
\[\xymatrix{
         &                                  & YS\ar[dl]\ar[dr]                 &                                   & & & & &  \\
         &Y\ar[dl]\ar[dr]                   &\stackrel{\gamma}{\Longleftarrow} & S\ar[dl]\ar[dr] & &\stackrel{l_{\alpha}}{\Longrightarrow} &   & S\ar[dl]\ar[dr] &             \\
Y \ar[drr]&\stackrel{1}{\Longleftarrow}  &  Y\ar[d]                 &\stackrel{\alpha}{\Longleftarrow}& X\ar[dll]   & & Y\ar[dr]    &           \stackrel{\alpha}{\Longleftarrow}       & X\ar[dl] \\
         &              &\mathcal{A}_0 &              & & & &\mathcal{A}_0 &          \\
         }
\]
and
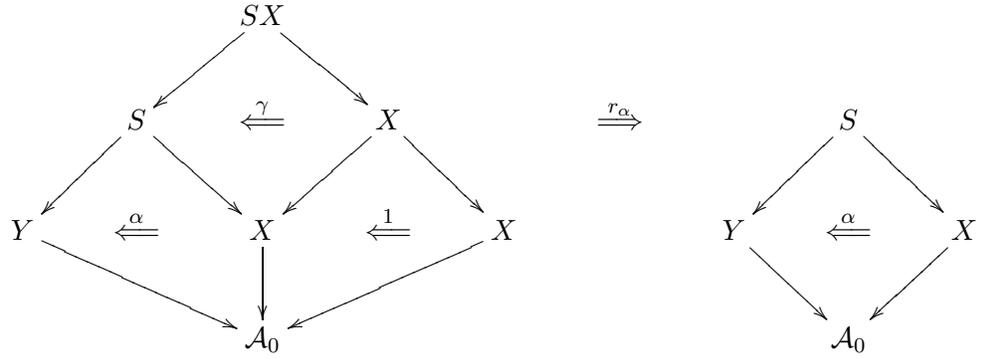
\[\xymatrix{
         &                                  & SX\ar[dl]\ar[dr]                 &                                   & & & & &  \\
         &S\ar[dl]\ar[dr]                   &\stackrel{\gamma}{\Longleftarrow} & X\ar[dl]\ar[dr] & &\stackrel{r_{\alpha}}{\Longrightarrow} &   & S\ar[dl]\ar[dr] &             \\
Y \ar[drr]&\stackrel{\alpha}{\Longleftarrow}  &  X\ar[d]                 &\stackrel{1}{\Longleftarrow}& X\ar[dll]   & & Y\ar[dr]    &           \stackrel{\alpha}{\Longleftarrow}       & X\ar[dl] \\
         &              &\mathcal{A}_0 &              & & & &\mathcal{A}_0 &          \\
         }
         \]
which can be described as follows. First, we describe $YS$ as the groupoid with objects $[(y,s,f)\mid y\stackrel{f}{\rightarrow} s_Y]$. Then $l_{\alpha}:YS\rightarrow S$ is the functor which simply maps $(y,s,f)$ to $s$. Similarly, $r_{\alpha}:SX\to S$ is the functor which map $(x,s,g)$ to $s$.
\end{itemize}
This data must satisfy the following identities:
\begin{itemize}
\item The pentagon identity for the associator:
Each vertex of the pentagon diagram is comprised of the composition of four spans $T, S, R, Q$ in different ways. 
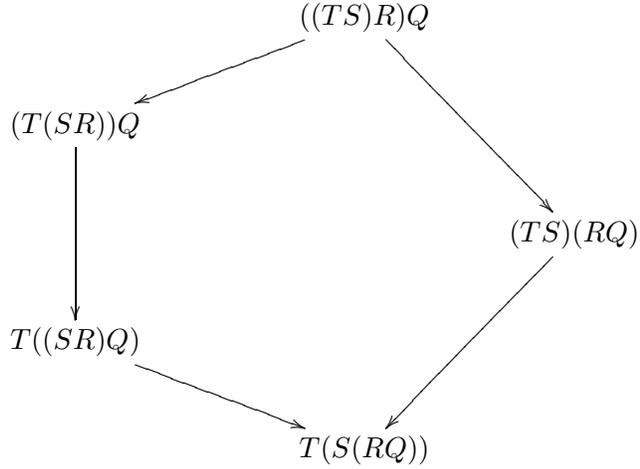
\[\xymatrix{
                 &      &((TS)R)Q\ar[dll]\ar[ddr]       &        \\
(T(SR))Q\ar[dd]  &      &                               &       \\  
                 &      &                               & (TS)(RQ)\ar[ddl] \\
T((SR)Q)\ar[drr] &      &                               & \\
                 &      & T(S(RQ))  & \\
}\]
We will show the process of describing one of the vertices in detail. For example, the vertex $T(S(RQ))$ is given by the following sequence of pullbacks:

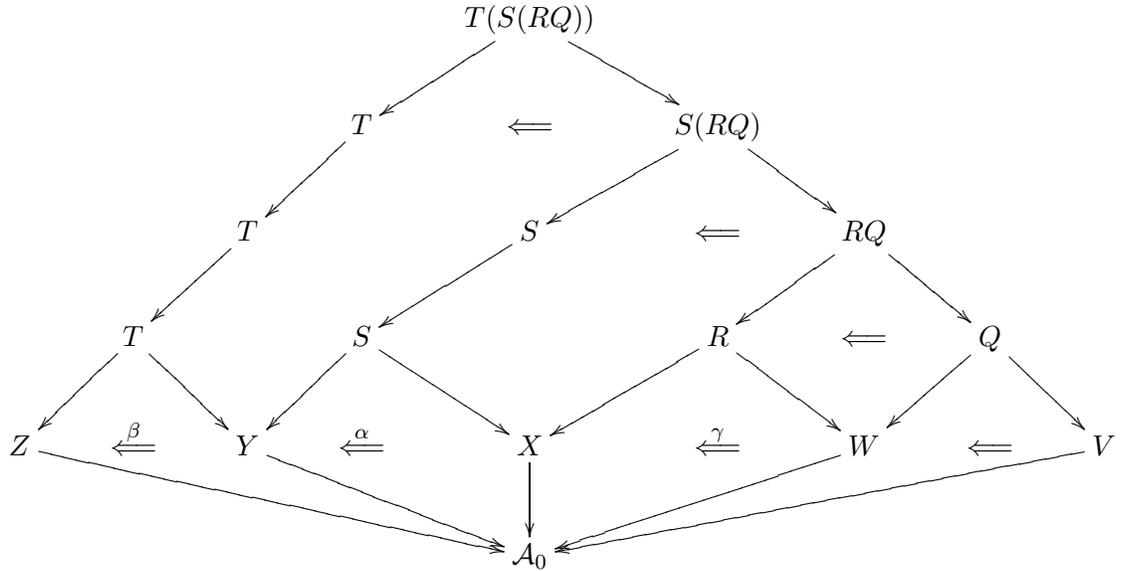
\[\xymatrix{
           &                &         &                &T(S(RQ))\ar[dl]\ar[dr] & & & & \\
           &                &         &T\ar[dl] &\stackrel{}{\Longleftarrow}                 &S(RQ)\ar[dl]\ar[dr]       & & & \\
           &                &T\ar[dl] &   &S\ar[dl] &\stackrel{}{\Longleftarrow}                &RQ\ar[dl]\ar[dr] & & \\
           &T\ar[dl]\ar[dr] &         &S\ar[dl]\ar[dr]       & &R\ar[dl]\ar[dr] &\stackrel{}{\Longleftarrow} &Q\ar[dl]\ar[dr] & \\
Z\ar[drrrr] & \stackrel{\beta}{\Longleftarrow} &Y\ar[drr] & \stackrel{\alpha}{\Longleftarrow} &X\ar[d]         &\stackrel{\gamma}{\Longleftarrow} &W\ar[dll] &\stackrel{}{\Longleftarrow} & V\ar[dllll]\\
           &                &         &                    &\mathcal{A}_0    &                & & & \\
}
\]
As a groupoid, we can describe the final result $T(S(RQ))$ in steps. First $T(S(RQ))=[(t,a,f)\mid a_Y\stackrel{f}{\rightarrow} t_Y]$ where $a\in S(RQ)$. we then describe $S(RQ)=[(s,b,g)\mid b_X\stackrel{g}{\rightarrow} s_X]$ where $b\in RQ$. Finally, $RQ=[(r,q,h)\mid q_W\stackrel{h}{\rightarrow} r_W]$. We can then work backwards to see that $T(S(RQ))$ can be rewritten as the groupoid $[(t,(s,(r,q,h),g),h)\mid q_W\stackrel{h}{\rightarrow}r_W, r_X\stackrel{g}{\rightarrow}s_X, s_Y\stackrel{f}{\rightarrow} t_Y]$.

We can also produce similar descriptions of the remaining vertices in the pentagon to get:

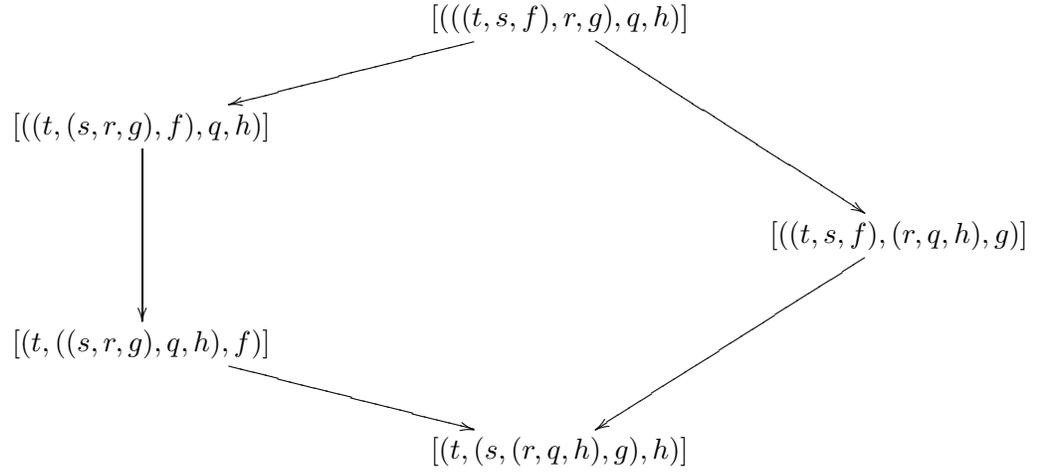
\[\xymatrix{
                 &      &[(((t,s,f),r,g),q,h)]\ar[dll]\ar[ddr]       &        \\
[((t,(s,r,g),f),q,h)]\ar[dd]  &      &                               &       \\  
                 &      &                               & [((t,s,f),(r,q,h),g)]\ar[ddl] \\
[(t,((s,r,g),q,h),f)]\ar[drr] &      &                               & \\
                 &      & [(t,(s,(r,q,h),g),h)]  & \\
}\]
which clearly commutes by simply rearranging and reparethesizing the tuples in each groupoid.

\item The Unitor Identity:

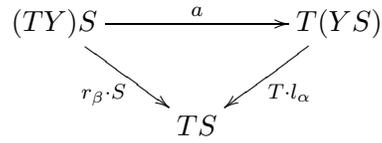
\[\xymatrix{
(TY)S\ar[rr]^a\ar[dr]_{r_\beta\cdot S} &    & T(YS)\ar[dl]^{T\cdot l_\alpha} \\
                                    & TS & \\
}\]
By a calculation similar the what we did for the associator pentagon identity, we can simplify each groupoid to the following:
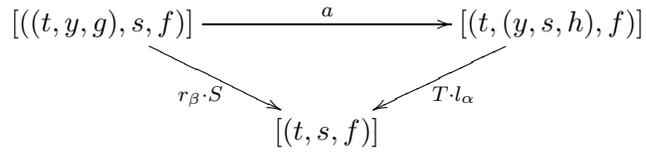
\[\xymatrix{
[((t,y,g),s,f)]\ar[rr]^a\ar[dr]_{r_\beta\cdot S} &    & [(t,(y,s,h),f)]\ar[dl]^{T\cdot l_\alpha} \\
                                    & [(t,s,f)] & \\
}\]
which clearly commutes.
\end{itemize}
\end{definition}

\section{The Monoidal Structure}\label{monbicat}
The bicategory $\Span(\Gpd\downarrow\mathcal{A}_0)$ carries a monoidal structure. Given two groupoids over $\mathcal{A}_0$:
\[\xymatrix{
X\ar[d]^f & & Y\ar[d]^g\\
\mathcal{A}_0 & & \mathcal{A}_0\\
}\]
we can construct the tensor product of these objects as the composite of $f\times g$ with direct sum of objects in $\mathcal{A}_0$, i.e.:
\[\xymatrix{
X \times Y\ar[d]^{f\times g}\\
\mathcal{A}_0\times \mathcal{A}_0\ar[d]^\oplus\\
\mathcal{A}_0\\
}\]
which makes $X\times Y$ a groupoid over $\mathcal{A}_0$.
Ww will now show that we have a monoidal bicategory as defined in the appendix \ref{monoidalDEF}. 
\begin{theorem}
The bicategory $\Span(\Gpd\downarrow\mathcal{A}_0)$ is a monoidal bicategory as follows:
\begin{itemize}
\item The tensor product $\otimes$ given by:
\[\xymatrix{
X\ar[dd]^f     &     & Y\ar[dd]^g     &   &X\times Y\ar[dd]^{\oplus\circ(f\times g)}\\
              &\ten &               &:= & \\
\mathcal{A}_0 &     & \mathcal{A}_0 &   &\mathcal{A}_0\\
}\]
\item The monoidal unit object $I$ given by the terminal groupoid $1$ equipped with the functor to the terminal object in $\mathcal{A}$.
\item The associator pseudonatural isomorphism $a$ given by the span:
\[\xymatrix{
 &(XY)Z\ar[dr]^{\id}\ar[dl]_{\hat{a}} & \\
X(YZ)\ar[dr] & \stackrel{\Leftarrow}{\alpha} & (XY)Z\ar[dl] \\
             & \mathcal{A}_0 & \\
}\]
where $ \hat{a}\maps (XY)Z \to X(YZ)$ is the associator for the monoidal category $\Gpd$, and $\alpha$ is given by the associator for the direct sum in $\mathcal{A}$.
\item The unitor pseudonatural isomorphism $l$ and $r$ are given by the spans:
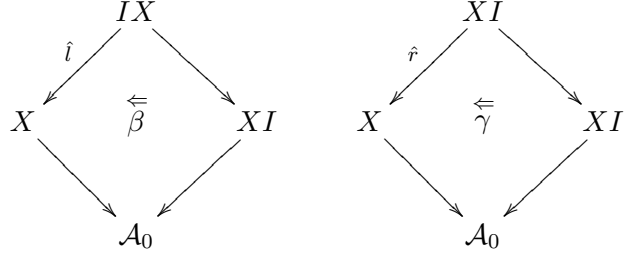
\[\xymatrix{
 & IX\ar[dl]_{\hat{l}}\ar[dr]& & &XI\ar[dl]_{\hat{r}}\ar[dr]&  \\
X\ar[dr] &\stackrel{\Leftarrow}{\beta} & XI\ar[dl]& X\ar[dr] & \stackrel{\Leftarrow}{\gamma} & XI\ar[dl] \\
& \mathcal{A}_0 & & &\mathcal{A}_0 & \\
}\]
where $\hat{l}$ and $\hat{r}$ are the unitors for the tensor product in $\Gpd$ and $\beta, \gamma$ are given by the unitors for the direct sum in $\mathcal{A}$.
\item the pentagonator $\pi$ is trivial; i.e. the pentagon equation holds on the nose. This is true since the tensor product in both $\Gpd$ and $\mathcal{A}$ are strictly associative (i.e. their respective pentagon equations hold on the nose).
\end{itemize}
\end{theorem}

\section{The Groupoid EXT}\label{EXT}
In this section we will study a groupoid whose cardinality is related to the coefficient of the braiding isomorphisms in $\Vect^K$. This groupoid will become the building block of the `braiding spans' in our monoidal bicategory $\Span(\Gpd\downarrow\mathcal{A}_0)$. This will allow us to describe a braiding structure on this monoidal bicategory, which is a categorification of the braiding on the monoidal category $\Vect^K$.\\

\begin{definition}
For a fixed pair of objects $M,N$ in the abelian category $\Rep(Q)$, define the groupoid $\EXT(M,N)$ to have:
\begin{itemize}
\item Objects - short exact sequences of the form: 
\[0\to N \to E \to M \to 0\]
\item Morphisms - a triple of isomorphisms $(\alpha, \beta, \gamma)$ such that the following diagram commutes:
\[
\xymatrix{
0 \ar[r] & 
N \ar[r]^{f} \ar[d]^{\alpha} & 
E \ar[r]^{g} \ar[d]^{\beta} & 
M \ar[r] \ar[d]^{\gamma} & 0 \\
0 \ar[r] & N \ar[r]^{f} & E^\prime \ar[r]^{g} & M \ar[r] & 0 \\
}
\]
\end{itemize}
\end{definition}

We start with a useful formula for $P^E_{MN}$ (the cardinality of the set $\mathcal{P}^E_{MN}$).
\begin{proposition}\label{Riedtmann}(Riedmann's Formula) For fixed $M$, $N$, and $E$;
\[P^E_{MN}=\frac{|\Ext^1(M,N)_E||\Aut(E)|}{|\Hom(M,N)|}\]
where $\Ext^1(M,N)_E$ is the set of all classes of extensions of $M$ by $N$ which are isomorphic to $E$.
\end{proposition}

\begin{proof}
This proof comes from \cite{Hubery}, but we will rewrite it here because it illustrates some important techniques for counting extensions.
Let $\alpha\in \Aut(0\to N \stackrel{f}{\rightarrow} E \stackrel{g}{\rightarrow} M \to 0)$. By the definition of morphism in the groupoid $\EXT(M,N)$, we get that this automorphism is equivalent to an automorphism of $E$ such that the following diagram commutes:
\[\xymatrix{
& & E \ar[dr]^g\ar[dd]^\alpha& & \\
0\ar[r]& N\ar[ur]^f\ar[dr]_{f} & & M\ar[r]& 0 \\
& & E\ar[ur]_{g} && \\
}\] 
First we note that since $\alpha f=f$, we have that $(\alpha -1)f=0$. From the long exact sequence for $\Hom(-,E)$, i.e.;
\[0\to\Hom(M,E)\stackrel{\tilde{g}}{\rightarrow}\Hom(E,E) \stackrel{\tilde{f}}{\rightarrow}\Hom(N,E)\to\cdots\]
we get that $(\alpha -1)\in\ker(\tilde{g})=\im(\tilde{f})$, so there is a unique $\phi\in\Hom(M,E)$ such that $\phi g = \alpha-1.$  Also, from the original diagram, we get that $g\alpha=g$, so $g(\alpha -1 )=0$. By substitution, this gives that $0=g\phi g$, and since $g$ is onto $g\phi=0$. We now consider another long exact sequence, namely the one for $\Hom(M,-)$:
\[0\to\Hom(M,N)\stackrel{\hat{f}}{\rightarrow}\Hom(M,E) \stackrel{\hat{g}}{\rightarrow}\Hom(M,M)\to\cdots\]
By a similar argument to above, the fact that $g\phi=0$ tells us there exist a unique $\delta\in\Hom(M,N)$ such that $\phi=f\delta$. Thus we have an injective map $\Aut(0\to N \stackrel{f}{\rightarrow} E \stackrel{g}{\rightarrow} M \to 0)\longrightarrow \Hom(M,N)$ which take $\alpha\in\Aut(0\to N \stackrel{f}{\rightarrow} E \stackrel{g}{\rightarrow} M \to 0)$ to the unique morphism $\delta\in\Hom(M,N)$ such that $\alpha=1+f\delta g$.\\
To show this map is surjective, we note that for any $\delta\in\Hom(M,N)$, the map $f\delta g\in \Hom(E,E)$ satisfies $(f\delta g)^2=0$, so that $1+f\delta g\in \Aut(E)$. Therefore $\Aut(0\to N \stackrel{f}{\rightarrow} E \stackrel{g}{\rightarrow} M \to 0)\iso \Hom(M,N)$.
\end{proof}
We now make use of this formula to see that the cardinality of $\EXT(M,N)$ is related to our braiding coefficient $q^{-\langle\underline{m},\underline{n}\rangle}$.
\begin{lemma}
\[|\EXT(M,N)|=\frac{q^{-\langle\underline{m},\underline{n}\rangle}}{|\Aut(N)||\Aut(M)|}\].
\end{lemma}

\begin{proof}
First, we note that $\EXT(M,N)$ as described is equivalent to the weak quotient groupoid $\coprod_E (\mathcal{P}^E_{MN})//(\Aut(N)\times\Aut(E)\times\Aut(E))$, and so its groupoid cardinality is simply 
\[\sum_E \frac{P_{MN}^E}{|\Aut(N)||\Aut(E)||\Aut(M)|}\] 
but by Riedtmann's Formula \ref{Riedtmann}. we have that for each fixed $E$: 
\[P^E_{MN}=\frac{|\Ext^1(M,N)_E||\Aut(E)|}{|\Hom(M,N)|}.\]
Also note that the sum of $|\Ext^1(M,N)_E|$ over all values of $E$ yields the the entire set of extensions $|\Ext^1(M,N)|$.
Thus by substitution, the groupoid cardinality of the groupoid of all short exact sequences $0\to N\to E\to M\to 0$ with fixed $M$ and $N$ is precisely:
\[\begin{array}{rl}
\displaystyle{\frac{|\Ext^1(M,N)|}{|\Hom(M,N)|}\frac{1}{|\Aut(N)||\Aut(M)|}} & \displaystyle{= \frac{q^{\dim\Ext^1(M,N)}}{q^{\dim\Hom(M,N)}}\frac{1}{|\Aut(N)||\Aut(M)|}} \\
 & \\
                                 & \displaystyle{= \frac{q^{\dim\Ext^1(M,N)-\dim\Hom(M,N)}}{|\Aut(N)||\Aut(M)|}}\\
 & \\
                                 & \displaystyle{= \frac{q^{-\langle\underline{m},\underline{n}\rangle}}{|\Aut(N)||\Aut(M)|}} \\
\end{array}\]
\end{proof}

%
%

We would also like to see that $\EXT(-,-)$ is bilinear for appropriate sums in $\Rep(Q)$ and $\SES(Q)$.
\begin{proposition}\label{EXTbil} The functor $\EXT(-,-)$ is bilinear, i.e.;
\begin{enumerate}
\item For any three representations $M_1,M_2,N\in \Rep(Q)$ we have:
\[\EXT(M_1\oplus M_2,N)\simeq \EXT(M_1,N)\times\EXT(M_2,N)\]
\item For any three representations $M, N_1, N_2\in \Rep(Q)$ we have:
\[\EXT(M,N_1\oplus N_2)\simeq \EXT(M,N_1)\times \EXT(M,N_2)\]
\end{enumerate}
\end{proposition}
\begin{proof}
\begin{enumerate}
\item We start with an extension in $\EXT(M_1\oplus M_2,N)$, say 
\[0\to N\to E\to M_1\oplus M_2\to 0.\] 
We want to construct from this a pair of extensions in $\EXT(M_1,N)\times\EXT(M_2,N)$. So beginning with an extension:
\[\xymatrix{0 \ar[r] & N \ar[r]^f & E\ar[r]^{g\;\;\;\;\;\;} & M_1\oplus M_2\ar[r] & 0 \\}\]
we will construct an extension in $\EXT(M_1,N)$. we start by adding the canonical injection:
\[\xymatrix{
         &            &           & M_1\ar[d]^{i_1} & \\
0 \ar[r] & N \ar[r]^f & E\ar[r]^{g\;\;\;\;\;\;} & M_1\oplus M_2\ar[r] & 0 \\}
\]
and then forming the pullback of the right side of the diagram:
\[\xymatrix{
         &            & E_1\ar[d]^{\pi_e}\ar[r]^{g_1} & M_1\ar[d]^{i_1} & \\
0 \ar[r] & N \ar[ur]^f\ar[r]^f & E\ar[r]^{g\;\;\;\;\;\;}       & M_1\oplus M_2\ar[r] & 0 \\}
\]
We claim that the representation $E_1$ will form an extension of $M_1$ by $N$. First, we note that the map $N\stackrel{f}{\rightarrow} E$ can be extended to an injective map sending $n\mapsto (n,0)$, since any element in $\im f\in E$ maps to $0$ in $M_1\oplus M_2$. Thus we have a short exact sequence:
\[\xymatrix{
0 \ar[r] & N \ar[r]^f & E_1\ar[r]^{g_1\;\;} & M_1\ar[r] & 0 \\}\]
which gives an extension on $M_1$ by $N$. Similarly, we can construct and extension:
\[\xymatrix{
0 \ar[r] & N \ar[r]^f & E_2\ar[r]^{g_2\;\;} & M_2\ar[r] & 0 \\}\]
giving a functor:
\[F_1:\EXT(M_1\oplus M_2,N)\to \EXT(M_1,N)\times\EXT(M_2,N).\]

Now we go the other direction. Given a pair of extensions:
\[\xymatrix{
0 \ar[r] & N \ar[r]^{f_1} & E_1\ar[r]^{g_1\;\;} & M_1\ar[r] & 0 \\
0 \ar[r] & N \ar[r]^{f_2} & E_2\ar[r]^{g_2\;\;} & M_2\ar[r] & 0 \\}\]
We form the sequence:
\[\xymatrix{
0 \ar[r] & N \ar[r]^{(f_1,0)\;\;\;\;\;\;\;\;\;} & (E_1\oplus E_2)/I_N\ar[r]^{\;\;\;(g_1,g_2)} & M_1\oplus M_2\ar[r] & 0 \\}\]
where $I_N=\{(f_1(n),-f_2(n))| n\in N\}$. This sequence is short exact since the map $(f_1,0)=(0,f_2)$ map injectively to $\ker g_1$ and $\ker g_2$, respectively. This gives a functor:
\[G_1:\EXT(M_1,N)\times\EXT(M_2,N)\to \EXT(M_1\oplus M_2,N).\]

Now we need to show that these two functors form an equivalence of groupoids. 
Starting with a pair of extensions:
\[\xymatrix{
0 \ar[r] & N \ar[r]^{f_1} & E_1\ar[r]^{g_1\;\;} & M_1\ar[r] & 0 \\
0 \ar[r] & N \ar[r]^{f_2} & E_2\ar[r]^{g_2\;\;} & M_2\ar[r] & 0 \\}\]
we apply $G_1$ to get the extension:
\[\xymatrix{
0 \ar[r] & N \ar[r]^{(f_1,0)\;\;\;\;\;} & E_1\oplus E_2\ar[r]^{(g_1,g_2)\;\;\;\;\;} & M_1\oplus M_2\ar[r] & 0 .\\}\]
To apply $F_1$ we consider the two pullbacks:
\[\xymatrix{
         &            & E_1\ar[d]\ar[r]^{g_1} & M_1\ar[d]^{i_1} & \\
0 \ar[r] & N \ar[r]^{(f_1,0)\;\;}\ar[ru]^{f_1}\ar[rd]_{f_2} & E_1\oplus E_2\ar[r]^{(g_1,g_2)\;\;\;\;\;\;}       & M_1\oplus M_2\ar[r] & 0 \\
         &            & E_2\ar[u]\ar[r]^{g_2} & M_2\ar[u]_{i_2} & \\}
\]
And it is easy to check that we get back to the original extensions.

In the other direction, we start with an extension:
\[\xymatrix{
0 \ar[r] & N \ar[r]^f & E\ar[r]^{g\;\;\;\;\;\;} & M_1\oplus M_2\ar[r] & 0 \\}\]
and apply $F_1$ to get the pair:
\[\xymatrix{
0 \ar[r] & N \ar[r]^{f} & E_1\ar[r]^{g_1\;\;} & M_1\ar[r] & 0 \\
0 \ar[r] & N \ar[r]^{f} & E_2\ar[r]^{g_2\;\;} & M_2\ar[r] & 0 \\}\]
we then put them back together via $G_1$. to get the extension:
\[\xymatrix{
0 \ar[r] & N \ar[r]^{(f,0)\;\;\;\;\;} & E_1\oplus E_2\ar[r]^{(g_1,g_2)\;\;\;\;\;} & M_1\oplus M_2\ar[r] & 0 .\\}\]
We can then produce a morphism between this result and the original sequence as follows:
\[\xymatrix{
0 \ar[r] & N \ar[d]_{id}\ar[r]^{(f_1,0)\;\;\;\;\;} & E_1\oplus E_2\ar[d]_{\pi_1+\pi_2}\ar[r]^{(g_1,g_2)\;\;\;\;\;} & M_1\oplus M_2\ar[d]_{id}\ar[r] & 0\\
0 \ar[r] & N \ar[r]^f & E\ar[r]^{g\;\;\;\;\;\;} & M_1\oplus M_2\ar[r] & 0 \\
}\]

\item for the second coordinate, we will use a similar technique for constructing our functors between the two groupoids. First we describe the functor
\[F_2\maps\EXT(M,N_1\oplus N_2)\to \EXT(M,N_1)\times \EXT(M,N_2)\]
as follows. Starting with an extension:
\[0\to N_1\oplus N_2\stackrel{f}{\to} E \stackrel{g}{\to}M\to 0\]
we can split this into the two extensions
\[0\to N_1\stackrel{f_1}{\to}E/\im N_2\stackrel{g_1}{\to} M\to 0\]
\[0\to N_2\stackrel{f_2}{\to}E/\im N_1\stackrel{g_2}{\to} M\to 0\] 
where $f_1, f_2, g_1$, and $g_2$ are the appropriate restrictions of $f$ and $g$, respectively. $g_1$ and $g_2$ are surjective, since $\im N_1, \im N_2\subset \im(N_1\oplus N_2)=\ker M$. \\
For the other direction we need a functor:
\[G_2\maps\EXT(M,N_1)\times \EXT(M,N_2)\to \EXT(M,N_1\oplus N_2)\] 
Starting with two extensions
\[0\to N_1\stackrel{f_1}{\to}E_1\stackrel{g_1}{\to} M\to 0\]
\[0\to N_2\stackrel{f_2}{\to}E_2\stackrel{g_2}{\to} M\to 0\] 
we put them together by direct sum together with a modification to the extension $E_1\oplus E_2$. Specifically, this is the extension:
\[0\to N_1\oplus N_2\stackrel{(f_1,f_2)}{\to} (E_1\oplus E_2)/J \stackrel{g_1+g_2}{\to}M\to 0\]
where $J=\{(e_1,e_2)\in E_1\oplus E_2-\im(N_1\oplus N_2)\mid g_1(e_1)+g_2(e_2)=0\}$. 
We can then check, just like in part $1$, that these functors form an equivalence of groupoids.
\end{enumerate}
\end{proof}

\section{The Braiding Span}\label{braidmonbicat}
In this section we will use the groupoids $\EXT(M,N)$ to describe a span which will serve as a braiding for the monoidal bicategory $\Span(\Gpd\downarrow\mathcal{A}_0)$.

If we define the maps $\pi_q,\pi_s\maps \SES(\mathcal{A})\to \mathcal{A}_0$ to be the quotient projection and the sub-object projection from the short exact sequence $0\to N\to E\to M\to 0$, then the span:
\[\xymatrix{ & \SES(\mathcal{A})\ar[dr]^{\pi_s\times\pi_q}\ar[dl]_{\pi_q\times\pi_s}& \\
\mathcal{A}_0\times \mathcal{A}_0\ar[dr]_\oplus & \stackrel{\alpha}{\Longleftarrow} & \mathcal{A}_0\times\mathcal{A}_0\ar[dl]^\oplus\\
 &\mathcal{A}_0 & \\
}\]
will give a braiding for the above category as follows. Given two groupoids over $\mathcal{A}_0$, say $f\maps X\to \mathcal{A}_0$ and $g\maps Y\to \mathcal{A}_0$, the braiding span gives a braid morphisms from $X\times Y$ to $Y\times X$ by the weak pullback of the diagram:
\[\xymatrix{
X\times Y\ar[dr]_{f\times g}&  & \SES(\mathcal{A})\ar[dr]^{\pi_s\times\pi_q}\ar[dl]_{\pi_q\times\pi_s}& & Y\times X\ar[dl]^{g\times f} \\
&\mathcal{A}_0\times \mathcal{A}_0\ar[dr]_\oplus &\stackrel{\alpha}{\Longleftarrow} & \mathcal{A}_0\times\mathcal{A}_0\ar[dl]^\oplus &\\
  & & \mathcal{A}_0 & & \\
}\]
it is easy enough to see that the pullbacks on the different sides of this diagram give equivalent objects. This object is precisely the groupoid 
\[B_{X,Y}=[(x,y,E)\mid 0\to \u{y}\to E\to \u{x}\to 0].\] 

We do however need both pullbacks, as each one give the map to the corresponding leg of the span. We can complete this span by taking a final weak pullback, which will simply give the same groupoid at the top. The result is the span:
\[\xymatrix{ & B_{XY}\ar[dl]_{\pi_x\times \pi_y}\ar[dr]^{\pi_y\times \pi_x} & \\
X\times Y\ar[dr]_{\oplus\circ (f\times g)} &\stackrel{}{\Longleftarrow} &   Y\times X\ar[dl]^{\oplus\circ (g\times f)}\\
 & \mathcal{A}_0&\\
}\] 

Next we describe how the pullbacks from the braiding span in our bicategory $\Span(\Gpd\downarrow\mathcal{A}_0)$ are related to the groupoids $\EXT(M,N)$.
\begin{proposition}\label{BsimEXT}
If $B_{X,Y}$ is the groupoid described by the weak pullback of the diagram:
\[\xymatrix{
X\times Y\ar[dr]_{f\times g}&  & \SES(\mathcal{A})\ar[dr]^{\pi_s\times\pi_q}\ar[dl]_{\pi_q\times\pi_s}& & Y\times X\ar[dl]^{g\times f} \\
&\mathcal{A}_0\times \mathcal{A}_0\ar[dr]_\oplus &\stackrel{\alpha}{\Longleftarrow} & \mathcal{A}_0\times\mathcal{A}_0\ar[dl]^\oplus &\\
  & & \mathcal{A}_0 & & \\
}\]
then $\displaystyle{B_{X,Y}\simeq \coprod_{[(x,y)]} \EXT(\u{x},\u{y})}$
\end{proposition}
\begin{proof}
We start by describing the groupoid $B_{X,Y}$ coming from the weak pullback. $B_{X,Y}$ is a result of the weak pullback of following diagram: 
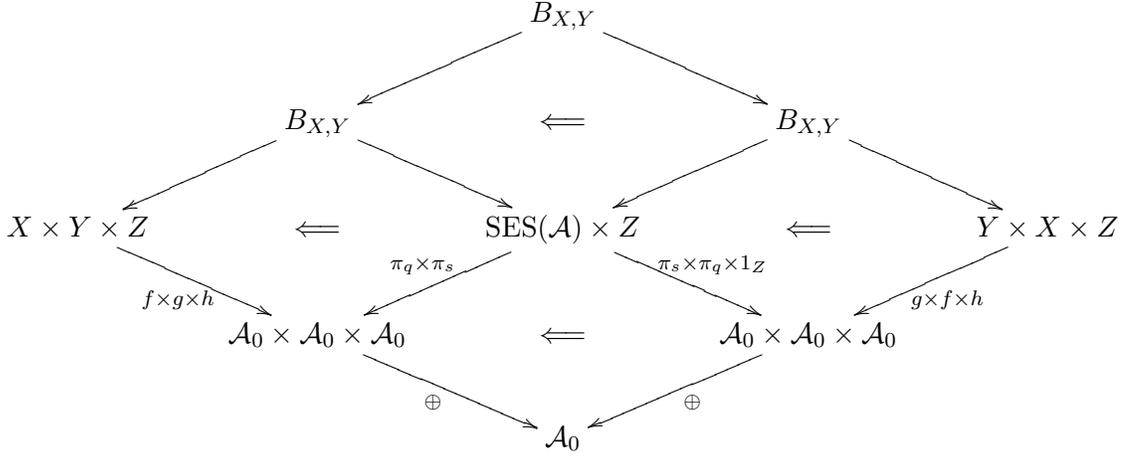
\[\xymatrix{
& &B_{X,Y}\ar[dr]\ar[dl] & & \\
&B_{X,Y}\ar[dr]\ar[dl] &\stackrel{}{\Longleftarrow} &B_{X,Y}\ar[dr]\ar[dl] & \\
X\times Y \times Z\ar[dr]_{f\times g\times h}& \stackrel{}{\Longleftarrow} & \SES(\mathcal{A})\times Z\ar[dr]^{\pi_s\times\pi_q\times 1_Z}\ar[dl]_{\pi_q\times\pi_s}&\stackrel{}{\Longleftarrow} & Y\times X\times Z\ar[dl]^{g\times f\times h} \\
&\mathcal{A}_0\times \mathcal{A}_0\times \mathcal{A}_0\ar[dr]_\oplus &\stackrel{}{\Longleftarrow} & \mathcal{A}_0\times\mathcal{A}_0\times \mathcal{A}_0\ar[dl]^\oplus &\\
  & & \mathcal{A}_0 & & \\
}\]
The apex of this span is the groupoid 
\[B_{X,Y}=\left[(x,y,E)|0\to \u{y}\to E \to \u{x}\to 0\;\;\emph{\rm is exact}\right]\] 
where the morphisms are isomorphisms of short exact sequences. We then define the functors $F_{X,Y}$ and $G_{X,Y}$ as follows. Let $\displaystyle{F_{X,Y}\maps B_{X,Y}\to \coprod_{[(x,y)]} \EXT(\u{x},\u{y})}$ be the functor that takes $(x,y,E)$ to the short exact sequence $0\to \u{y}\to E \to \u{x}\to 0$. Also, let $\displaystyle{G_{X,Y}\maps \coprod_{[(x,y)]} \EXT(\u{x},\u{y})\to B_{X,Y}}$ takes the short exact sequence $0\to \u{y}\to E \to \u{x}\to 0$ to $(x,y,E)$. It is easy to check these functor form an equivalence for these groupoids.
\end{proof}
Unlike the monoidal structure on $\Span(\Gpd\downarrow\mathcal{A}_0)$, the braiding will not be trivial. We first need to define the hexagonator's $R$ and $S$. we start by considering the hexagon diagrams for each of these $2$-morphisms.
Starting with the hexagon identity:
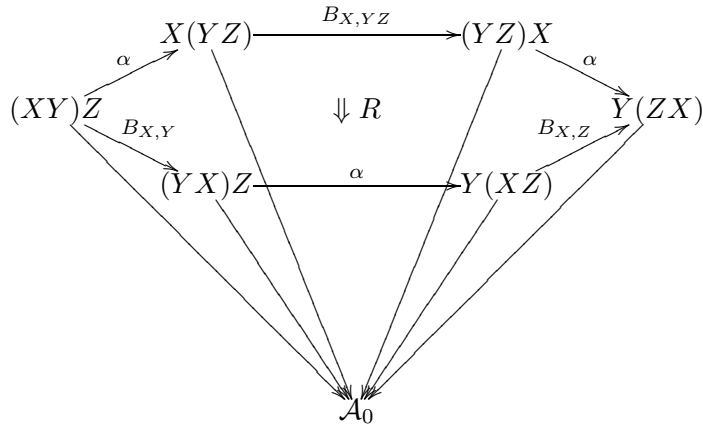
\[\xy
(0,-40)*{\mathcal{A}_0}="A";
(-20,-10)*{(YX)Z}="LB";
(20,-10)*{Y(XZ)}="RB";
(-40,0)*{(XY)Z}="L";
(40,0)*{Y(ZX)}="R";
(-20,10)*{X(YZ)}="LT";
(20,10)*{(YZ)X}="RT";
{\ar^{\alpha} "L";"LT"};
{\ar^{B_{X,Y}} "L";"LB"};
{\ar^{\alpha} "RT";"R"};
{\ar^{B_{X,Z}} "RB";"R"};
{\ar^{\alpha} "LB";"RB"};
{\ar^{B_{X,YZ}} "LT";"RT"};
{\ar^{} "L";"A"};
{\ar^{} "LT";"A"};
{\ar^{} "LB";"A"};
{\ar^{} "R";"A"};
{\ar^{} "RT";"A"};
{\ar^{} "RB";"A"};
(0,0)*{\Downarrow R}
\endxy\]
where $\alpha$ is the associator and the $B$'s represent the braiding spans, which will describe in detail. As noted before, $B_{X,Y}$ is a result of the weak pullback of following diagram: 
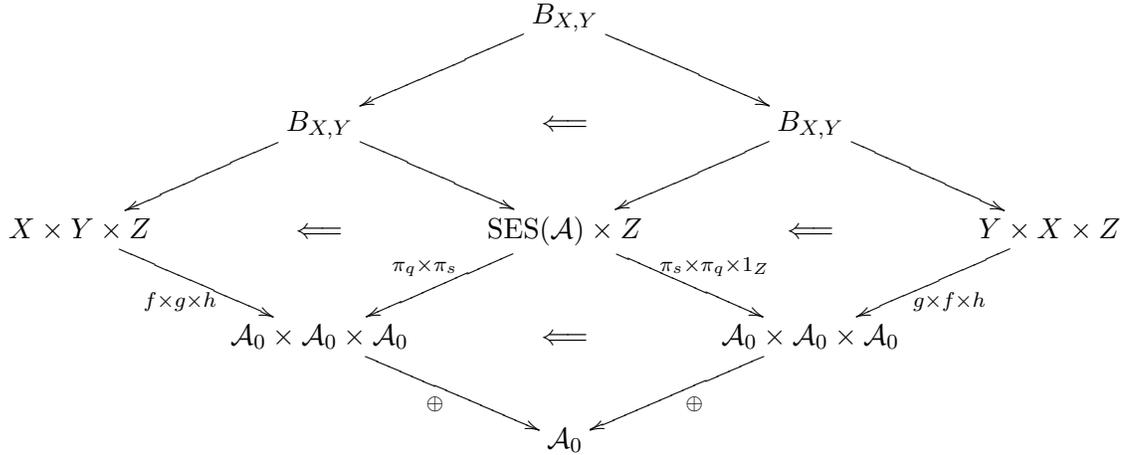
\[\xymatrix{
& &B_{X,Y}\ar[dr]\ar[dl] & & \\
&B_{X,Y}\ar[dr]\ar[dl] &\stackrel{}{\Longleftarrow} &B_{X,Y}\ar[dr]\ar[dl] & \\
X\times Y \times Z\ar[dr]_{f\times g\times h}& \stackrel{}{\Longleftarrow} & \SES(\mathcal{A})\times Z\ar[dr]^{\pi_s\times\pi_q\times 1_Z}\ar[dl]_{\pi_q\times\pi_s}&\stackrel{}{\Longleftarrow} & Y\times X\times Z\ar[dl]^{g\times f\times h} \\
&\mathcal{A}_0\times \mathcal{A}_0\times \mathcal{A}_0\ar[dr]_\oplus &\stackrel{}{\Longleftarrow} & \mathcal{A}_0\times\mathcal{A}_0\times \mathcal{A}_0\ar[dl]^\oplus &\\
  & & \mathcal{A}_0 & & \\
}\]
Intuitively, this can be thought of as the groupoid $B_{X,Y}=\left[(x,y,E)|0\to \u{y}\to E \to \u{x}\to 0\;\;\emph{\rm is exact}\right]$ where the morphisms are isomorphisms of short exact sequences. Following this idea, we need to then compose the span:
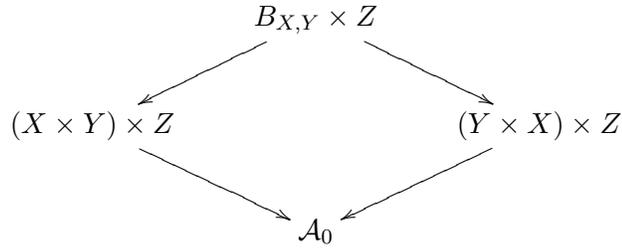
\[\xymatrix{ & B_{X,Y}\times Z\ar[dl]\ar[dr] & \\
(X\times Y)\times Z\ar[dr] & &   (Y\times X)\times Z\ar[dl]\\
 & \mathcal{A}_0&\\
}\]
with the span
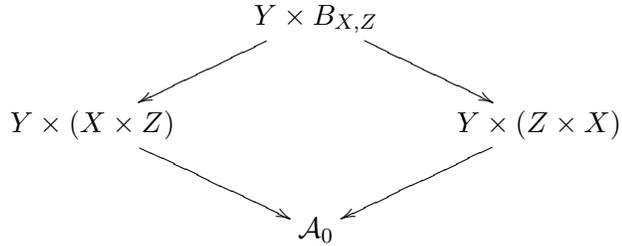
\[\xymatrix{ & Y\times B_{X,Z}\ar[dl]\ar[dr] & \\
Y\times (X\times Z)\ar[dr] & &   Y\times (Z\times X)\ar[dl]\\
 & \mathcal{A}_0&\\
}\]
and compare it to the span:
\[\xymatrix{ & B_{X,Y\times Z}\ar[dl]\ar[dr] & \\
(Y\times Z)\times X\ar[dr] & &   X\times (Y\times Z)\ar[dl]\\
 & \mathcal{A}_0&\\
}\]
By our intuitive description we have that:
\[B_{X,Y\times Z}=[(x,y,z,E)|0\to \u{y}\oplus \u{z}\to E \to \u{x}\to 0]\]
and the composition of the first two spans looks like:
\[(Y\times B_{X,Z})(B_{X,Y}\times Z)=[(x,y,z,E_1,E_2)|0\to \u{y}\to E_1 \to \u{x}\to 0, 0\to \u{z}\to E_2 \to \u{x}\to 0] \]
By Proposition \ref{BsimEXT} (or a similar argument) we get:
\[B_{X,Y\times Z}\simeq\coprod_{[(x,y,z)]} \EXT(\u{x}, \u{y}\oplus \u{z}).\]
Also, with a little extra work, we get:
\[(Y\times B_{X,Z})(B_{X,Y}\times Z)\simeq \coprod_{[(x,y,z)]} \EXT(\u{x},\u{y})\otimes\EXT(\u{x},\u{z}).\]
Thus, by Proposition \ref{EXTbil};
\[\coprod_{[(x,y,z)]} \EXT(\u{x}, \u{y}\oplus \u{z})\simeq\coprod_{[(x,y,z)]} \EXT(\u{x},\u{y})\otimes\EXT(\u{x},\u{z})\]
therefore $B_{X,Y\times Z}\simeq B_{X,Y\times Z}$ and the hexagon commutes up to the equivalence:
\[R:= (G_{X,Y}\times G_{X,Z})\circ F_2 \circ F_{X,Y\times Z}\]
By following this composition we can define $R$ explicitly. Specifically, we see that 
\[R\maps (x,y,z,E)\to (x,y,z,E/\underline{z},E/\underline{y})\]
Similarly, consider the hexagon diagram for $S$:

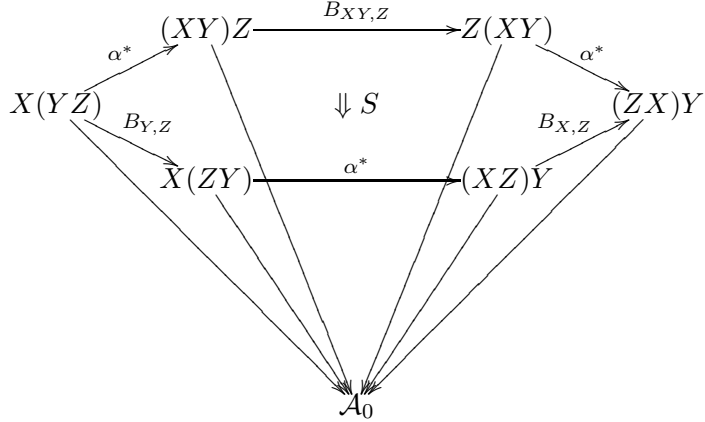
\[\xy
(0,-40)*{\mathcal{A}_0}="A";
(-20,-10)*{X(ZY)}="LB";
(20,-10)*{(XZ)Y}="RB";
(-40,0)*{X(YZ)}="L";
(40,0)*{(ZX)Y}="R";
(-20,10)*{(XY)Z}="LT";
(20,10)*{Z(XY)}="RT";
{\ar^{\alpha^*} "L";"LT"};
{\ar^{B_{Y,Z}} "L";"LB"};
{\ar^{\alpha^*} "RT";"R"};
{\ar^{B_{X,Z}} "RB";"R"};
{\ar^{\alpha^*} "LB";"RB"};
{\ar^{B_{XY,Z}} "LT";"RT"};
{\ar^{} "L";"A"};
{\ar^{} "LT";"A"};
{\ar^{} "LB";"A"};
{\ar^{} "R";"A"};
{\ar^{} "RT";"A"};
{\ar^{} "RB";"A"};
(0,0)*{\Downarrow S}
\endxy\]
which will commute up to the equivalence:
\[S:=(G_{X,Z}\times G_{Y,Z})\circ F_1\circ F_{X\times Y,Z}\]
Just like before, we can also define $S$ explicitly:
\[S\maps (x,y,z,E)\to (x,y,z,g^{-1}(\underline{x}),g^{-1}(\underline{y}))\]
We are now ready to state the main theorem of this section. We will be using the definition of braided monoidal bicategory \ref{braidedDEF} from the appendix.
\begin{theorem}
The monoidal bicategory $\Span(\Gpd\downarrow\mathcal{A}_0)$ is a braided monoidal bicategory where:
\begin{itemize}
\item The adjoint equivalence $b$ is given by the span:
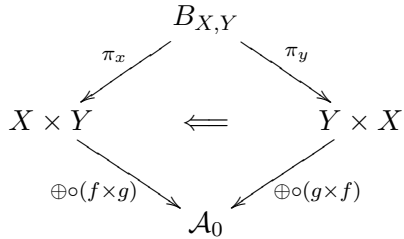
\[\xymatrix{ & B_{X,Y}\ar[dl]_{\pi_x}\ar[dr]^{\pi_y} & \\
X\times Y\ar[dr]_{\oplus\circ (f\times g)} &\stackrel{}{\Longleftarrow} &   Y\times X\ar[dl]^{\oplus\circ (g\times f)}\\
 & \mathcal{A}_0&\\
}\]
\item The invertible modifications $R$ and $S$ are given above.
\end{itemize}
\end{theorem}
\begin{proof}
The work in this proof will be to check the four coherence laws in the definition of braided monoidal bicategory \ref{braidedDEF}.
Since the tensor product for $\Span(\Gpd\downarrow\mathcal{A}_0)$ in Section \ref{monbicat} has an associated that satisfies the pentagon identity on the nose, we will be able to simplify our work. When we compare the groupoids $(AB)C$ with $A(BC)$, we see that the only real difference is in the way they arte parenthesized. Specifically, and object in $(AB)C$ is a triple $((a,b),c)$, while an object in $A(BC)$ is a triple $(a,(b,c))$. So these are similar enough that even though the associator is nontrivial, we will pretend it is and write $(a,b,c)$ for an object of $ABC$, and remove any pentagon diagrams from our coherence laws. Without the need for the associator pentagon in any of the diagrams, we can reduce each diagram to some simplified polytopes. We will draw each simplified polytope before checking them.
\begin{itemize}
\item The first simplified polytope governs ways to shuffle 1 object through 3 objects. There are four ways to shuffle 1 object into 3 objects, so this polytope will be a tetrahedron, with the tensor product of four objects on each corner. For compactness we will write these objects based on the order or the four objects without the $\otimes$, which gives the diagram:
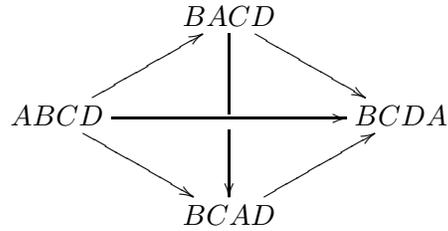
\[\xymatrix{
       &BACD\ar[dr]\ar @{-}[d] & \\
ABCD\ar[ur]\ar[rr]\ar[dr] &{ }\ar[d] & BCDA\\
& BCAD\ar[ur] & \\
}\]
where each side of the tetrahedron is filled in with the appropriate $R$. Verifying this diagram comes down to checking that the two composites of $R$'s going from the shortest path to the longest path are equal. To do this, we need to describe the groupoids at the apex of the four paths around the tetrahedron, then describe what the map between these span coming from $R$ does. The shortest path:
\[\xymatrix{ABCD\ar[rr]& & BCDA}\]
is just the braiding span 
\[B_{A,BCD}=[(a,b,c,d,E\mid\]
\[ 0\to \u{b}\oplus\u{c}\oplus\u{d}\to E\to \u{a}\to 0].\]
The path:
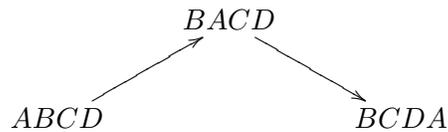
\[\xymatrix{
       &BACD\ar[dr] & \\
ABCD\ar[ur] & & BCDA\\
}\]
gives the composite of the two braiding spans: 
\[(B_{A,B}\times CD)(B\times B_{A,CD})=[(a,b,c,d,E_1,E_2)\mid\]
\[0\to\u{b}\to E_1\to \u{a}\to 0, 0\to \u{c}\oplus\u{d}\to E_2\to\u{a}\to 0]\]
Also, the path:
\[\xymatrix{
ABCD\ar[dr] & & BCDA\\
& BCAD\ar[ur] & \\
}\]
gives the composite of the two braiding spans:
\[(B_{A,BC}\times D)(BC\times B_{A,D})=[(a,b,c,d,E_3,E_4)\mid\]
\[0\to\u{b}\oplus \u{c}\to E_3\to \u{a}\to 0, 0\to \u{d}\to E_4\to\u{a}\to 0]\]
Finally, we get the longest path:
\[\xymatrix{
       &BACD\ar[dd] & \\
ABCD\ar[ur] & & BCDA\\
& BCAD\ar[ur] & \\
}\]
which gives the composite of the three braiding spans:
\[(B_{A,B}\times CD)(B\times B_{A,C}\times D)(BC\times B_{A,D})=[(a,b,c,d,E_5,E_6,E_7)\mid\]
\[ 0\to \u{b}\to E_5\to\u{a}\to 0,0\to \u{c}\to E_6\to\u{a}\to 0,0\to \u{d}\to E_7\to\u{a}\to 0]\]
What we need to show is that the two way to get from the short path to the long path are the same. This is the same as showing the following diagram commutes:
\[\xymatrix{
B_{A,BCD}\ar[r]^{R}\ar[d]_{R} & (B_{A,B}\times CD)(B\times B_{A,CD})\ar[d]^{R}\\
(B_{A,BC}\times D)(BC\times B_{A,D})\ar[r]_{R\;\;\;\;\;\;\;\;\;\;\;\;\;\;\;\;\;} & (B_{A,B}\times CD)(B\times B_{A,C}\times D)(BC\times B_{A,D})\\
}\]
Following $R$ along the top path of the square, we get that the short exact sequence:
\[0\to \u{b}\oplus\u{c}\oplus\u{d}\to E\to \u{a}\to 0\]
is split once at $\u{b}$:
\[\xymatrix{0\to \u{b}\to E/(\u{c}\oplus\u{d})\to \u{a}\to 0\\
0\to \u{c}\oplus\u{d}\to E/\u{b}\to \u{a}\to 0\\
}\]
and then again between $\u{c}$ and $\u{d}$:
\[\xymatrix{
0\to \u{b}\to E/(\u{c}\oplus\u{d})\to \u{a}\to 0\\
0\to \u{c}\to (E/\u{b})/\u{d}\to \u{a}\to 0\\
0\to \u{d}\to (E/\u{b})/\u{c}\to \u{a}\to 0\\
}
\]
Similarly, following the bottom path, we get the same short exact sequence:
\[0\to \u{b}\oplus\u{c}\oplus\u{d}\to E\to \u{a}\to 0 \]
but this time we split off $\u{d}$:
\[\xymatrix{
0\to \u{b}\oplus\u{c}\to E/\u{d}\to \u{a}\to 0\\
0\to \u{d}\to E/(\u{b}\oplus\u{c})\to \u{a}\to 0\\
}\]
followed by a split between $\u{b}$ and $\u{c}$:
\[\xymatrix{
0\to \u{b}\to (E/\u{d})\u{c}\to \u{a}\to 0\\
0\to \u{c}\to (E/\u{d})/\u{b}\to \u{a}\to 0\\
0\to \u{d}\to E/(\u{b}\oplus\u{c})\to \u{a}\to 0\\
}
\]
To show that these two paths are the same, we need the following general fact. Given a module $E$ and a submodule $A\oplus B$, then the following are  natural isomorphic:
\[(E/A)/B\iso E/(A\oplus B).\]
With this, the three short exact sequences of the first path are natural isomorphic to the three short exact sequence of the second path. Thus the composite of the two functors labeling the front faces of our tetrahedron is naturally isomorphic to the composite of the back two. Recall, that in our category $\Span(\Gpd\downarrow \mathcal{A}_0)$, the $2$-morphisms are equivalence classes of maps between spans, so our polytope commutes.
\item
The second simplified polytope governs the ways to shuffle 3 objects through 1 object. Similar to the first diagram, there are four ways to shuffle three objects $A, B, C$ into one other object $D$; these four objects would form the points of a tetrahedron with braidings as edges. Again, the four vertices can be described by the order of the four objects as follows:
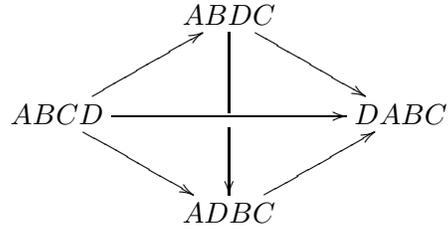
\[\xymatrix{
       &ABDC\ar[dr]\ar @{-}[d] & \\
ABCD\ar[ur]\ar[rr]\ar[dr] &{ }\ar[d] & DABC\\
& ADBC\ar[ur] & \\
}\]
However, the sides of this tetrahedron are filled in with the appropriate $S$. Verifying this diagram comes down to checking that the two composites of $S$'s going from the shortest path to the longest path are equal. To do this, we need to describe the groupoids at the apex of the four paths around the tetrahedron, then describe what the map between these span coming from $R$ does. The shortest path:
\[\xymatrix{ABCD\ar[rr]& & DABC}\]
Is just the braiding span:
\[B_{ABC,D} = \{(a,b,c,d,E)\mid 0\to \u{d}\to E\to \u{a}\oplus\u{b}\oplus\u{c}\to 0\}\]
The upper path:
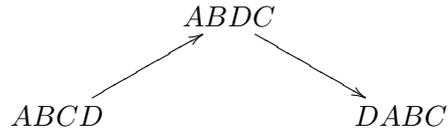
\[\xymatrix{
       &ABDC\ar[dr] & \\
ABCD\ar[ur] & & DABC\\
}\]
gives the composite of the two braiding spans: 
\[(AB\times B_{C,D})(B_{AB,D}\times C).\]
the lower path:
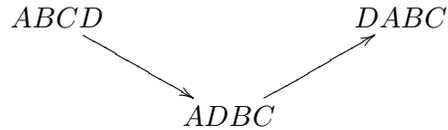
\[\xymatrix{
ABCD\ar[dr] & & DABC\\
& ADBC\ar[ur] & \\
}\]
gives the composite of the two braiding spans:
\[(A\times B_{BC,D})(B_{A,D}\times BC).\]
The longest path:
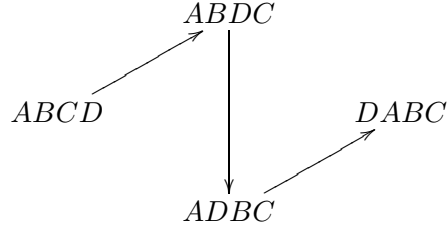
\[\xymatrix{
       &ABDC\ar[dd] & \\
ABCD\ar[ur] & & DABC\\
& ADBC\ar[ur] & \\
}\]
which gives the composite of the three braiding spans:
\[(AB\times B_{C,D})(A\times B_{B,D}\times C)(B_{A,D}\times BC).\]
What we need to show is that the two way to get from the short path to the long path are the same. This is the same as showing the following diagram commutes:
\[\xymatrix{
B_{ABC,D}\ar[r]^{S}\ar[d]_{S} & (AB\times B_{C,D})(B_{AB,D}\times C)\ar[d]^{S}\\
(A\times B_{BC,D})(B_{A,D}\times BC)\ar[r]_{S\;\;\;\;\;\;\;\;\;\;\;\;\;\;\;\;\;} & (AB\times B_{C,D})(A\times B_{B,D}\times C)(B_{A,D}\times BC)\\
}\]
Following $R$ along the top path of the square, we get that the short exact sequence:
\[0\to \u{d}\to E\to \u{a}\oplus\u{b}\oplus\u{c}\to 0\]
is split once at $\u{c}$:
\[\xymatrix{0\to \u{d}\to g^{-1}(\u{c})\to \u{c}\to 0\\
0\to \u{d}\to g^{-1}(\u{a}\oplus\u{b})\to \u{a}\oplus\u{b}\to 0\\
}\]
and then again between $\u{a}$ and $\u{b}$:
\[\xymatrix{
0\to \u{d}\to g^{-1}(\u{c})\to \u{c}\to 0\\
0\to \u{d}\to g^{-1}(\u{b})\to \u{b}\to 0\\
0\to \u{d}\to g^{-1}(\u{a})\to \u{a}\to 0\\
}
\]
Similarly, following the bottom path, we get the same short exact sequence:
\[0\to \u{b}\oplus\u{c}\oplus\u{d}\to E\to \u{a}\to 0 \]
but this time we split off $\u{d}$:
\[\xymatrix{
0\to \u{d}\to g^{-1}(\u{a})\to \u{a}\to 0\\
0\to \u{d}\to g^{-1}(\u{b}\oplus\u{c})\to \u{b}\oplus\u{c}\to 0\\
}\]
followed by a split between $\u{b}$ and $\u{c}$:
\[\xymatrix{
0\to \u{d}\to g^{-1}(\u{c})\to \u{c}\to 0\\
0\to \u{d}\to g^{-1}(\u{b})\to \u{b}\to 0\\
0\to \u{d}\to g^{-1}(\u{a})\to \u{a}\to 0\\
}
\]
it is clear that the paths are equal.
\item
The third simplified polytope governs the ways to shuffle 2 objects through 2 other objects. Just like before we can ignore the pentagon identity in the diagram. However, the ways to shuffle two objects $A,B$ into two others $C,D$ form a six-vertex polytope, a cube with two of its corners completely truncated. For this diagram, it is easier to see if we split it into a `front' and `back' view:
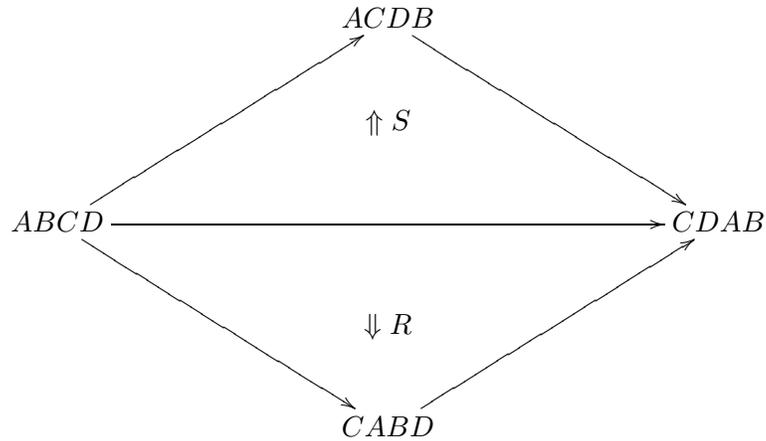
\[\xymatrix{
                                    & & & ACDB\ar[ddrrr] & & &      \\
                                    & & & \Uparrow S     & & &      \\
ABCD\ar[uurrr]\ar[ddrrr]\ar[rrrrrr] & & &                & & & CDAB \\
                                    & & & \Downarrow R   & & &      \\      
                                    & & & CABD\ar[uurrr] & & &      \\
}\]
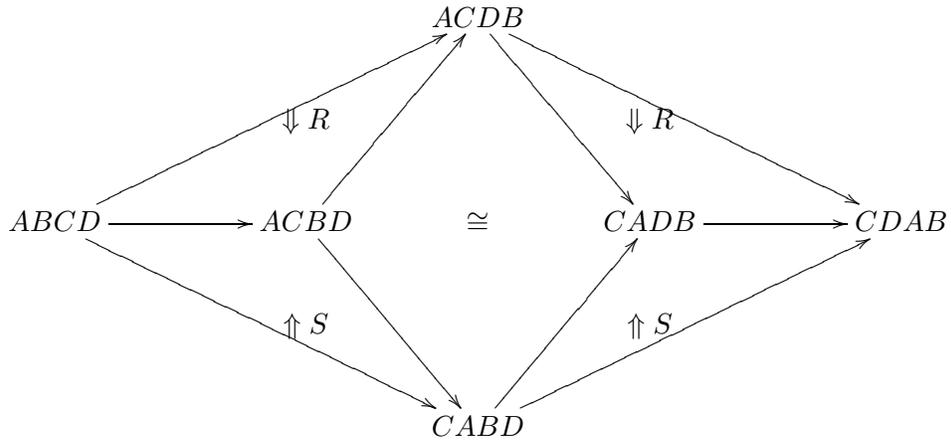
\[\xymatrix{                                   
                                & &                      & ACDB\ar[ddrrr]\ar[ddr] &              & &     \\
                                & &\Downarrow R          &                        &\Downarrow R  & &     \\
ABCD\ar[uurrr]\ar[ddrrr]\ar[rr] & & ACBD\ar[uur]\ar[ddr] &      \iso              & CADB \ar[rr] & &CDAB \\      
                                & &\Uparrow S            &                        &\Uparrow S    & &     \\
                                & &                      & CABD\ar[uurrr]\ar[uur] &              & &     \\
}\]
Again, checking this polytope involves verifying the two composites of faces from the shortest path to the longest path are isomorphic. In total there are six different paths on this polytope from $ABCD$ to $CDAB$.
Starting with the shortest path:
\[\xymatrix{
ABCD\ar[rrrrrr] & & &                & & & CDAB \\
}\]
we get the braiding span with apex groupoid:
\[B_{AB,CD}\{(a,b,c,d,E_1)\mid\]
\[0\to \u{c}\oplus\u{d}\to E_1\to \u{a}\oplus\u{b}\to 0\}.\]
Working towards the top we get the path: 
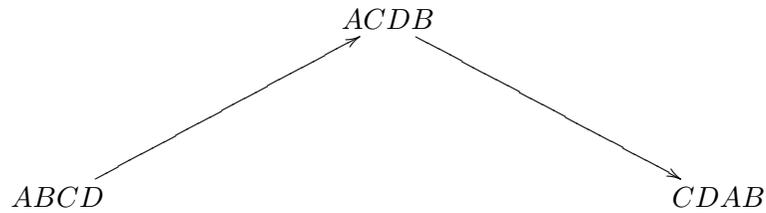
\[\xymatrix{
                                    & & & ACDB\ar[ddrrr] & & &      \\
                                    & & &     & & &      \\
ABCD\ar[uurrr] & & &                & & & CDAB \\
}\]
which is the composite of the two braiding spans:
\[(B_{A,CD}\times B)(A\times B_{B,CD})=\{(a,b,c,d,E_2,E_3)\mid\]
\[0\to \u{c}\oplus\u{d}\to E_2\to \u{a}\to 0, 0\to \u{c}\oplus\u{d}\to E_3\to \u{b}\to 0\}.\]
Going towards the bottom of the polytope, we get the path:
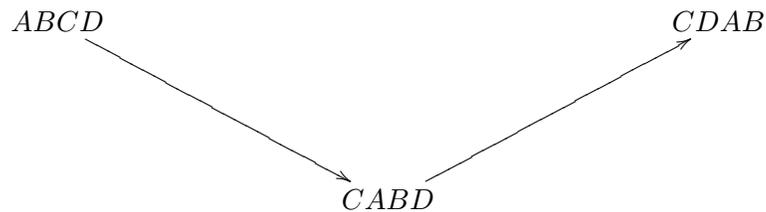
\[\xymatrix{
ABCD\ar[ddrrr] & & &                & & & CDAB \\
                                    & & &   & & &      \\      
                                    & & & CABD\ar[uurrr] & & &      \\
                                    }\]
which is the composite of the two braiding spans:
\[(C\times B_{AB,D})(B_{AB,C}\times D)=\{(a,b,c,d,E_4, E_5)\mid\]
\[0\to\u{c}\to E_4\to \u{a}\oplus\u{b}\to 0, 0\to \u{d}\to E_5\to \u{a}\oplus\u{b}\to 0\}.\]
Along the top back of the polytope, the path:
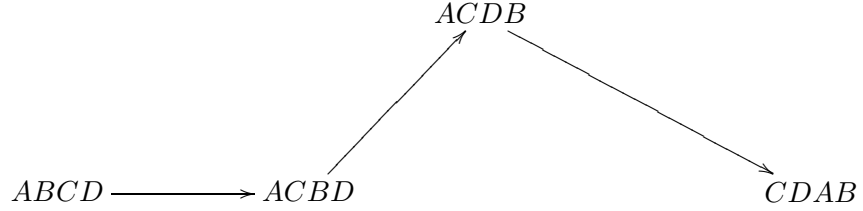
\[\xymatrix{
                               & &                      & ACDB\ar[ddrrr] &              & &     \\
                                & &        &                        &  & &     \\
ABCD\ar[rr] & & ACBD\ar[uur] &              & & &CDAB \\      
}\]
is the composite of the three spans:
\[(B_{A,CD}\times B)(A\times C B_{B,D})(A\times B_{B,C}\times D)=\{(a,b,c,d,E_6,E_7,E_8\mid \]
\[0\to \u{c}\oplus\u{d}\to E_6\to \u{a}\to 0, 0\to \u{c}\to E_7\to \u{b}\to 0, 0\to \u{d}\to E_8\to \u{b}\to 0\}. \]
Along the bottom back of the polytope, we get the path:
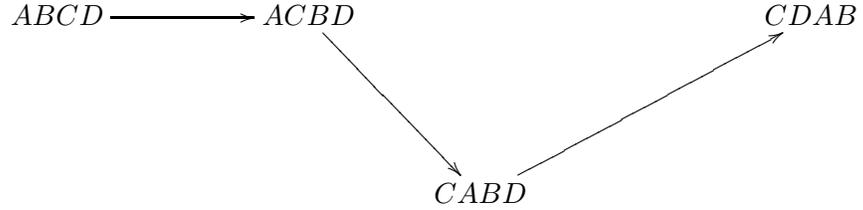
\[\xymatrix{
ABCD\ar[rr] & & ACBD\ar[ddr] &                &  & &CDAB \\      
            & &              &                &  & &     \\
            & &              & CABD\ar[uurrr] &  & &     \\
}\]
which gives the composite of the three braiding spans:
\[(C\times B_{AB,D})(B_{A,C}\times BD)(A\times B_{B,C}\times D)=\{(a,b,c,d,E_9,E_{10},E_{11})\mid\]
\[0\to \u{c}\to E_{9}\to \u{a}\to 0, 0\to \u{d}\to E_{10}\to \u{a}\to 0, 0\to \u{c}\oplus\u{d}\to E_{11}\to \u{b}\to 0\}\]
Finally, we have the longest path:
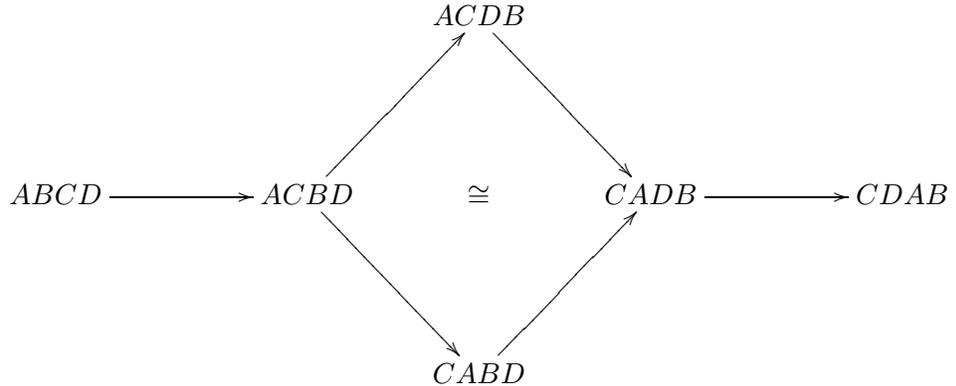
\[\xymatrix{
                                & &                      & ACDB\ar[ddr] &              & &     \\
                                & &         &                        &  & &     \\
ABCD\ar[rr] & & ACBD\ar[uur]\ar[ddr] &      \iso              & CADB \ar[rr] & &CDAB \\      
                                & &           &                        &    & &     \\
                                & &                      & CABD\ar[uur] &              & &     \\
}\]
Which gives the composite of the four braiding spans:
\[(C \times B_{A,D}\times B)(B_{A,C}\times DB)(AC\times B_{B,D})(A\times B_{B,C}\times D)=\{(a,b,c,d,E_7,E_8,E_9,E_{10}\mid\]
\[0\to \u{c}\to E_{9}\to \u{a}\to 0, 0\to \u{d}\to E_{10}\to \u{a}\to 0,\]
\[0\to \u{c}\to E_7\to \u{b}\to 0, 0\to \u{d}\to E_8\to \u{b}\to 0\}\]
Putting all of these together, the polytope will commute if the composite of functors:\\
\[\xymatrix{
B_{AB,CD}\ar[d]^S \\
(B_{A,CD}\times B)(A\times B_{B,CD})\ar[d]^R \\
(B_{A,CD}\times B)(A\times C B_{B,D})(A\times B_{B,C}\times D)\ar[d]^R \\
(C \times B_{A,D}\times B)(B_{A,C}\times DB)(AC\times B_{B,D})(A\times B_{B,C}\times D) \\
}\]
is equivalent to the composite of functors:
\[\xymatrix{
B_{AB,CD}\ar[d]^R\\
(C\times B_{AB,D})(B_{AB,C}\times D)\ar[d]^S\\
(C\times B_{AB,D})(B_{A,C}\times BD)(A\times B_{B,C}\times D)\ar[d]^S\\
(C \times B_{A,D}\times B)(B_{A,C}\times DB)(AC\times B_{B,D})(A\times B_{B,C}\times D) \\
}\]
as maps between the spans. Just like the previous polytopes, verifying this requires a bit of algebra. 
\end{itemize}
\end{proof}

\chapter{The Hopf 2-Algebra Structure}\label{H2A}
In order to complete the groupoidification of the Hall algebra, we would need to construct spans of groupoids which will stand in for the multiplication, unit, comultiplication, counit, and antipode, and then show that with these spans we get a Hopf $2$-algebra in the braided monoidal bicategory $\Span(\Gpd\downarrow\mathcal{A}_0)$. What will do here is describe the multiplication and comultiplication spans, and show that they degroupoidify into the multiplication and comultiplication in Chapter \ref{HA}. In later work we will define the unit and counit, as well as the coherence isomorphisms that are part of the definition of Hopf $2$-algebra in Neuchl \cite{Neuchl} and Pfeiffer \cite{Pfeiffer}.
\section{The Multiplication and Comultiplication Spans}
We start with the multiplication span. Since the Hall algebra product can be seen as a linear operator
\[  
\begin{array}{ccc} 
 \R[\u{X}] \tensor \R[\u{X}] &\to& \R[\u{X}] \\
         a \tensor b       & \mapsto & a\cdot b 
\end{array}
\]
it is natural to seek a span of groupoids 
\[
\xymatrix@!C{
& {\rm ???}\ar[dl]_{q} \ar[dr]^{p} & \\
X & & X\times X
}
\]
that gives this operator.  Indeed, there is a very natural span that
gives this product.  This will allow us to groupoidify the algebra
$U_q^+ \g$.  

We start by defining a groupoid $\SES(Q)$ to serve as the apex of this
span.  An object of $\SES(Q)$ is a short exact sequence in $\Rep(Q)$,
and a morphism from
\[ 0\to N \stackrel{f}{\to} E \stackrel{g}{\to} M\to 0 \]
to
\[ 0\to N' \stackrel{f'}{\to} E' \stackrel{g'}{\to} M'\to 0 \]
is a commutative diagram
\[
\xymatrix{
0 \ar[r] & 
N \ar[r]^{f} \ar[d]^{\alpha} & 
E \ar[r]^{g} \ar[d]^{\beta} & 
M \ar[r] \ar[d]^{\gamma} & 0 \\
0 \ar[r] & N' \ar[r]^{f'} & E' \ar[r]^{g'} & M' \ar[r] & 0 \\
}
\]
where $\alpha,\beta,$ and $\gamma$ are isomorphisms of
quiver representations.

Next, we define the span
\[
\xymatrix@!C{
& \SES(X)\ar[dl]_{q} \ar[dr]^{p} & \\
X & & X\times X
}
\]
where $p$ and $q$ are given on objects by
\[
\begin{array}{ccl}
  p(0\to N \stackrel{f}{\to} E \stackrel{g}{\to} M\to 0) &=& (M,N) \\
  q(0\to N \stackrel{f}{\to} E \stackrel{g}{\to} M\to 0) &=& E 
\end{array}
\]
and defined in the natural way on morphisms.  This span captures the
idea behind the standard Hall algebra multiplication.  Given two
quiver representations $M$ and $N$, this span relates them to every 
representation $E$ that is an extension of $M$ by $N$.

Before we degroupoidify this span, we need to decide on a convention. As stated in previous parts of this work \ref{groupoidification}, the correct choice is to work with homology and $\alpha$-degroupoidification
with $\alpha = 1$, as described in \cite{BaezHoffnungWalker:2009HDA7}. 
Recall that a span of finite type
\[\xymatrix{
 & S\ar[dl]_q\ar[dr]^p & \\
 Y & & X \\
}\]
yields an operator 
\[ \utilde{S}\maps \R[\u{X}] \to \R[\u{Y}]  \]
given by:
\[
\utilde{S} [x] =
\sum_{[y] \in \u{Y}} \;\,
\sum_{[s]\in\u{p^{-1}(x)}\bigcap \u{q^{-1}(y)}}
\frac{|\Aut(y)|}{|\Aut(s)|} \,\, [y]   \, .
\]
We can rewrite this using groupoid cardinality as follows:
\[
\utilde{S} [x] =
\sum_{[y] \in \u{Y}} \;\,
|\Aut(y)| \,\, |(p \times q)^{-1}(x,y)| \,\, [y]   \, .
\]
Applying this procedure to the span with $\SES(Q)$ as its apex, we get
an operator
\[  m \maps \R[\u{X}] \tensor \R[\u{X}] \to \R[\u{X}] \]
with
\[ m([M] \tensor [N]) 
= \sum_{E\in\mathcal{P}^E_{MN}}
|\Aut(E)| \, |(p\times q)^{-1}(M,N,E)|  \, \, [E] .  \]
We wish to show this matches the Hall algebra product $[M] \cdot [N]$.

For this, we must make a few observations. First, we note that the
group $\Aut(N) \times \Aut(E) \times \Aut(M)$ acts on the set
$\mathcal{P}^E_{MN}$. This action is not necessarily free, but this is
just the sort of situation groupoid cardinality is designed to handle.
Taking the weak quotient, we obtain a groupoid equivalent to the
groupoid where objects are short exact sequences of the form $0\to
N\to E\to M\to 0$ and morphisms are isomorphisms of short exact
sequences.  So, the weak quotient is equivalent to the groupoid
$(p\times q)^{-1}(M,N,E)$.  Remembering that groupoid cardinality is
preserved under equivalence, we see:
\[
\begin{array}{rcl}
|(p\times q)^{-1}(M,N,E)| &=& 
|\mathcal{P}^E_{MN}/\!/(\Aut(N)\times\Aut(E)\times\Aut(M))| \\  \\
&=& 
\displaystyle{\frac{|\mathcal{P}^E_{MN}|}
{|\Aut(N)| \, |\Aut(E)| \, |\Aut(M)|} }
\end{array}
\]
So, we obtain
\[ m([M] \tensor [N]) 
= \sum_{E\in\mathcal{P}^E_{MN}}
\frac{|\mathcal{P}^E_{MN}|} {|\Aut(M)| \, |\Aut(N)|}  \,\, [E] .  \]
which is precisely the Hall algebra product $[M] \cdot [N]$.

Similarly, we construct the comultiplication span as the adjoint of the multiplication span:
\[
\xymatrix@!C{
& \SES(X)\ar[dl]_{p} \ar[dr]^{q} & \\
X\times X & & X\\
}
\]
Applying the same procedure to this, we get
an operator
\[  \Delta \maps \R[\u{X}] \to \R[\u{X}] \tensor \R[\u{X}] \]
with
\[ \Delta([E]) 
= \sum_{[M],[N]}
|\Aut(M)|\,|\Aut(N)| \, |(p\times q)^{-1}(M,N,E)|  \, \, [M]\otimes[N] .  \]
We wish to show this matches the Hall algebra comultiplication. Just like before, we make the substitution:
\[
\begin{array}{rcl}
|(p\times q)^{-1}(M,N,E)| &=& 
|\mathcal{P}^E_{MN}/\!/(\Aut(N)\times\Aut(E)\times\Aut(M))| \\  \\
&=& 
\displaystyle{\frac{|\mathcal{P}^E_{MN}|}
{|\Aut(N)| \, |\Aut(E)| \, |\Aut(M)|} }
\end{array}
\]
which simplifies this formula as follows:
\[ \Delta([M]) 
= \sum_{[M],[N]}
\frac{|\mathcal{P}^E_{MN}|} {|\Aut(E)|}  \,\, [M]\otimes [N] .  \]
This is precisely the comultiplication for the Hall algebra.

\section{Summary and Future Work}
It is prudent at this point to provide a complete list of the important results, how each is a categorification of a structure for the Hall algebra, and conjectures on how this will be extended in later work.

\begin{theorem}
Let $Q$ be a simply laced quiver, $\mathcal{A}=\Rep(Q)$ its category of finite dimensional representations over a finite field $\F_q$, and $\mathcal{A}_0$ the underlying groupoid of $\mathcal{A}$. Then the braided monoidal bicategory $\Span(\Gpd\downarrow \mathcal{A}_0)$ degroupoidifies into the braided monoidal category $\Vect^K$, where:
\begin{itemize}
\item A groupoid over $\mathcal{A}_0$ gives a $K$-graded vector space.
\item A span of groupoids over $\mathcal{A}_0$ gets sent to a map of $K$-graded vector spaces.
\item The tensor product of groupoids over $\mathcal{A}_0$ gets sent to the tensor product of $K$-graded vector spaces.
\item The associator and unitor in $\Span(\Gpd \downarrow \mathcal{A}_0)$ get sent to the associator and unitor in $\Vect^K$.
\item The braiding span in $\Span(\Gpd \downarrow \mathcal{A}_0)$ gets sent to the braiding we constructed in $\Vect^K$.
\end{itemize}
Also, using this process we can show:
\begin{itemize}
\item $\mathcal{A}_0$, viewed as a groupoid over itself, gives the Hall algebra with its `standard' $K$-grading.
\item the multiplication and comultiplication spans give the product and coproduct for the Hall algebra.
\end{itemize}
\end{theorem}

The results of the second half of this theorem give us a clue that we could extend these structures to ones which will degroupoidify to the Hopf algebra structure for the Hall algebra. Specifically, we make the following conjecture for continuing this work.

\begin{conjecture}
The groupoid $\mathcal{A}_0$ viewed as a groupoid over itself, along with the multiplication and comultiplication spans described in this chapter, can be extended to a Hopf $2$-algebra in the braided monoidal bicategory $\Span(\Gpd\downarrow \mathcal{A}_0)$.
\end{conjecture}

\appendix
\chapter{Definitions}
In this section we give various definitions used throughout the paper. The specific definitions given and the associated diagrams where provide by Mike Stay.

\begin{defn}\label{bicatDEF}
  A {\bf bicategory} is a horizontal categorification of a monoidal category---that is, any one-object bicategory is a monoidal category---so much of the definition is similar.  A bicategory $\CC$ consists of
  \begin{enumerate}
    \item a collection of {\bf objects}
    \item for each pair of objects $A, B$ in $\CC$, we have a category $\CC(A,B)$; the objects of $\CC(A,B)$ are called {\bf 1-morphisms}, while the morphisms of $\CC(A,B)$ are called {\bf 2-morphisms}.
    \item for each triple of objects $A, B, C$ in $\CC$, a {\bf composition} functor 
        \[ \circ_{A,B,C}\maps \CC(B, C) \times \CC(A, B) \to \CC(A, C).\]
      We will leave off the subscript, since it should be clear from the context.
    \item for each object $A$ in $\CC$, an object $1_A$ in $\CC(A,A)$ called the {\bf identity 1-morphism on $A$}.  We will often write this simply as $A$.
    \item for each quadruple of objects $A, B, C, D$, a natural isomorphism called the {\bf associator for composition}; if $(f,g,h)$ is an object of $\CC(C,D) \times \CC(B,C) \times \CC(A,B),$ then 
      $$a_{A,B,C,D}(f,g,h) \maps (f \circ g) \circ h \to f \circ (g \circ h).$$
    \item for each pair of objects $A, B$ in $\CC$, natural isomorphisms called {\bf left and right unitors for composition.}  If $f$ is an object of $\CC(A,B)$, then
      \[\begin{array}{l}
        l_{A,B}(f)\maps B \circ f \stackrel{\sim}{\to} f\\
        r_{A,B}(f)\maps f \circ A \stackrel{\sim}{\to} f
      \end{array}\]
  \end{enumerate}
  such that $a, l,$ and $r$ satisfy the following coherence laws:
  \begin{enumerate}
    \item for all $(f,g,h,j)$ in $\CC(D,E) \times \CC(C,D) \times \CC(B,C) \times \CC(A,B),$ the following diagram, called the {\bf pentagon equation}, commutes:
      \[\begin{diagram}
        \node[2]{((f \circ g) \circ h) \circ j} \arrow{sse,t}{a_{f\circ g, h, j}} \arrow{sw,t}{a_{f,g,h} \circ j} \\
        \node{(f \circ (g \circ h)) \circ j} \arrow[2]{s,l}{a_{f, g\circ h, j}}\\
        \node[3]{(f \circ g) \circ (h \circ j)} \arrow{ssw,b}{a_{f,g,h \circ j}}\\
        \node{f \circ ((g \circ h) \circ j)} \arrow{se,b}{f \circ a_{g,h,j}}\\
        \node[2]{f \circ (g \circ (h \circ j))}
      \end{diagram}\]
    \item for all $(f,g)$ in $\CC(B,C)\times \CC(A,B)$ the following diagram, called the {\bf triangle equation}, commutes:
      \[\begin{diagram}
        \node{(f \circ B) \circ g} \arrow[2]{e,t}{a} \arrow{se,b}{r(f) \circ g} \node[2]{f \circ (B \circ g)} \arrow{sw,b}{f \circ l(g)} \\
        \node[2]{f \circ g}
      \end{diagram}\]
  \end{enumerate}
\end{defn}

The associator $a$ and unitors $r, l$ for composition of 1-morphisms are necessary, but when we are drawing commutative diagrams of 1-morphisms they are very hard to show; fortunately any consistent choice is equivalent to any other, so we leave them out.
\newpage
\begin{defn}\label{monoidalDEF} A {\bf monoidal} bicategory $\CC$ consists of the following data subject to the following axioms. \\
DATA:
\begin{itemize}
\item A bicategory $\CC$.
\item A {\bf tensor product} functor $\otimes:\CC \times \CC \rightarrow \CC$.
\item A {\bf monoidal unit} object $I$.
\item An {\bf associator} pseudonatural isomorphism
\[ a\maps (A\tensor B) \tensor C \Rightarrow A \tensor (B \tensor C) \]
for moving parentheses around among three tensored objects.  Here and below we are using expressions like $(A \tensor B) \tensor C$ to denote a functor like $\tensor \circ (\tensor \times \CC)\maps \CC^3 \to \CC.$
\item {\bf Unitor} pseudonatural isomorphisms
\[ l\maps I \tensor A \Rightarrow A \]
\[ r\maps A \tensor I \Rightarrow A \]
for the interaction of one object with $I$;
\item A {\bf pentagonator} isomorphism 2-cell $\pi$ (i.e., an invertible modification) for moving parentheses among four objects.
\begin{center}
  \begin{tikzpicture}
    \node (ABCD1) at (0,3) {$((A \tensor B) \tensor C) \tensor D$};
    \node (ABCD2) at (2,4) {$(A\tensor B)(C\tensor D)$}
      edge [<-] node [l, above left] {$a$} (ABCD1);
    \node (ABCD3) at (4,2) {$A \tensor (B \tensor (C \tensor D))$}
      edge [<-] node [l, above right] {$a$} (ABCD2);
    \node (ABCD4) at (2,0) {$A \tensor ((B \tensor C) \tensor D)$}
      edge [->] node [l, below right] {$A \tensor a$} (ABCD3);
    \node (ABCD5) at (0,1) {$(A \tensor (B \tensor C)) \tensor D$}
      edge [->] node [l, below left] {$a$} (ABCD4)
      edge [<-] node [l, left] {$D \tensor a$} (ABCD1);
    \node at (1,2) {\tikz\node [rotate=90] {$\Rightarrow$};};
    \node at (1.5,2) {$\pi$};
  \end{tikzpicture}
\end{center}
\newpage
\item Invertible modifications $\lambda, \mu,$ and $\rho$ for the interaction of two objects with $I$.
\begin{center}
  \begin{tikzpicture}
    \node (AB1) at (0,0) {$A \tensor B$};
    \node (AB2) at (3,0) {$A \tensor B$}
      edge [<-] node [l, below] {$A\tensor B$} (AB1);
    \node (AIB2) at (3,3) {$A\tensor (I\tensor B)$}
      edge [->] node [l, right] {$A \tensor l$} (AB2);
    \node (AIB1) at (0,3) {$(A\tensor I) \tensor B$}
      edge [->] node [l, above] {$a$} (AIB2)
      edge [<-] node [l, left] {$r^* \tensor B$} (AB1);
    \node at (1,1.5) {\tikz\node [rotate=-90] {$\Rightarrow$};};
    \node at (1.5,1.5) {$\mu$};
  \end{tikzpicture}
\end{center}
\begin{center}
  \begin{tikzpicture}
    \node (IAB1) at (0,2) {$(I\tensor A)\tensor B$};
    \node (IAB2) at (3,1) {$I\tensor (A \tensor B)$}
      edge [<-] node [l, above right] {$a$} (IAB1);
    \node (AB) at (0,0) {$A \tensor B$}
      edge [<-] node [l, below right] {$l$} (IAB2)
      edge [<-] node [l, left] {$l \tensor B$} (IAB1);
    \node at (1,1) {$\Rightarrow \lambda$};
  \end{tikzpicture}
\end{center}
\begin{center}
  \begin{tikzpicture}
    \node (ABI1) at (0,2) {$A \tensor (B \tensor I)$};
    \node (ABI2) at (3,1) {$(A \tensor B) \tensor I$}
      edge [<-] node [l, above right] {$a^*$} (ABI1);
    \node (AB) at (0,0) {$A \tensor B$}
      edge [<-] node [l, below right] {$r$} (ABI2)
      edge [<-] node [l, left] {$A \tensor r$} (ABI1);
    \node at (1,1) {$\Rightarrow \rho$};
  \end{tikzpicture}
\end{center}
\end{itemize}
AXIOMS:
\begin{itemize}
  \item The following equation of 2-morphisms holds in the bicategory $\CC$, where we have used parentheses instead of $\tensor$ for compactness and the unmarked isomorphisms are naturality isomorphisms for the associator.  The equation governs moving parentheses around among five objects. 

    \begin{center}
      \begin{tikzpicture}[line join=round,scale=.7]
        \begin{scope}[font=\fontsize{8}{8}\selectfont]
          \node (A) at (-2.391,3.949) {$(A(B(CD)))E$};
          \node (B) at (-5.127,.19) {$(A((BC)D))E$}
            edge [->] node [l, above left] {$(Aa)E$} (A);
          \node (C) at (-5.127,-2.581) {$((A(BC))D)E$}
            edge [->] node [l, left] {$aE$} (B);
          \node (D) at (-4.306,-3.532) {$(((AB)C)D)E$}
            edge [->] node [l, left] {$(aD)E$} (C);
          \node (E) at (3.212,-4.9) {$((AB)C)(DE)$}
            edge [<-] node [l, below] {$a$} (D);
          \node (F) at (4.306,-3.396) {$(AB)(C(DE))$}
            edge [<-] node [l, below right] {$a$} (E);
          \node (G) at (4.306,3.532) {$A(B(C(DE)))$}
            edge [<-] node [l, right] {$a$} (F);
          \node (H) at (2.872,5.07) {$A(B((CD)E))$}
            edge [->] node [l, above right] {$A(Ba)$} (G);
          \node (I) at (-.135,5.617) {$A((B(CD))E)$}
            edge [->] node [l, above] {$Aa$} (H)
            edge [<-] node [l, above left] {$a$} (A);
          \node (M) at (-2.391,-.901) {$((AB)(CD))E$}
            edge [<-] node [l, right] {$aE$} (D)
            edge [->] node [l, right] {$aE$} (A);
          \node (N) at (2.872,-1.858) {$(AB((CD)E))$}
            edge [<-] node [l, above] {$a$} (M)
            edge [->] node [l, left] {$a$} (H)
            edge [->] node [l, left] {$(AB)a$} (F);
          \node at (-4,-.5) {$\Rightarrow \pi$};
          \node at (0,-3) {\tikz\node [rotate=-90] {$\Rightarrow$};};
          \node at (0.5,-3) {$\pi$};
          \node at (0,2) {\tikz\node [rotate=-45] {$\Rightarrow$};};
          \node at (0.5,2) {$\pi$};
          \node at (3.5,1) {$\cong$};
        \end{scope}
      \end{tikzpicture}
      \\
      =
      \\
      \begin{tikzpicture}[line join=round,scale=.7]
        \begin{scope}[font=\fontsize{8}{8}\selectfont]
          \node (A) at (-2.391,3.949) {$(A(B(CD)))E$};
          \node (B) at (-5.127,.19) {$(A((BC)D))E$}
            edge [->] node [l, above left] {$(Aa)E$} (A);
          \node (C) at (-5.127,-2.581) {$((A(BC))D)E$}
            edge [->] node [l, left] {$aE$} (B);
          \node (D) at (-4.306,-3.532) {$(((AB)C)D)E$}
            edge [->] node [l, left] {$(aD)E$} (C);
          \node (E) at (3.212,-4.9) {$((AB)C)(DE)$}
            edge [<-] node [l, below] {$a$} (D);
          \node (F) at (4.306,-3.396) {$(AB)(C(DE))$}
            edge [<-] node [l, below right] {$a$} (E);
          \node (G) at (4.306,3.532) {$A(B(C(DE)))$}
            edge [<-] node [l, right] {$a$} (F);
          \node (H) at (2.872,5.07) {$A(B((CD)E))$}
            edge [->] node [l, above right] {$A(Ba)$} (G);
          \node (I) at (-.135,5.617) {$A((B(CD))E)$}
            edge [->] node [l, above] {$Aa$} (H)
            edge [<-] node [l, above left] {$a$} (A);
          \node (J) at (-2.872,1.858) {$A(((BC)D)E)$}
            edge [->] node [l, below right] {$A(aE)$} (I)
            edge [<-] node [l, below right] {$a$} (B);
          \node (K) at (2.391,-3.949) {$(A(BC))(DE)$}
            edge [<-] node [l, left] {$a(DE)$} (E)
            edge [<-] node [l, above] {$a$} (C);
          \node (L) at (2.391,.901) {$A((BC)(DE))$}
            edge [<-] node [l, left] {$a$} (K)
            edge [<-] node [l, above] {$Aa$} (J)
            edge [->] node [l, above left] {$Aa$} (G);
          \node at (-1,-1) {\tikz\node [rotate=-45] {$\Rightarrow$};};
          \node at (-.5,-1) {$\pi$};
          \node at (1,3) {\tikz\node [rotate=-45] {$\Rightarrow$};};
          \node at (1.5,3) {$\pi$};
          \node at (3,-1.5) {\tikz\node [rotate=-45] {$\Rightarrow$};};
          \node at (3.5,-1.5) {$\pi$};
          \node at (-1,-3.7) {$\cong$};
          \node at (-2.5,3) {$\cong$};
        \end{scope}
      \end{tikzpicture}
    \end{center}
\newpage
  \item Two equations of 2-morphisms in $\CC$ for the interaction of $I$ with three other objects; the only cases not covered by the interaction with one or two other objects is where $I$ appears just to the left or right of the middle object.  Note that $I$ is an object, the unitors governing $I$ and one object are morphisms, the modifications governing $I$ and two objects are 2-morphisms, and this is an equation.  The unmarked isomorphisms are either naturality isomorphisms for the associator or unique coherence isomorphisms from $\CC$.  Each equation is a cube with one edge ``half-truncated'':
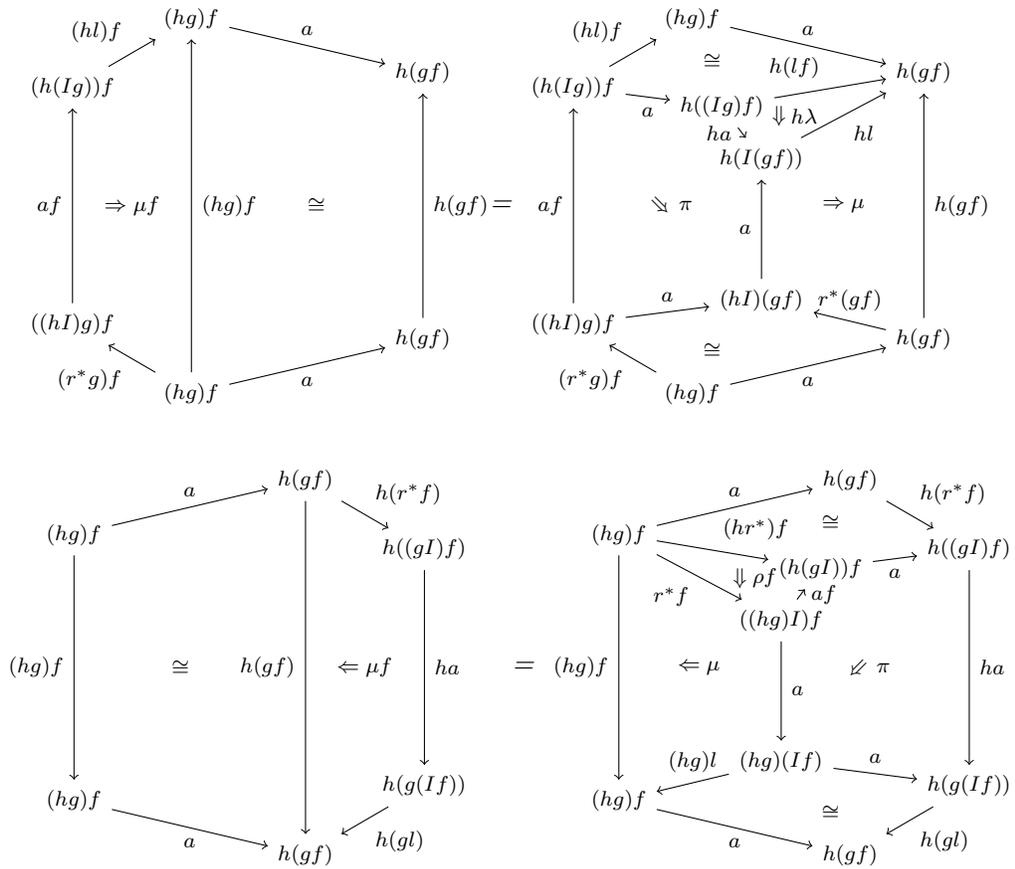
    \begin{center}
      \begin{tikzpicture}[line join=round,scale=.7]
        \begin{scope}[font=\fontsize{8}{8}\selectfont]
          \node (A) at (-.875,-3.562) {$(hg)f$};
          \node (B) at (-3.133,-2.249) {$((hI)g)f$}
            edge [<-] node [l, below left] {$(r^*g)f$} (A);
          \node (C) at (-3.133,2.249) {$(h(Ig))f$}
            edge [<-] node [l, left] {$af$} (B);
          \node (D) at (-.875,3.562) {$(hg)f$}
            edge [<-] node [l, above left] {$(hl)f$} (C)
            edge [<-] node [l, right] {$(hg)f$} (A);
          \node (E) at (3.522,2.529) {$h(gf)$}
            edge [<-] node [l, above] {$a$} (D);
          \node (F) at (3.522,-2.529) {$h(gf)$}
            edge [->] node [l, right] {$h(gf)$} (E)
            edge [<-] node [l, below] {$a$} (A);
          \node at (-2,0) {$\Rightarrow \mu f$};
          \node at (1.5,0) {$\cong$};
        \end{scope}
        \node at (5,0) {=};
      \end{tikzpicture}
      \begin{tikzpicture}[line join=round,scale=.7]
        \begin{scope}[font=\fontsize{8}{8}\selectfont]
          \node (A) at (-.875,-3.562) {$(hg)f$};
          \node (B) at (-3.133,-2.249) {$((hI)g)f$}
            edge [<-] node [l, below left] {$(r^*g)f$} (A);
          \node (C) at (-3.133,2.249) {$(h(Ig))f$}
            edge [<-] node [l, left] {$af$} (B);
          \node (D) at (-.875,3.562) {$(hg)f$}
            edge [<-] node [l, above left] {$(hl)f$} (C);
          \node (E) at (3.522,2.529) {$h(gf)$}
            edge [<-] node [l, above] {$a$} (D);
          \node (F) at (3.522,-2.529) {$h(gf)$}
            edge [->] node [l, right] {$h(gf)$} (E)
            edge [<-] node [l, below] {$a$} (A);
          \node (G) at (.439,-1.788) {$(hI)(gf)$}
            edge [<-] node [l, above] {$r^*(gf)$} (F)
            edge [<-] node [l, above] {$a$} (B);
          \node (H) at (.439,.894) {$h(I(gf))$}
            edge [<-] node [l, left] {$a$} (G)
            edge [->] node [l, below right] {$hl$} (E);
          \node (I) at (-.309,1.885) {$h((Ig)f)$}
            edge [->] node [l, left] {$ha$} (H)
            edge [->] node [l, above left] {$h(lf)$} (E)
            edge [<-] node [l, below] {$a$} (C);
          \node at (-1.5,0) {\tikz\node [rotate=-45] {$\Rightarrow$};};
          \node at (-1,0) {$\pi$};
          \node at (2,0) {$\Rightarrow \mu$};
          \node at (-.5,2.75) {$\cong$};
          \node at (-.5,-2.75) {$\cong$};
          \node at (.75,1.7) {\tikz\node [rotate=-90] {$\Rightarrow$};};
          \node at (1.25,1.7) {$h\lambda$};
        \end{scope}
      \end{tikzpicture}
    \end{center}
    \begin{center}
      \begin{tikzpicture}[line join=round,scale=.7]
        \begin{scope}[font=\fontsize{8}{8}\selectfont]
          \node (A) at (.875,-3.562) {$h(gf)$};
          \node (B) at (3.133,-2.249) {$h(g(If))$}
            edge [->] node [l, below right] {$h(gl)$} (A);
          \node (C) at (3.133,2.249) {$h((gI)f)$}
            edge [->] node [l, right] {$ha$} (B);
          \node (D) at (.875,3.562) {$h(gf)$}
            edge [->] node [l, above right] {$h(r^*f)$} (C)
            edge [->] node [l, left] {$h(gf)$} (A);
          \node (E) at (-3.522,2.529) {$(hg)f$}
            edge [->] node [l, above] {$a$} (D);
          \node (F) at (-3.522,-2.529) {$(hg)f$}
            edge [<-] node [l, left] {$(hg)f$} (E)
            edge [->] node [l, below] {$a$} (A);
          \node at (2,0) {$\Leftarrow \mu f$};
          \node at (-1.5,0) {$\cong$};
        \end{scope}
        \node at (5,0) {=};
      \end{tikzpicture}
      \begin{tikzpicture}[line join=round,scale=.7]
        \begin{scope}[font=\fontsize{8}{8}\selectfont]
          \node (A) at (.875,-3.562) {$h(gf)$};
          \node (B) at (3.133,-2.249) {$h(g(If))$}
            edge [->] node [l, below right] {$h(gl)$} (A);
          \node (C) at (3.133,2.249) {$h((gI)f)$}
            edge [->] node [l, right] {$ha$} (B);
          \node (D) at (.875,3.562) {$h(gf)$}
            edge [->] node [l, above right] {$h(r^*f)$} (C);
          \node (E) at (-3.522,2.529) {$(hg)f$}
            edge [->] node [l, above] {$a$} (D);
          \node (F) at (-3.522,-2.529) {$(hg)f$}
            edge [<-] node [l, left] {$(hg)f$} (E)
            edge [->] node [l, below] {$a$} (A);
          \node (G) at (-.439,-1.788) {$(hg)(If)$}
            edge [->] node [l, above] {$(hg)l$} (F)
            edge [->] node [l, above] {$a$} (B);
          \node (H) at (-.439,.894) {$((hg)I)f$}
            edge [->] node [l, right] {$a$} (G)
            edge [<-] node [l, below left] {$r^*f$} (E);
          \node (I) at (.309,1.885) {$(h(gI))f$}
            edge [<-] node [l, right] {$af$} (H)
            edge [<-] node [l, above right] {$(hr^*)f$} (E)
            edge [->] node [l, below] {$a$} (C);
          \node at (1,0) {\tikz\node [rotate=-135] {$\Rightarrow$};};
          \node at (1.5,0) {$\pi$};
          \node at (-2,0) {$\Leftarrow \mu$};
          \node at (.5,2.75) {$\cong$};
          \node at (.5,-2.75) {$\cong$};
          \node at (-1.25,1.7) {\tikz\node [rotate=-90] {$\Rightarrow$};};
          \node at (-.75,1.7) {$\rho f$};
        \end{scope}
      \end{tikzpicture}
    \end{center}
\end{itemize}
\end{defn}
\newpage
\begin{defn}\label{braidedDEF}
  A {\bf braided} monoidal bicategory $\CC$ consists of the following data subject to the following axioms. \\
  DATA:
  \begin{itemize}
    \item A monoidal bicategory $\CC$;
    \item A pseudonatural isomorphism
      \[ b\maps A \tensor B \Rightarrow B \tensor A. \]
    \item Invertible modifications for braiding.
    \begin{center}
      \begin{tikzpicture}
        \node (A) at (  0:2cm) {$B(CA)$};
        \node (B) at ( 60:2cm) {$(BC)A$}
          edge [->] node [l, above right] {$a$} (A);
        \node (C) at (120:2cm) {$A(BC)$}
          edge [->] node [l, above] {$b$} (B);
        \node (D) at (180:2cm) {$(AB)C$}
          edge [->] node [l, above left] {$a$} (C);
        \node (E) at (240:2cm) {$(BA)C$}
          edge [<-] node [l, below left] {$bC$} (D);
        \node (F) at (300:2cm) {$B(AC)$}
          edge [<-] node [l, below] {$a$} (E)
          edge [->] node [l, below right] {$Bb$} (A);
        \node at (-0.25,0) {\tikz\node [rotate=-90] {$\Rightarrow$};};
        \node at (0.25,0) {$R$};
      \end{tikzpicture}
    \end{center}
    \begin{center}
      \begin{tikzpicture}
        \node (A) at (  0:2cm) {$(CA)B$};
        \node (B) at ( 60:2cm) {$C(AB)$}
          edge [->] node [l, above right] {$a^*$} (A);
        \node (C) at (120:2cm) {$(AB)C$}
          edge [->] node [l, above] {$b$} (B);
        \node (D) at (180:2cm) {$A(BC)$}
          edge [->] node [l, above left] {$a^*$} (C);
        \node (E) at (240:2cm) {$A(CB)$}
          edge [<-] node [l, below left] {$Ab$} (D);
        \node (F) at (300:2cm) {$(AC)B$}
          edge [<-] node [l, below] {$a^*$} (E)
          edge [->] node [l, below right] {$bB$} (A);
        \node at (-0.25,0) {\tikz\node [rotate=-90] {$\Rightarrow$};};
        \node at (0.25,0) {$S$};
      \end{tikzpicture}
    \end{center}
  \end{itemize}
  AXIOMS:
  \begin{itemize}
\item This equation governs shuffling one object $A$ and three objects $B, C, D$; for all objects $A, B, C$ and $D$ of $\CC$ the following equation holds:     
    
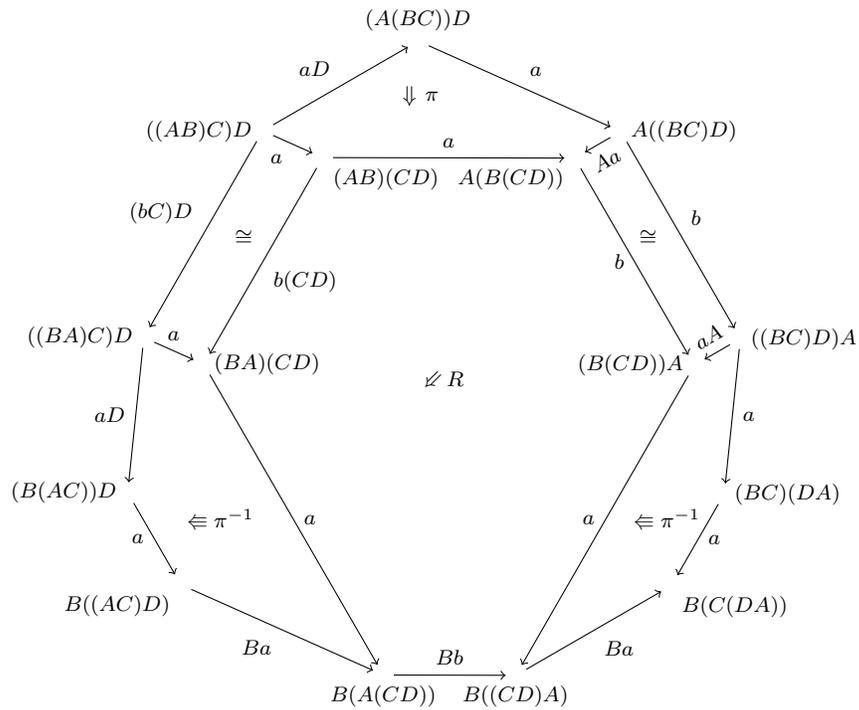
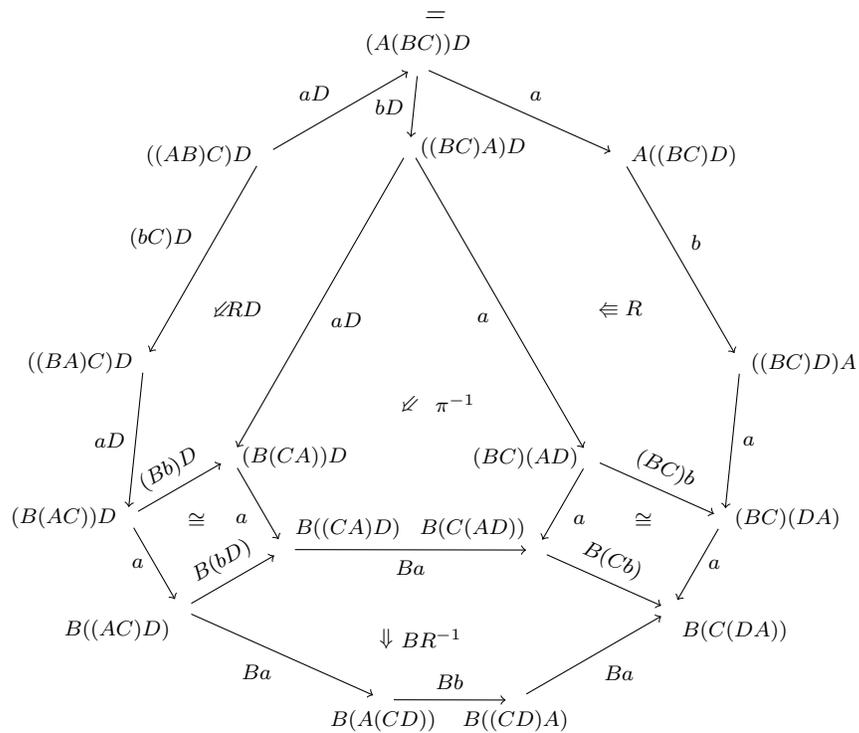
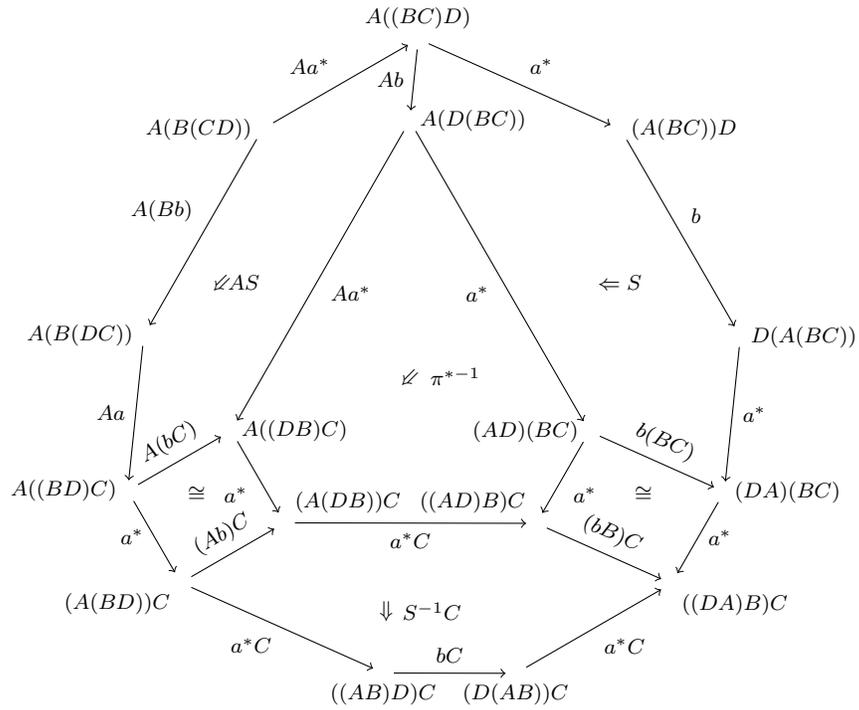
\begin{center}
      \begin{tikzpicture}[line join=round,scale=1.25]
        \begin{scope}[font=\fontsize{8}{8}\selectfont]
          \node (A) at (-.144,3.591) {};
          \node[anchor=south] at (A) {$(A(BC))D$};
          \node (B) at (2.009,2.629) {}
            edge [<-] node [l, above right] {$a$} (A);
          \node[anchor=west] at (B) {$A((BC)D)$};
          \node (C) at (3.281,.425) {}
            edge [<-] node [l, above right] {$b$} (B);
          \node[anchor=west] at (C) {$((BC)D)A$};
          \node (D) at (3.11,-1.226) {}
            edge [<-] node [l, right] {$a$} (C);
          \node[anchor=west] at (D) {$(BC)(DA)$};
          \node (E) at (2.548,-2.203) {}
            edge [<-] node [l, right] {$a$} (D);
          \node[anchor=north west] at (E) {$B(C(DA))$};
          \node (F) at (.894,-3.158) {}
            edge [->] node [l, below right] {$Ba$} (E);
          \node[anchor=north] at (F) {$B((CD)A)$};
          \node (G) at (-.515,-3.155) {}
            edge [->] node [l, above] {$Bb$} (F);
          \node[anchor=north] at (G) {$B(A(CD))$};
          \node (H) at (-2.668,-2.194) {}
            edge [->] node [l, below left] {$Ba$} (G);
          \node[anchor=north east] at (H) {$B((AC)D)$};
          \node (I) at (-3.234,-1.214) {}
            edge [->] node [l, left] {$a$} (H);
          \node[anchor=east] at (I) {$(B(AC))D$};
          \node (J) at (-3.062,.436) {}
            edge [->] node [l, left] {$aD$} (I);
          \node[anchor=east] at (J) {$((BA)C)D$};
          \node (K) at (-1.798,2.636) {}
            edge [->] node [l, above left] {$(bC)D$} (J)
            edge [->] node [l, above left] {$aD$} (A);
          \node[anchor=east] at (K) {$((AB)C)D$};
          \node (L) at (-1.16,2.351) {}
            edge [<-] node [l, below left] {$a$} (K);
          \node[anchor=north west] at (L) {$(AB)(CD)$};
          \node (M) at (1.519,2.346) {}
            edge [<-] node [l, above] {$a$} (L)
            edge [<-] node [l, below, sloped] {$Aa$} (B);
          \node[anchor=north east] at (M) {$A(B(CD))$};
          \node (N) at (2.791,.142) {}
            edge [<-] node [l, left] {$b$} (M)
            edge [->] node [l, left] {$a$} (F)
            edge [<-] node [l, above, sloped] {$aA$} (C);
          \node[anchor=east] at (N) {$(B(CD))A$};
          \node (O) at (-2.425,.152) {}
            edge [<-] node [l, below right] {$b(CD)$} (L)
            edge [<-] node [l, above] {$a$} (J)
            edge [->] node [l, right] {$a$} (G);
          \node[anchor=west] at (O) {$(BA)(CD)$};
          \node at (-0.25,3) {\tikz\node [rotate=-90] {$\Rightarrow$};};
          \node at (0,3) {$\pi$};
          \node at (-2,1.5) {$\cong$};
          \node at (2.3,1.5) {$\cong$};
          \node at (-2.25,-1.5) {$\Lleftarrow \pi^{-1}$};
          \node at (2.5,-1.5) {$\Lleftarrow \pi^{-1}$};
          \node at (-0,0) {\tikz\node [rotate=-135] {$\Rightarrow$};};
          \node at (0.25,0) {$R$};
        \end{scope}
      \end{tikzpicture}
      \\
      =
      \\
      \begin{tikzpicture}[line join=round,scale=1.25]
        \begin{scope}[font=\fontsize{8}{8}\selectfont]
          \node (A) at (-.144,3.591) {};
          \node[anchor=south] at (A) {$(A(BC))D$};
          \node (B) at (2.009,2.629) {}
            edge [<-] node [l, above right] {$a$} (A);
          \node[anchor=west] at (B) {$A((BC)D)$};
          \node (C) at (3.281,.425) {}
            edge [<-] node [l, above right] {$b$} (B);
          \node[anchor=west] at (C) {$((BC)D)A$};
          \node (D) at (3.11,-1.226) {}
            edge [<-] node [l, right] {$a$} (C);
          \node[anchor=west] at (D) {$(BC)(DA)$};
          \node (E) at (2.548,-2.203) {}
            edge [<-] node [l, right] {$a$} (D);
          \node[anchor=north west] at (E) {$B(C(DA))$};
          \node (F) at (.894,-3.158) {}
            edge [->] node [l, below right] {$Ba$} (E);
          \node[anchor=north] at (F) {$B((CD)A)$};
          \node (G) at (-.515,-3.155) {}
            edge [->] node [l, above] {$Bb$} (F);
          \node[anchor=north] at (G) {$B(A(CD))$};
          \node (H) at (-2.668,-2.194) {}
            edge [->] node [l, below left] {$Ba$} (G);
          \node[anchor=north east] at (H) {$B((AC)D)$};
          \node (I) at (-3.234,-1.214) {}
            edge [->] node [l, left] {$a$} (H);
          \node[anchor=east] at (I) {$(B(AC))D$};
          \node (J) at (-3.062,.436) {}
            edge [->] node [l, left] {$aD$} (I);
          \node[anchor=east] at (J) {$((BA)C)D$};
          \node (K) at (-1.798,2.636) {}
            edge [->] node [l, above left] {$(bC)D$} (J)
            edge [->] node [l, above left] {$aD$} (A);
          \node[anchor=east] at (K) {$((AB)C)D$};
          \node (P) at (-2.132,-.578) {}
            edge [<-] node [l, above, sloped] {$(Bb)D$} (I);
          \node[anchor=west] at (P) {$(B(CA))D$};
          \node (Q) at (-.235,2.722) {}
            edge [->] node [l, below right] {$aD$} (P)
            edge [<-] node [l, left] {$bD$} (A);
          \node[anchor=west] at (Q) {$((BC)A)D$};
          \node (R) at (1.675,-.585) {}
            edge [<-] node [l, below left] {$a$} (Q)
            edge [->] node [l, above, sloped] {$(BC)b$} (D);
          \node[anchor=east] at (R) {$(BC)(AD)$};
          \node (S) at (1.112,-1.562) {}
            edge [<-] node [l, below right] {$a$} (R)
            edge [->] node [l, above, sloped] {$B(Cb)$} (E);
          \node[anchor=south east] at (S) {$B(C(AD))$};
          \node (T) at (-1.566,-1.557) {}
            edge [->] node [l, below] {$Ba$} (S)
            edge [<-] node [l, below left] {$a$} (P)
            edge [<-] node [l, above, sloped] {$B(bD)$} (H);
          \node[anchor=south west] at (T) {$B((CA)D)$};
          \node at (-0.25,0) {\tikz\node [rotate=-135] {$\Rightarrow$};};
          \node at (0.25,0) {$\pi^{-1}$};
          \node at (2,1) {$\Lleftarrow R$};
          \node at (-2.25,1) {\tikz\node [rotate=-135] {$\Rightarrow$};};
          \node at (-2,1) {$RD$};
          \node at (2.25,-1.25) {$\cong$};
          \node at (-2.5,-1.25) {$\cong$};
          \node at (-0.5,-2.5) {\tikz\node [rotate=-90] {$\Rightarrow$};};
          \node at (0,-2.5) {$BR^{-1}$};
        \end{scope}
      \end{tikzpicture}
    \end{center}
\newpage
    \item This equation governs shuffling three objects $A,B,C$ and one object $D$; for all objects $A, B, C$ and $D$ of $\CC$ the following equation holds: \\
    \begin{center}
      \begin{tikzpicture}[line join=round,scale=1.25]
        \begin{scope}[font=\fontsize{8}{8}\selectfont]
          \node (A) at (-.144,3.591) {};
          \node[anchor=south] at (A) {$A((BC)D)$};
          \node (B) at (2.009,2.629) {}
            edge [<-] node [l, above right] {$a^*$} (A);
          \node[anchor=west] at (B) {$(A(BC))D$};
          \node (C) at (3.281,.425) {}
            edge [<-] node [l, above right] {$b$} (B);
          \node[anchor=west] at (C) {$D(A(BC))$};
          \node (D) at (3.11,-1.226) {}
            edge [<-] node [l, right] {$a^*$} (C);
          \node[anchor=west] at (D) {$(DA)(BC)$};
          \node (E) at (2.548,-2.203) {}
            edge [<-] node [l, right] {$a^*$} (D);
          \node[anchor=north west] at (E) {$((DA)B)C$};
          \node (F) at (.894,-3.158) {}
            edge [->] node [l, below right] {$a^*C$} (E);
          \node[anchor=north] at (F) {$(D(AB))C$};
          \node (G) at (-.515,-3.155) {}
            edge [->] node [l, above] {$bC$} (F);
          \node[anchor=north] at (G) {$((AB)D)C$};
          \node (H) at (-2.668,-2.194) {}
            edge [->] node [l, below left] {$a^*C$} (G);
          \node[anchor=north east] at (H) {$(A(BD))C$};
          \node (I) at (-3.234,-1.214) {}
            edge [->] node [l, left] {$a^*$} (H);
          \node[anchor=east] at (I) {$A((BD)C)$};
          \node (J) at (-3.062,.436) {}
            edge [->] node [l, left] {$Aa^*$} (I);
          \node[anchor=east] at (J) {$A(B(DC))$};
          \node (K) at (-1.798,2.636) {}
            edge [->] node [l, above left] {$a(Bb)$} (J)
            edge [->] node [l, above left] {$Aa^*$} (A);
          \node[anchor=east] at (K) {$A(B(CD))$};
          \node (L) at (-1.16,2.351) {}
            edge [<-] node [l, below left] {$a^*$} (K);
          \node[anchor=north west] at (L) {$(AB)(CD)$};
          \node (M) at (1.519,2.346) {}
            edge [<-] node [l, above] {$a^*$} (L)
            edge [<-] node [l, below, sloped] {$a^*D$} (B);
          \node[anchor=north east] at (M) {$((AB)C)D$};
          \node (N) at (2.791,.142) {}
            edge [<-] node [l, left] {$b$} (M)
            edge [->] node [l, left] {$a^*$} (F)
            edge [<-] node [l, above, sloped] {$Da^*$} (C);
          \node[anchor=east] at (N) {$D((AB)C)$};
          \node (O) at (-2.425,.152) {}
            edge [<-] node [l, below right] {$(AB)b$} (L)
            edge [<-] node [l, above] {$a^*$} (J)
            edge [->] node [l, right] {$a^*$} (G);
          \node[anchor=west] at (O) {$(AB)(DC)$};
          \node at (-0.25,3) {\tikz\node [rotate=-90] {$\Rightarrow$};};
          \node at (0,3) {$\pi^*$};
          \node at (-2,1.5) {$\cong$};
          \node at (2.3,1.5) {$\cong$};
          \node at (-2.25,-1.5) {$\Leftarrow \pi^{*-1}$};
          \node at (2.5,-1.5) {$\Leftarrow \pi^{*-1}$};
          \node at (-0,0) {\tikz\node [rotate=-135] {$\Rightarrow$};};
          \node at (0.25,0) {$S$};
        \end{scope}
      \end{tikzpicture}
      \\
      =
      \\
      \begin{tikzpicture}[line join=round,scale=1.25]
        \begin{scope}[font=\fontsize{8}{8}\selectfont]
          \node (A) at (-.144,3.591) {};
          \node[anchor=south] at (A) {$A((BC)D)$};
          \node (B) at (2.009,2.629) {}
            edge [<-] node [l, above right] {$a^*$} (A);
          \node[anchor=west] at (B) {$(A(BC))D$};
          \node (C) at (3.281,.425) {}
            edge [<-] node [l, above right] {$b$} (B);
          \node[anchor=west] at (C) {$D(A(BC))$};
          \node (D) at (3.11,-1.226) {}
            edge [<-] node [l, right] {$a^*$} (C);
          \node[anchor=west] at (D) {$(DA)(BC)$};
          \node (E) at (2.548,-2.203) {}
            edge [<-] node [l, right] {$a^*$} (D);
          \node[anchor=north west] at (E) {$((DA)B)C$};
          \node (F) at (.894,-3.158) {}
            edge [->] node [l, below right] {$a^*C$} (E);
          \node[anchor=north] at (F) {$(D(AB))C$};
          \node (G) at (-.515,-3.155) {}
            edge [->] node [l, above] {$bC$} (F);
          \node[anchor=north] at (G) {$((AB)D)C$};
          \node (H) at (-2.668,-2.194) {}
            edge [->] node [l, below left] {$a^*C$} (G);
          \node[anchor=north east] at (H) {$(A(BD))C$};
          \node (I) at (-3.234,-1.214) {}
            edge [->] node [l, left] {$a^*$} (H);
          \node[anchor=east] at (I) {$A((BD)C)$};
          \node (J) at (-3.062,.436) {}
            edge [->] node [l, left] {$Aa$} (I);
          \node[anchor=east] at (J) {$A(B(DC))$};
          \node (K) at (-1.798,2.636) {}
            edge [->] node [l, above left] {$A(Bb)$} (J)
            edge [->] node [l, above left] {$Aa^*$} (A);
          \node[anchor=east] at (K) {$A(B(CD))$};
          \node (P) at (-2.132,-.578) {}
            edge [<-] node [l, above, sloped] {$A(bC)$} (I);
          \node[anchor=west] at (P) {$A((DB)C)$};
          \node (Q) at (-.235,2.722) {}
            edge [->] node [l, below right] {$Aa^*$} (P)
            edge [<-] node [l, left] {$Ab$} (A);
          \node[anchor=west] at (Q) {$A(D(BC))$};
          \node (R) at (1.675,-.585) {}
            edge [<-] node [l, below left] {$a^*$} (Q)
            edge [->] node [l, above, sloped] {$b(BC)$} (D);
          \node[anchor=east] at (R) {$(AD)(BC)$};
          \node (S) at (1.112,-1.562) {}
            edge [<-] node [l, below right] {$a^*$} (R)
            edge [->] node [l, above, sloped] {$(bB)C$} (E);
          \node[anchor=south east] at (S) {$((AD)B)C$};
          \node (T) at (-1.566,-1.557) {}
            edge [->] node [l, below] {$a^*C$} (S)
            edge [<-] node [l, below left] {$a^*$} (P)
            edge [<-] node [l, above, sloped] {$(Ab)C$} (H);
          \node[anchor=south west] at (T) {$(A(DB))C$};
          \node at (-0.25,0) {\tikz\node [rotate=-135] {$\Rightarrow$};};
          \node at (0.25,0) {$\pi^{*-1}$};
          \node at (2,1) {$\Leftarrow S$};
          \node at (-2.25,1) {\tikz\node [rotate=-135] {$\Rightarrow$};};
          \node at (-2,1) {$AS$};
          \node at (2.25,-1.25) {$\cong$};
          \node at (-2.5,-1.25) {$\cong$};
          \node at (-0.5,-2.5) {\tikz\node [rotate=-90] {$\Rightarrow$};};
          \node at (0,-2.5) {$S^{-1}C$};
        \end{scope}
      \end{tikzpicture}
    \end{center}

    \item 
This equation governs shuffling two objects $A,B$ and two other objects $C,D$; For all objects $A, B, C$ and $D$ of $\CC$ the following equation holds: \\
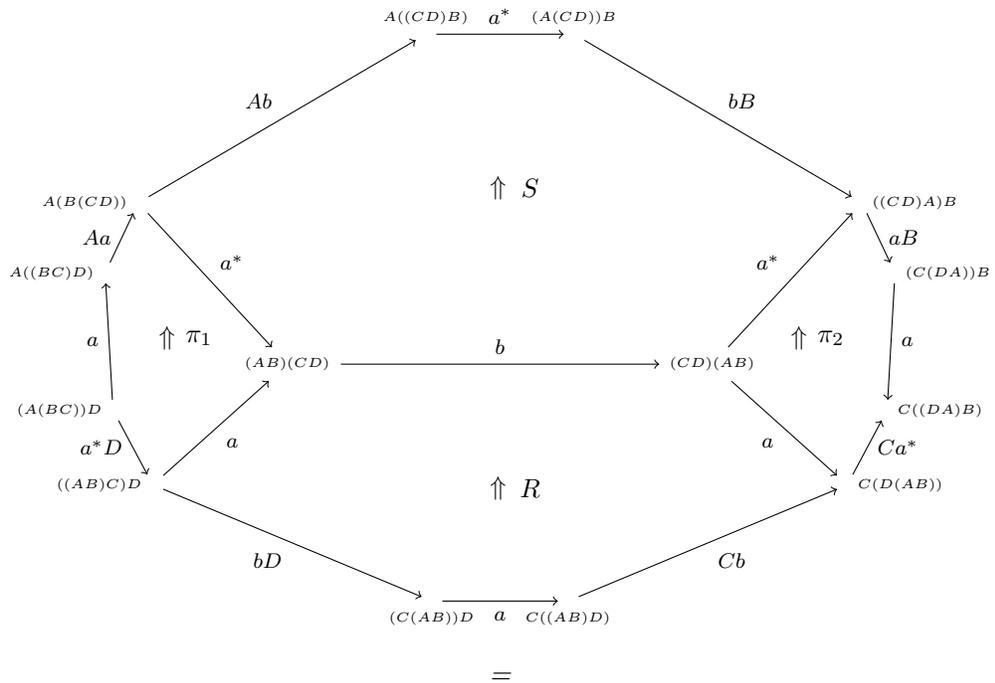
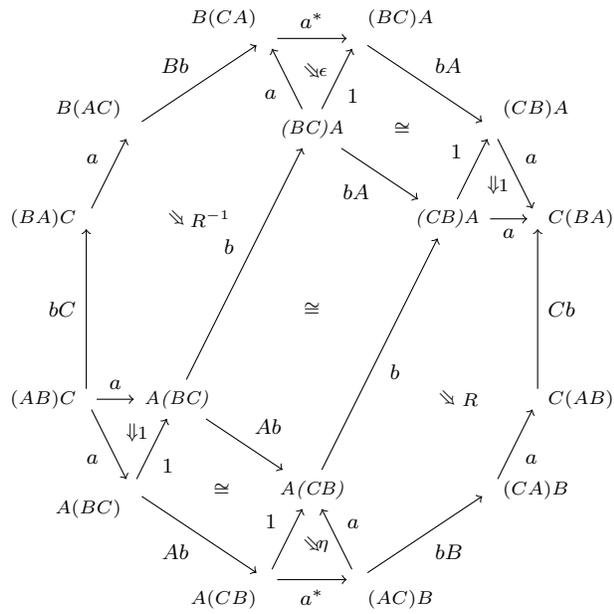
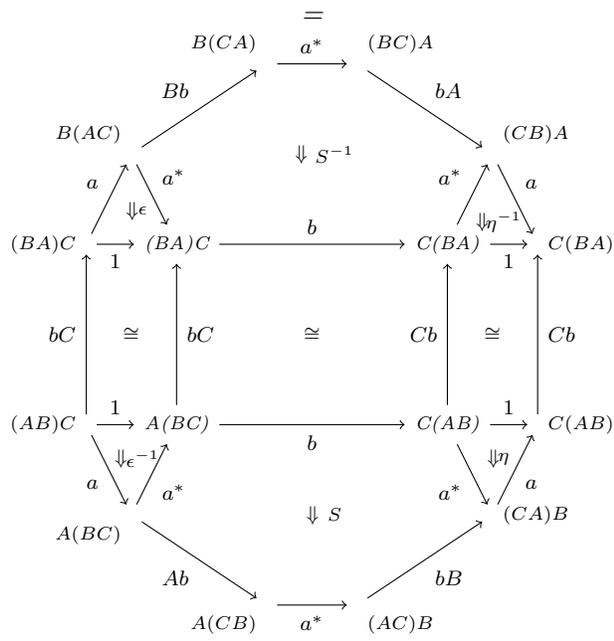
    \begin{center}
      \begin{tikzpicture}[line join=round, scale=2, font=\fontsize{5}{5}\selectfont]
        \node (A) at (-2.573,-.48) {};
        \node [left] at (A) {$(A(BC))D$};

        \node (B) at (-2.625,.433) {}
          edge [<-] node [l, left] {$a$} (A);
        \node [left] at (B) {$A((BC)D)$};

        \node (C) at (-2.405,.899) {}
          edge [<-] node [l, left] {$Aa$} (B);
        \node [left] at (C) {$A(B(CD))$};

        \node (D) at (-.492,2.021) {}
          edge [<-] node [l, above left] {$Ab$} (C);
        \node [above] at (D) {$A((CD)B)$};

        \node (E) at (.492,2.021) {}
          edge [<-] node [l, above] {$a^*$} (D);
        \node [above] at (E) {$(A(CD))B$};

        \node (F) at (2.405,.899) {}
          edge [<-] node [l, above right] {$bB$} (E);
        \node [right] at (F) {$((CD)A)B$};

        \node (G) at (2.625,.433) {}
          edge [<-] node [l, right] {$aB$} (F);
        \node [right] at (G) {$(C(DA))B$};

        \node (H) at (2.573,-.48) {}
          edge [<-] node [l, right] {$a$} (G);
        \node [right] at (H) {$C((DA)B)$};

        \node (I) at (2.308,-.977) {}
          edge [->] node [l, right] {$Ca^*$} (H);
        \node [right] at (I) {$C(D(AB))$};

        \node (J) at (.453,-1.75) {}
          edge [->] node [l, below right] {$Cb$} (I);
        \node [below] at (J) {$C((AB)D)$};

        \node (K) at (-.453,-1.75) {}
          edge [->] node [l, below] {$a$} (J);
        \node [below] at (K) {$(C(AB))D$};

        \node (L) at (-2.308,-.977) {}
          edge [->] node [l, below left] {$bD$} (K)
          edge [<-] node [l, left] {$a^*D$} (A);
        \node [left] at (L) {$((AB)C)D$};

        \node (M) at (-1.412,-.174) {$(AB)(CD)$}
          edge [<-] node [l, below right] {$a$} (L)
          edge [<-] node [l, above right] {$a^*$} (C);

        \node (N) at (1.412,-.174) {$(CD)(AB)$}
          edge [<-] node [l, above] {$b$} (M)
          edge [->] node [l, below left] {$a$} (I)
          edge [->] node [l, above left] {$a^*$} (F);
        \begin{scope}[font=\fontsize{10}{10}\selectfont]
          \node at (-2.2,0) {\tikz\node [rotate=90] {$\Rightarrow$};};
          \node at (-2,0) {$\pi_1$};
          \node at (0,1) {\tikz\node [rotate=90] {$\Rightarrow$};};
          \node at (.2,1) {$S$};
          \node at (0,-1) {\tikz\node [rotate=90] {$\Rightarrow$};};
          \node at (.2,-1) {$R$};
          \node at (2,0) {\tikz\node [rotate=90] {$\Rightarrow$};};
          \node at (2.2,0) {$\pi_2$};
        \end{scope}
      \end{tikzpicture}
      \\ = \\
      \begin{tikzpicture}[line join=round, scale=2, font=\fontsize{5}{5}\selectfont]

        \node (A) at (-2.573,-.48) {};
        \node [left] at (A) {$(A(BC))D$};

        \node (B) at (-2.625,.433) {}
          edge [<-] node [l, left] {$a$} (A);
        \node [left] at (B) {$A((BC)D)$};

        \node (C) at (-2.405,.899) {}
          edge [<-] node [l, left] {$Aa$} (B);
        \node [left] at (C) {$A(B(CD))$};

        \node (D) at (-.492,2.021) {}
          edge [<-] node [l, above left] {$Ab$} (C);
        \node [above] at (D) {$A((CD)B)$};

        \node (E) at (.492,2.021) {}
          edge [<-] node [l, above] {$a^*$} (D);
        \node [above] at (E) {$(A(CD))B$};

        \node (F) at (2.405,.899) {}
          edge [<-] node [l, above right] {$bB$} (E);
        \node [right] at (F) {$((CD)A)B$};

        \node (G) at (2.625,.433) {}
          edge [<-] node [l, right] {$aB$} (F);
        \node [right] at (G) {$(C(DA))B$};

        \node (H) at (2.573,-.48) {}
          edge [<-] node [l, right] {$a$} (G);
        \node [right] at (H) {$C((DA)B)$};

        \node (I) at (2.308,-.977) {}
          edge [->] node [l, right] {$Ca^*$} (H);
        \node [right] at (I) {$C(D(AB))$};

        \node (J) at (.453,-1.75) {}
          edge [->] node [l, below right] {$Cb$} (I);
        \node [below] at (J) {$C((AB)D)$};

        \node (K) at (-.453,-1.75) {}
          edge [->] node [l, below] {$a$} (J);
        \node [below] at (K) {$(C(AB))D$};

        \node (L) at (-2.308,-.977) {}
          edge [->] node [l, below left] {$bD$} (K)
          edge [<-] node [l, left] {$a^*D$} (A);
        \node [left] at (L) {$((AB)C)D$};

        \node (O) at (-2.075,-.32) {}
          edge [<-] node [l, below, sloped] {$(Ab)D$} (A);
        \node [above right] at (O) {$(A(CB))D$};

        \node (P) at (-2.109,.417) {}
          edge [<-] node [l, left] {$a$} (O)
          edge [<-] node [l, above, sloped] {$A(bD)$} (B);
        \node [below right] at (P) {$A((BC)D)$};

        \node (Q) at (-1.72,.642) {}
          edge [<-] node [l, above, sloped] {$Aa$} (P);
        \node [right] at (Q) {$A(C(BD))$};

        \node (R) at (-.944,.116) {$(AC)(BD)$}
          edge [<-] node [l, above right] {$a^*$} (Q);

        \node (S) at (-1.678,-.488) {}
          edge [->] node [l, below right] {$a$} (R)
          edge [<-] node [l, below, sloped] {$a^*D$} (O);
        \node [right] at (S) {$((AC)B)D$};

        \node (T) at (-.546,-1.564) {}
          edge [<-] node [l, above, sloped] {$(bB)D$} (S)
          edge [<-] node [right] {$a^*D$} (K);
        \node at (T) {$((CA)B)D$};
        
        \node (U) at (0,-.807) {$(CA)(BD)$}
          edge [<-] node [l, above left] {$a$} (T)
          edge [<-] node [l, above, sloped] {$b(BD)$} (R);

        \node (V) at (.546,-1.564) {}
          edge [<-] node [l, above right] {$a$} (U)
          edge [->] node [left] {$Ca^*$} (J);
        \node at (V) {$C(A(BD))$};

        \node (W) at (1.678,-.488) {}
          edge [<-] node [l, above, sloped] {$C(Ab)$} (V);
        \node [left] at (W) {$C(A(DB))$};

        \node (X) at (2.075,-.32) {}
          edge [<-] node [l, below, sloped] {$Ca^*$} (W)
          edge [->] node [l, below, sloped] {$C(bB)$} (H);
        \node [above left] at (X) {$C((AD)B)$};

        \node (Y) at (2.109,.417) {}
          edge [->] node [l, right] {$a$} (X)
          edge [->] node [l, above, sloped] {$(Cb)B$} (G);
        \node [below left] at (Y) {$(C(AD))B$};

        \node (Z) at (1.72,.642) {}
          edge [->] node [l, above, sloped] {$aB$} (Y);
        \node [left] at (Z) {$((CA)D)B$};

        \node (A') at (.944,.116) {$(CA)(DB)$}
          edge [->] node [l, above left] {$a^*$} (Z)
          edge [->] node [l, below left] {$a$} (W)
          edge [<-] node [l, above, sloped] {$(CA)b$} (U);

        \node (B') at (0,1.079) {$(AC)(DB)$}
          edge [->] node [l, below, sloped] {$b(DB)$} (A')
          edge [<-] node [l, below, sloped] {$(AC)b$} (R);

        \node (C') at (-.588,1.829) {}
          edge [->] node [l, below left] {$a^*$} (B')
          edge [<-] node [l, below, sloped] {$A(Cb)$} (Q)
          edge [<-] node [right] {$Aa$} (D);
        \node at (C') {$A(C(DB))$};

        \node (D') at (.588,1.829) {}
          edge [<-] node [l, below right] {$a^*$} (B')
          edge [->] node [l, below, sloped] {$(bD)B$} (Z)
          edge [->] node [left] {$aB$} (E);
        \node at (D') {$((AC)D)B$};

        \begin{scope} [font=\fontsize{10}{10}\selectfont]
          \node at (-2.4,0) {$\cong$};
          \node at (-1.7,0) {\tikz\node [rotate=90] {$\Rightarrow$};};
          \node at (-1.5,0) {$\pi_1$};
          \node at (-2,-.8) {\tikz\node [rotate=90] {$\Rightarrow$};};
          \node at (-1.8,-.8) {$SD$};
          \node at (-2,.9) {\tikz\node [rotate=90] {$\Rightarrow$};};
          \node at (-1.75,.9) {$AR^{-1}$};
          \node at (-0.1,1.6) {\tikz\node [rotate=90] {$\Rightarrow$};};
          \node at (0.1,1.6) {$\pi_4$};
          \node at (-.75,.75) {$\cong$};
          \node at (-.75,-.75) {$\cong$};
          \node at (0,0) {$\cong$};
          \node at (.75,-.75) {$\cong$};
          \node at (.75,.75) {$\cong$};
          \node at (-0.1,-1.5) {\tikz\node [rotate=90] {$\Rightarrow$};};
          \node at (0.1,-1.5) {$\pi_3$};
          \node at (1.8,.9) {\tikz\node [rotate=90] {$\Rightarrow$};};
          \node at (2.05,.9) {$R^{-1}B$};
          \node at (1.8,-.8) {\tikz\node [rotate=90] {$\Rightarrow$};};
          \node at (2, -.8) {$CS$};
          \node at (1.5,0) {\tikz\node [rotate=90] {$\Rightarrow$};};
          \node at (1.7,0) {$\pi_2$};
          \node at (2.4,0) {$\cong$};
        \end{scope}
      \end{tikzpicture}
    \end{center}
    \item If the tensor product were associative, the Yang-Baxter equations would hold:
    \begin{center}
      \begin{tikzpicture}
        \begin{scope}[scale=.3, font=\fontsize{8}{8}\selectfont]
          \node (ACB) at (5,0) {ACB};
          \node (ABC) at (0,3) {ABC}
            edge [->] node [l, below left] {$Ab$} (ACB);
          \node (CAB) at (10,3) {CAB}
            edge [<-] node [l, below right] {$b B$} (ACB);
          \node (BAC) at (0,9) {BAC}
            edge [<-] node [l, left] {$b C$} (ABC);
          \node (CBA) at (10,9) {CBA}
            edge [<-] node [l, above left] {$b$} (ACB)
            edge [<-] node [l, right] {$Cb$} (CAB);
          \node (BCA) at (5,12) {BCA}
            edge [<-] node [l, above left] {$Bb$} (BAC)
            edge [<-] node [l, below right] {$b$} (ABC)
            edge [->] node [l, above right] {$b A$} (CBA);

          \node (EQ) at (13,6) {=};
        \end{scope}
      \end{tikzpicture}
      \begin{tikzpicture}
        \begin{scope}[scale=.3, font=\fontsize{8}{8}\selectfont]
          \node (ACB) at (5,0) {ACB};
          \node (ABC) at (0,3) {ABC}
            edge [->] node [l, below left] {$Ab$} (ACB);
          \node (CAB) at (10,3) {CAB}
            edge [<-] node [l, above] {$b$} (ABC)
            edge [<-] node [l, below right] {$b B$} (ACB);
          \node (BAC) at (0,9) {BAC}
            edge [<-] node [l, left] {$b C$} (ABC);
          \node (CBA) at (10,9) {CBA}
            edge [<-] node [l, below] {$b$} (BAC)
            edge [<-] node [l, right] {$Cb$} (CAB);
          \node (BCA) at (5,12) {BCA}
            edge [<-] node [l, above left] {$Bb$} (BAC)
            edge [->] node [l, above right] {$b A$} (CBA);
        \end{scope}
      \end{tikzpicture}
    \end{center}
    
    Again, relaxing the associativity truncates all the corners and some of the edges.  For all objects $A, B$ and $C$ of $\CC$ the following equation holds:\\
    \begin{center}
      \begin{tikzpicture}
        \begin{scope}[scale=.3, font=\fontsize{7}{7}\selectfont]
          \node (ACB1) at (8, 0) {};
          \node [below left] at (ACB1) {$A(CB)$};
          \node (ACB2) at (12, 0) {}
            edge [<-] node [l, below] {$a^*$} (ACB1);
          \node [below right] at (ACB2) {$(AC)B$};
          \node (ACB3) at (10, 4) {A(CB)}
            edge [<-] node [l, above left] {$1$} (ACB1)
            edge [<-] node [l, above right] {$a$} (ACB2);
          \node (ABC1) at (2, 4) {}
            edge [->] node [l, below left] {$Ab$} (ACB1);
          \node [below left] at (ABC1) {$A(BC)$};
          \node (ABC2) at (4, 8) {A(BC)}
            edge [<-] node [l, below right] {$1$} (ABC1)
            edge [->] node [l, above right] {$Ab$} (ACB3);
          \node (ABC3) at (0,8) {}
            edge [->] node [l, above] {$a$} (ABC2)
            edge [->] node [l, below left] {$a$} (ABC1);
          \node [left] at (ABC3) {$(AB)C$};
          \node (CAB1) at (18, 4) {}
            edge [<-] node [l, below right] {$b B$} (ACB2);
          \node [right] at (CAB1) {$(CA)B$};
          \node (CAB2) at (20, 8) {}
            edge [<-] node [l, below right] {$a$} (CAB1);
          \node [right] at (CAB2) {$C(AB)$};
          \node (BAC1) at (0, 16) {}
            edge [<-] node [l, left] {$b C$} (ABC3);
          \node [left] at (BAC1) {$(BA)C$};
          \node (BAC3) at (2, 20) {}
            edge [<-] node [l, above left] {$a$} (BAC1);
          \node [above left] at (BAC3) {$B(AC)$};
          \node (CBA1) at (16, 16) {(CB)A}
            edge [<-] node [l, below right] {$b$} (ACB3);
          \node (CBA2) at (20, 16) {}
            edge [<-] node [l, below] {$a$} (CBA1)
            edge [<-] node [l, right] {$Cb$} (CAB2);
          \node [right] at (CBA2) {$C(BA)$};
          \node (CBA3) at (18, 20) {}
            edge [<-] node [l, above left] {$1$} (CBA1)
            edge [->] node [l, above right] {$a$} (CBA2);
          \node [above right] at (CBA3) {$(CB)A$};
          \node (BCA1) at (10, 20) {(BC)A}
            edge [<-] node [l, above left] {$b$} (ABC2)
            edge [->] node [l, below left] {$b A$} (CBA1);
          \node (BCA2) at (12, 24) {}
            edge [<-] node [l, below right] {$1$} (BCA1)
            edge [->] node [l, above right] {$b A$} (CBA3);
          \node [above right] at (BCA2) {$(BC)A$};
          \node (BCA3) at (8, 24) {}
            edge [<-] node [l, above left] {$Bb$} (BAC3)
            edge [<-] node [l, below left] {$a$} (BCA1)
            edge [->] node [l, above] {$a^*$} (BCA2);
          \node [above left] at (BCA3) {$B(CA)$};
            
          \node (R1A) at (4, 16) {\tikz\node [rotate=-45] {$\Rightarrow$};};
          \node (R1B) at (5.5, 16) {$R^{-1}$};
          \node (R2A) at (16, 8) {\tikz\node [rotate=-45] {$\Rightarrow$};};
          \node (R2B) at (17, 8) {$R$};
          \node (N1) at (14, 20) {$\cong$};
          \node (N2) at (10, 12) {$\cong$};
          \node (N3) at (6, 4) {$\cong$};
          \node (D1) at (18, 17.5) {\tikz\node [rotate=-90] {$\Rightarrow$};};
          \node at (18.5,17.5) {$1$};
          \node (D2) at (2, 6.5) {\tikz\node [rotate=-90] {$\Rightarrow$};};
          \node at (2.5,6.5) {$1$};
          \node (E1A) at (10, 22.5) {\tikz\node [rotate=-45] {$\Rightarrow$};};
          \node (E1B) at (10.5, 22.5) {$\epsilon$};
          \node (E2A) at (10, 1.5) {\tikz\node [rotate=-45] {$\Rightarrow$};};
          \node (E2B) at (10.5, 1.5) {$\eta$};
        \end{scope}
      \end{tikzpicture}
      \\ = \\
      \begin{tikzpicture}
        \begin{scope}[scale=.3, font=\fontsize{7}{7}\selectfont]

          \node (ACB1) at (8, 0) {};
          \node [below left] at (ACB1) {$A(CB)$};
          \node (ACB2) at (12, 0) {}
            edge [<-] node [l, below] {$a^*$} (ACB1);
          \node [below right] at (ACB2) {$(AC)B$};
          \node (ABC1) at (2, 4) {}
            edge [->] node [l, below left] {$Ab$} (ACB1);
          \node [below left] at (ABC1) {$A(BC)$};
          \node (ABC2) at (4, 8) {A(BC)}
            edge [<-] node [l, below right] {$a^*$} (ABC1);
          \node (ABC3) at (0,8) {}
            edge [->] node [l, above] {$1$} (ABC2)
            edge [->] node [l, below left] {$a$} (ABC1);
          \node [left] at (ABC3) {$(AB)C$};
          \node (CAB1) at (18, 4) {}
            edge [<-] node [l, below right] {$b B$} (ACB2);
          \node [right] at (CAB1) {$(CA)B$};
          \node (CAB2) at (20, 8) {}
            edge [<-] node [l, below right] {$a$} (CAB1);
          \node [right] at (CAB2) {$C(AB)$};
          \node (CAB3) at (16, 8) {C(AB)}
            edge [<-] node [l, below] {$b$} (ABC2)
            edge [->] node [l, above] {$1$} (CAB2)
            edge [->] node [l, below left] {$a^*$} (CAB1);
          \node (BAC1) at (0, 16) {}
            edge [<-] node [l, left] {$b C$} (ABC3);
          \node [left] at (BAC1) {$(BA)C$};
          \node (BAC2) at (4, 16) {(BA)C}
            edge [<-] node [l, below] {$1$} (BAC1)
            edge [<-] node [l, right] {$b C$} (ABC2);
          \node (BAC3) at (2, 20) {}
            edge [<-] node [l, above left] {$a$} (BAC1)
            edge [->] node [l, above right] {$a^*$} (BAC2);
          \node [above left] at (BAC3) {$B(AC)$};
          \node (CBA1) at (16, 16) {C(BA)}
            edge [<-] node [l, above] {$b$} (BAC2)
            edge [<-] node [l, left] {$C b$} (CAB3);
          \node (CBA2) at (20, 16) {}
            edge [<-] node [l, below] {$1$} (CBA1)
            edge [<-] node [l, right] {$C b$} (CAB2);
          \node [right] at (CBA2) {$C(BA)$};
          \node (CBA3) at (18, 20) {}
            edge [<-] node [l, above left] {$a^*$} (CBA1)
            edge [->] node [l, above right] {$a$} (CBA2);
          \node [above right] at (CBA3) {$(CB)A$};
          \node (BCA2) at (12, 24) {}
            edge [->] node [l, above right] {$b A$} (CBA3);
          \node [above right] at (BCA2) {$(BC)A$};
          \node (BCA3) at (8, 24) {}
            edge [<-] node [l, above left] {$Bb$} (BAC3)
            edge [->] node [l, above] {$a^*$} (BCA2);
          \node [above left] at (BCA3) {$B(CA)$};

          \node (S1) at (9.5, 20) {\tikz\node [rotate=-90] {$\Rightarrow$};};
          \node at (11,20) {$S^{-1}$};
          \node (S2) at (10, 4) {\tikz\node [rotate=-90] {$\Rightarrow$};};
          \node at (11,4) {$S$};
          \node (N1) at (2, 12) {$\cong$};
          \node (N1) at (10, 12) {$\cong$};
          \node (N1) at (18, 12) {$\cong$};
          \node (E1) at (2, 17.5) {\tikz\node [rotate=-90] {$\Rightarrow$};};
          \node at (2.5,17.5) {$\epsilon$};
          \node (E2) at (1.5, 6.5) {\tikz\node [rotate=-90] {$\Rightarrow$};};
          \node at (2.5,6.5) {$\epsilon^{-1}$};
          \node (E3) at (17.5, 17) {\tikz\node [rotate=-90] {$\Rightarrow$};};
          \node at (18.5,17) {$\eta^{-1}$};
          \node (E4) at (18, 6.5) {\tikz\node [rotate=-90] {$\Rightarrow$};};
          \node at (18.5,6.5) {$\eta$};
        \end{scope}
      \end{tikzpicture}
    \end{center}
  \end{itemize}
\end{defn}


\newpage
\begin{thebibliography}{10}

\bibitem{Baez:2008}
J.\ Baez, Groupoidification.  Available at\hfill\break
\href{http://math.ucr.edu/home/baez/groupoidification/}{http://math.ucr.edu/home/baez/groupoidification/}.

\bibitem{BaezDolan:1998}
J.\ Baez and J.\ Dolan, Categorification, in {\sl Higher Category
Theory,} eds.\ Ezra Getzler and Mikhail Kapranov,
Contemp.\ Math.\ 230,  American Mathematical Society, 
Providence, Rhode Island, 1998, pp.\ 1--36.

\bibitem{BaezDolan:2001}
J.\ Baez and J.\ Dolan, From finite sets to Feynman diagrams, in {\em
Mathematics Unlimited---2001 and Beyond}, eds.\ Bj\"orn Engquist 
and Wilfried Schmid, Springer, Berlin, 2001, pp.\ 29--50.  Also available as
\href{http://arxiv.org/abs/math/0004133}{arXiv:math/0004133}.

\bibitem{BaezHoffnungWalker:2009}
J.\ Baez, A.\ Hoffnung, and C.\ Walker, Groupoidification Made Easy. Available at
\href{http://arxiv.org/abs/0812.4864}{\texttt{arXiv:0812.4864v1}}.

\bibitem{BaezHoffnungWalker:2009HDA7}
J.\ Baez, A.\ Hoffnung, and C.\ Walker, Higher Dimensional Algebra VII: Groupoidification. \textit{Theory Appl. Categ.} vol. 24, 2010, No. 18, 489--553. Available at \href{http://www.tac.mta.ca/tac/volumes/24/18/24-18abs.html}{http://www.tac.mta.ca/tac/volumes/24/18/24-18abs.html}

\bibitem{HDA8}
J.\ Baez and A.\ Hoffnung, Higher Dimensional Algebra VIII: The Hecke Bicategory. \textit{preprint}.

\bibitem{ChuangRouquier}
J.\ Chuang and R.\ Rouquier, Derived Equivalences for Symmetric Groups and $\mathfrak{sl}_2$ Categorification. \textit{Ann. Math.} vol. 167, 2008, 245--298.

\bibitem{CraneFrenkel}
L.\ Crane and I.\ Frenkel, Four Dimensional Quantum Field Theory, Hopf Categories, and the Canonical Basis. \textit{Jour. Math. Phys.} vol. 35, 1994, 5136--5154. 
\bibitem{drinfeld}
V.\ Drinfeld, Quantum Groups. \textit{Proc. Int. Cong. Math. (Berkeley, 1986)}, 798--820. 

\bibitem{Fregier:H2A}
Y.\ Fregier and F.\ Wagemann, On Hopf $2$-algebras. Available at \href{http://arxiv.org/abs/0908.2353}{\texttt{arXiv:0908.2353}}.

\bibitem{Gabriel}
P.\ Gabriel, Unzerlebare Darstellungen I, \textit{Manuscr.\ Math.\ }
{\bf 6} (1972), 71--103.

\bibitem{GPS}
R.\ Gordon, A.\ Powers, and R.\ Street, Coherence for Tricategories, \textit{Memoirs Amer. Math. Soc.} 117, 1995. No. 558.  

\bibitem{Hoffnung:tricat}
A.\ Hoffnung, Spans in 2-Categories: A One Object Tetracategory. Available at \href{http://mysite.science.uottawa.ca/hoffnung/papers_files/spans.pdf}{http://mysite.science.uottawa.ca/hoffnung/papers$\underline{\emph{ } }$files/spans.pdf}

\bibitem{Hubery} A.\ W.\ Hubery, Ringel--Hall algebras, lecture
notes available at 
\href{http://www.maths.leeds.ac.uk/~ahubery/RHAlgs.html}
{$\langle$http:/\!/www.maths.leeds.ac.uk/~ahubery/RHAlgs.html$\rangle$}.

\bibitem{JoyalStreet}
A.\ Joyal and R.\ Street, Braided tensor categories, \textsl{Adv.\ Math.} {\bf 102} (1993), 20--78.

\bibitem{Kapranov}
M.\ Kapranov, Eisenstein Series and Quantum Affine Algebras, \textsl{J.\ Math.\ Sci. (New York)} {\bf 84} (1997), 1311--1360.

\bibitem{Kashiwara1}
M.\ Kashiwara, Crystalizing the $q$-Analogue of Universal Enveloping Algebras. \textsl{Comm. Math. Phys.} {\bf 133} (1990), 249--260.

\bibitem{Kashiwara2}
M.\ Kashiwara, On Crystal Bases of the $q$-Analogue of Universal Enveloping Algebras. \textsl{Duke Math. Jour.} {\bf 63} (1991), 465--516.

\bibitem{Kashiwara3}
M.\ Kashiwara, Global Crystal Bases of Quantum Groups. \textsl{Duke Math. Jour.} {\bf 69} (1993), 455--485.

\bibitem{KL1}
M.\ Khovanov and A.\ Lauda, A diagrammatic approach to categorification of quantum groups I. \textsl{Represent. Theory} {\bf 13} (2009), 309--347.

\bibitem{KL2}
M.\ Khovanov and A.\ Lauda, A diagrammatic approach to categorification of quantum groups II.

\bibitem{KL3}
M.\ Khovanov and A.\ Lauda, A diagrammatic approach to categorification of quantum groups III.

\bibitem{Khovanov}
M.\ Khovanov, Hopfological algebra and categorification at a root of unity: the first steps. Available as \href{http://arxiv.org/abs/0509083}{arXiv:math/0509083}.

\bibitem{Lauda1:2008}
A.\ Lauda, A categorification of quantum $\mathfrak{sl}(2)$.  Available
at\hfill\break
\href{http://arxiv.org/abs/0803.3652}{\texttt{arXiv:0803.3652v3}}.
	
\bibitem{Lauda2:2008}
A.\ Lauda, Categorified quantum $\mathfrak{sl}(2)$ and equivariant
cohomology of iterated flag varieties. Available at
\href{http://arxiv.org/abs/0803.3848}{\texttt arXiv:0803.3848}.

\bibitem{Lusztig1}
G.\ Lusztig, Canonical Bases Arising from Quantized Universal Enveloping Algebras. \textit{Jour. Amer. Math. Soc.} {\bf 3} (1990), 447--498.

\bibitem{Lusztig2}
G.\ Lusztig, Canonical Bases Arising from Quantized Universal Enveloping Algebras II. \textit{Progr. Theoret. Phys. Suppl.} {\bf 4} (1990), 175--201.

\bibitem{Lusztig3}
G.\ Lusztig, Quivers, Perverse Sheaves, and Quantized Enveloping Algebras. \textit{Jour. Amer. Math. Soc.} {\bf 4} (1991), 365--421.

\bibitem{Mac Lane} S.\ Mac Lane, {\sl Categories for the 
Working Mathematician}, Springer, Berlin, 1998.

\bibitem{Majid} S.\ Majid, Braided Groups, {\sl J.\ P.\ App.\ Alg.} {\bf 86} (1993), 187--221.

\bibitem{Majid2} S.\ Majid, Algebras and Hopf algebras in braided categories, in \textsl{Advances in
Hopf Algebras (Chicago, IL, 1992)}, Lecture Notes in Pure and Appl.
Math. \textbf{158}, Dekker, New York, 1994, pp. 55--105.

\bibitem{Morton:2006}
J.\ Morton, Categorified algebra and quantum mechanics, {\em Theory
Appl. Categ.}{\bf 16}, (2006), 785-854.  Also available as
\href{http://arxiv.org/abs/math/0601458}{arXiv:math/0601458}.

\bibitem{Neuchl}
M.\ Neuchl, Representation Theory of Hopf Categories, Available at \href{http://math.ucr.edu/home/baez/neuchl.ps}{http://math.ucr.edu/home/baez/neuchl.ps}

\bibitem{Pfeiffer}
H.\ Pfeiffer, 2-Groups, Trialgebras and Their Hopf Categories of Representations, \textsl{Adv. Math.} {\bf 202} No. 1 (2004), 62--108.

\bibitem{Ringel} C.\ Ringel, Hall algebras and quantum groups, 
\textsl{Invent.\ Math.\ } {\bf 101} (1990), 583--591.

\bibitem{Ringel2} C.\ Ringel, Hall algebras revisited,
\textsl{Israel Math.\ Conf.\ Proc.\ }{\bf 7} (1993), 171--176. 
Also available at \hfill \break
\href{http://www.mathematik.uni-bielefeld.de/~ringel/opus/hall-rev.pdf}
{$\langle$http:/\!/www.mathematik.uni-bielefeld.de/$\sim$ringel/opus/hall-rev.pdf$\rangle$}.

\bibitem{RingelGreen} C.\ Ringel, Green's theorem on Hall algebras, in \textsl{Representation Theory of
Algebras and Related Topics (Mexico City, 1994)}, CMS Conf.\ Proc.\ {\bf 19}, Amer. Math. Soc., Providence, RI, 1996, pp. 185--225.

\bibitem{Rouquier}
R.\ Rouquier, 2-Kac-Moody Algebras. Available at \href{http://arxiv.org/abs/math/0812.5023}{arXiv:math/0812.5023}.

\bibitem{Schiffman} 
O.\ Schiffman, Lectures on Hall algebras, available as
\href{http://arxiv.org/abs/math/0611617}{arXiv:math/0611617}.

\bibitem{stay_def}
M.\ Stay, Compact Closed Bicategories, \textit{preprint}.

\bibitem{Varagnolo}
M.\ Varagnolo and E.\ Vasserot, Canonical Basis and Khovanov-Lauda Algebras. Available as \href{http://arxiv.org/abs/math/0901.3992}{arXiv:math/0901.3992}.

\bibitem{HopfObject}
C.\ Walker, Hall Algebras as Hopf Objects, preprint. available at 
\href{http://math.ucr.edu/\~{ }cwalker66/HopfObject}{http://math.ucr.edu/\~{ }cwalker66/HopfObject}

\end{thebibliography}
\end{document}